\documentclass[reqno,a4paper]{amsart}
\pdfoutput=1
\usepackage{fullpage}
\usepackage{subcaption}
\makeatletter
\renewcommand\p@subfigure{\thefigure.}
\makeatother

\usepackage[]{enumitem}
\setlist{itemsep=1ex,topsep=1ex} 
\setlist[1]{leftmargin=*}
\setlist[itemize,1]{labelindent=1em}
\setlist[itemize,2]{leftmargin=2pc,labelsep=*}

\usepackage[]{hyperref} 

\usepackage[nobysame,alphabetic,initials,msc-links]{amsrefs}

\DefineSimpleKey{bib}{how}
\DefineSimpleKey{bib}{mrclass}
\DefineSimpleKey{bib}{mrnumber}
\DefineSimpleKey{bib}{fjournal}
\DefineSimpleKey{bib}{mrreviewer}

\renewcommand{\PrintDOI}[1]{%
  \href{http://dx.doi.org/#1}{{\tt DOI:#1}}%
}
\renewcommand{\eprint}[1]{#1}
\BibSpec{book}{%
    +{}  {\PrintPrimary}                {transition}
    +{.} { \PrintDate}                  {date}
    +{.} { \textit}                     {title}
    +{.} { }                            {part}
    +{:} { \textit}                     {subtitle}
    +{,} { \PrintEdition}               {edition}
    +{}  { \PrintEditorsB}              {editor}
    +{,} { \PrintTranslatorsC}          {translator}
    +{,} { \PrintContributions}         {contribution}
    +{,} { }                            {series}
    +{,} { \voltext}                    {volume}
    +{,} { }                            {publisher}
    +{,} { }                            {organization}
    +{,} { }                            {address}
    +{,} { }                            {status}
    +{,} { \PrintDOI}                   {doi}
    +{,} { \PrintISBNs}                 {isbn}
    +{}  { \parenthesize}               {language}
    +{}  { \PrintTranslation}           {translation}
    +{;} { \PrintReprint}               {reprint}
    +{.} { }                            {note}
    +{.} {}                             {transition}
    +{}  {\SentenceSpace \PrintReviews} {review}
}
\BibSpec{article}{%
    +{}  {\PrintAuthors}                {author}
    +{,} { \textit}                     {title}
    +{.} { }                            {part}
    +{:} { \textit}                     {subtitle}
    +{,} { \PrintContributions}         {contribution}
    +{.} { \PrintPartials}              {partial}
    +{,} { }                            {journal}
    +{}  { \textbf}                     {volume}
    +{}  { \PrintDatePV}                {date}
    +{,} { \issuetext}                  {number}
    +{,} { \eprintpages}                {pages}
    +{,} { }                            {status}
    +{,} { \PrintDOI}                   {doi}
    +{,} { \eprint}        {eprint}
    +{}  { \parenthesize}               {language}
    +{}  { \PrintTranslation}           {translation}
    +{;} { \PrintReprint}               {reprint}
    +{.} { }                            {note}
    +{.} {}                             {transition}
    +{}  {\SentenceSpace \PrintReviews} {review}
}
\BibSpec{collection.article}{%
    +{}  {\PrintAuthors}                {author}
    +{,} { \textit}                     {title}
    +{.} { }                            {part}
    +{:} { \textit}                     {subtitle}
    +{,} { \PrintContributions}         {contribution}
    +{,} { \PrintConference}            {conference}
    +{}  {\PrintBook}                   {book}
    +{,} { }                            {booktitle}
    +{,} { \PrintDateB}                 {date}
    +{,} { pp.~}                        {pages}
    +{,} { }                            {publisher}
    +{,} { }                            {organization}
    +{,} { }                            {address}
    +{,} { }                            {status}
    +{,} { \PrintDOI}                   {doi}
    +{,} { \eprint}        {eprint}
    +{}  { \parenthesize}               {language}
    +{}  { \PrintTranslation}           {translation}
    +{;} { \PrintReprint}               {reprint}
    +{.} { }                            {note}
    +{.} {}                             {transition}
    +{}  {\SentenceSpace \PrintReviews} {review}
}
\BibSpec{misc}{%
  +{}{\PrintAuthors}  {author}
  +{,}{ \textit}      {title}
  +{.}{ }             {how}
  +{}{ \parenthesize} {date}
  +{,} { available at \eprint}        {eprint}
  +{,}{ available at \url}{url}
  +{,}{ }             {note}
  +{.}{}              {transition}
}

\usepackage{amssymb}
\usepackage{mathrsfs}
\usepackage{bbm}

\usepackage{graphics}
\input pgfexternal.tex
\usepackage{tikz}
\usetikzlibrary{knots,cd,backgrounds}
\usetikzlibrary{math}

\tikzset{
  x=0.7cm, 
  y=0.7cm,
  knot diagram/every knot diagram/.append style={
    consider self intersections,
    clip width=5,
    end tolerance=0.05pt,
    clip radius=7pt,
  },
  knot diagram/every strand/.append style={
    thick,
  }
}
\newcommand{\stwbr}[2][]{
\strand[#1] (1, #2+2) .. controls +(0,-0.3) and +(0,0.3) .. (0,#2+1)
.. controls +(0,-0.3) and +(0,0.3) .. (1,#2); }
\newcommand{\ctwbr}[3][]{
\strand[evaluate={\str=(sqrt(#2));},#1]
(#2,#3+#2+#2) .. controls +(0,-0.3*\str) and +(0,0.3*\str) ..
(0,#3+#2) .. controls +(0,-0.3*\str) and +(0,0.3*\str) .. (#2,#3); }
\newcommand{\sribtwbr}[2][]{
\strand[#1] (1,#2+2) .. controls +(0,-0.2) and +(-0.17,0) ..
(1.13,#2+2-0.57) ..  controls +(0.11,0) and +(0,-0.06) .. (1.28,
#2+2-0.4) .. controls +(0,0.1) and +(0.15,0) .. (1.05,#2+2-0.2) ..
controls +(-0.15,0) and +(0,0.1) .. (0.75, #2+2-0.4) .. controls
+(0,-0.25) and +(0,0.25) .. (0,#2+2-1.1) .. controls +(0,-0.3) and
+(0,0.3) .. (1,#2); }
\newcommand{\cribtwbr}[3][]{
\strand[evaluate={\str=(sqrt(#2));\dummy=#2-0.25;\wstr=(sqrt(\dummy));},#1]
(#2,#3+#2+#2) .. controls +(0,-0.2) and +(-0.17,0) ..
(#2+0.13,#3++#2+#2-0.57) ..  controls +(0.11,0) and +(0,-0.06) ..
(#2+0.28, #3+#2+#2-0.4) .. controls +(0,0.1) and +(0.15,0) ..
(#2+0.05,#3+#2+#2-0.2) .. controls +(-0.15,0) and +(0,0.1) ..
(#2-0.25, #3+#2+#2-0.4) .. controls +(0,-0.3*\wstr) and
+(0,0.3*\wstr) .. (0,#3+#2-0.1) .. controls +(0,-0.3*\str) and
+(0,0.3*\str) .. (#2,#3); }
\newcommand{\varsribtwbr}[2][]{
\strand[#1] (1,#2) .. controls +(0,0.2) and +(-0.17,0) ..
(1.13,#2+0.57) ..  controls +(0.11,0) and +(0,0.06) .. (1.28,
#2+0.4) .. controls +(0,-0.1) and +(0.15,0) .. (1.05,#2+0.2) ..
controls +(-0.15,0) and +(0,-0.1) .. (0.75, #2+0.4) .. controls
+(0,0.25) and +(0,-0.25) .. (0,#2+1.1) .. controls +(0,0.3) and
+(0,-0.3) .. (1,#2+2); }

\newcommand{\braidgen}[3][]{
\strand[#1] (#2,#3) .. controls +(0,0.3) and +(0,-0.3) ..
(#2+1,#3+1); \strand[#1] (#2+1,#3) .. controls +(0,0.3) and
+(0,-0.3) .. (#2,#3+1); }
\newcommand{\braidtwocldbl}[1][]{
  \def\StyleOpt{{#1}}%
  \braidtwocldblintnl%
}
\newcommand{\braidtwocldblintnl}[3][5pt]{
\strand[evaluate={\optstyle=\StyleOpt;\ddist=#1;},only when
rendering/.style={
    double distance=\ddist,
  },\optstyle] (#2,#3) .. controls +(0,0.3) and +(0,-0.3) .. (#2+1,#3+1);
\strand[evaluate={\optstyle=\StyleOpt;},\optstyle] (#2+1,#3) ..
controls +(0,0.3) and +(0,-0.3) .. (#2,#3+1); }
\newcommand{\braidtencldbl}[1][]{
  \def\StyleOpt{{#1}}%
  \braidtencldblintnl%
}
\newcommand{\braidtencldblintnl}[3][5pt]{
\strand[evaluate={\optstyle=\StyleOpt;},\optstyle] (#2,#3) ..
controls +(0,0.3) and +(0,-0.3) .. (#2+1,#3+1);
\strand[evaluate={\optstyle=\StyleOpt;\ddist=#1;},only when
rendering/.style={
    double distance=\ddist,
  },\optstyle] (#2+1,#3) .. controls +(0,0.3) and +(0,-0.3) .. (#2,#3+1);
}

\numberwithin{equation}{section}

\theoremstyle{plain}
\newtheorem{Theorem}{Theorem}[section]
\newtheorem{Lem}[Theorem]{Lemma}
\newtheorem{Prop}[Theorem]{Proposition}
\newtheorem{Cor}[Theorem]{Corollary}
\newtheorem{Conj}[Theorem]{Conjecture}

\theoremstyle{definition}
\newtheorem{Def}[Theorem]{Definition}

\theoremstyle{remark} 
\newtheorem{Rem}[Theorem]{Remark}
\newtheorem{Rems}[Theorem]{Remarks}

\newcommand\bp{\begin{proof}}
\newcommand\ep{\end{proof}}

\mathchardef\mhyph="2D

\DeclareMathOperator{\ad}{ad}
\DeclareMathOperator{\Ad}{Ad}

\DeclareMathOperator{\End}{End}

\DeclareMathOperator{\id}{id}
\DeclareMathOperator{\Id}{Id}

\DeclareMathOperator{\Mor}{Mor}

\DeclareMathOperator{\Rep}{Rep}
\DeclareMathOperator{\Sym}{Sym}
\DeclareMathOperator{\Tr}{Tr}

\DeclareMathOperator{\ev}{ev}

\newcommand{\hotimes}{\mathbin{\hat{\otimes}}}

\newcommand{\op}{\mathrm{op}}

\newcommand{\ns}{\mathrm{ns}}
\newcommand{\un}{\mathbbm{1}}

\newcommand{\msA}{\mathscr{A}}

\newcommand{\msE}{\mathscr{E}}
\newcommand{\wmsE}{\widetilde{\msE}}

\newcommand{\msK}{\mathscr{K}}
\newcommand{\wmsK}{\widetilde{\mathscr{K}}}

\newcommand{\msR}{\mathscr{R}}
\newcommand{\wmsR}{\widetilde{\msR}}

\newcommand{\msU}{\mathscr{U}}

\newcommand{\mfa}{\mathfrak{a}}

\newcommand{\mfb}{\mathfrak{b}}

\newcommand{\mfe}{\mathfrak{e}}
\newcommand{\mfg}{\mathfrak{g}}

\newcommand{\mfh}{\mathfrak{h}}

\newcommand{\mfk}{\mathfrak{k}}

\newcommand{\mfm}{\mathfrak{m}}
\newcommand{\mfn}{\mathfrak{n}}

\newcommand{\mfp}{\mathfrak{p}}

\newcommand{\mfs}{\mathfrak{s}}
\newcommand{\mfsl}{\mathfrak{sl}}
\newcommand{\mfso}{\mathfrak{so}}
\newcommand{\mfsp}{\mathfrak{sp}}
\newcommand{\mfsu}{\mathfrak{su}}
\newcommand{\mft}{\mathfrak{t}}
\newcommand{\mfu}{\mathfrak{u}}
\newcommand{\mfureg}{\mfu_{\mathrm{r}}}

\newcommand{\mcC}{\mathcal{C}}
\newcommand{\mcD}{\mathcal{D}}

\newcommand{\mcI}{\mathcal{I}}

\newcommand{\mcS}{\mathcal{S}}
\newcommand{\mcT}{\mathcal{T}}

\newcommand{\mbc}{\mathbf{c}}
\newcommand{\mbq}{\mathbf{q}}
\newcommand{\mbs}{\mathbf{s}}
\newcommand{\mbt}{\mathbf{t}}

\newcommand{\C}{\mathbb{C}}

\newcommand{\N}{\mathbb{N}}
\newcommand{\R}{\mathbb{R}}
\newcommand{\T}{\mathbb{T}}

\newcommand{\Z}{\mathbb{Z}}

\newcommand{\walpha}{\widetilde{\alpha}}

\newcommand{\Mod}[1]{#1\mhyph\mathrm{Mod}}

\newcommand{\absv}[1]{\left|#1\right|}

\newcommand{\qbin}[3]{\genfrac{[}{]}{0pt}{}{#1}{#2}_{#3}}

\newcommand{\Nu}{N}

\begin{document}

\title{Ribbon braided module categories, quantum symmetric pairs and Knizhnik--Zamolodchikov equations}

\date{February 13, 2018}

\author{Kenny De Commer}
\address{Vrije Universiteit Brussel}
\email{kenny.de.commer@vub.be}

\author{Sergey Neshveyev}
\address{Universitetet i Oslo}
\email{sergeyn@math.uio.no}

\author{Lars Tuset}
\address{OsloMet - storbyuniversitetet}
\email{lars.tuset@hioa.no}

\author{Makoto Yamashita}
\address{Ochanomizu University}
\email{yamashita.makoto@ocha.ac.jp}

\thanks{The work of K.~De Commer was partially supported by the FWO grant G.0251.15N and the grant H2020-MSCA-RISE-2015-691246-QUANTUM DYNAMICS. The work of S.~Neshveyev was partially supported by the ERC grant 307663-NCGQG}


\begin{abstract}
Let $\mfu$ be a compact semisimple Lie algebra, and $\sigma$ be a Lie algebra involution of $\mfu$. Let $\Rep_q(\mfu)$ be the ribbon braided tensor C$^*$-category of admissible $U_q(\mfu)$-representations for $0<q<1$. We introduce three module C$^*$-categories over $\Rep_q(\mfu)$ starting from the input data $(\mfu,\sigma)$. The first construction is based on the theory of $2$- cyclotomic KZ-equations. The second construction uses the notion of quantum symmetric pair as developed by G.~Letzter. The third construction uses a variation of Drinfeld twisting. In all three cases the module C$^*$-category is ribbon twist-braided in the sense of A.~Brochier---this is essentially due to B.~Enriquez in the first case, is proved by S.~Kolb in the second case, and is closely related to work of J.~Donin, P.~Kulish, and A.~Mudrov in the third case. We formulate a conjecture concerning equivalence of these ribbon twist-braided module C$^*$-categories, and confirm it in the rank one case.
\end{abstract}

\maketitle

\section*{Introduction}

The \emph{Tannaka--Kre{\u\i}n principle} for quantum groups asserts that there is a duality between Hopf algebras on the one hand, and \emph{concretely represented} tensor categories with duals on the other. Here, we say that a tensor category $\mcC$ is concretely represented if there is a fiber functor realizing $\mcC$ within the category of vector spaces. From this viewpoint, a tensor category is more fundamental than a Hopf algebra, as the latter can be seen as a concrete manifestation of the former through a fiber functor, while an abstract tensor category can occur in more general contexts, for instance, as a hidden symmetry in quantum field theories.

This viewpoint is beautifully illustrated by the result of V.~Drinfeld~\cite{MR1047964} which in the framework of braided tensor categories provides two apparently unrelated constructions starting from a semisimple Lie algebra $\mfg$. The first construction is algebraic, where the category is the representation category $\Rep_q(\mfg)$ of $U_q(\mfg)$, the \emph{$q$-deformed universal enveloping algebra} of $\mfg$, viewed as a quasi-triangular Hopf algebra. The second construction is analytic, and builds a new associator on the category of $\mfg$-representations by means of the monodromy of the \emph{Knizhnik--Zamolodchikov equations} arising from conformal field theory. In both approaches, the crucial point is that one should look for a system which contains a \emph{braiding} obtained by deforming the symmetric braiding appearing in the classical Tannaka--Kre{\u\i}n duality theorem.

If quantum groups correspond to tensor categories, their actions correspond to \emph{module categories}. Such structures are more elaborate, partly reflected through new phenomena such as the distinction between `homogeneous spaces' and `subgroups' already present in the case $\mfg = \mfsl_2(\C)$, see~\cite{MR919322}. While a full classification seems difficult, there is a very useful ansatz producing module categories from the theory of \emph{reflection equations}, which stems from the study of quantum scattering on the half-line by I.~Cherednik~\cite{MR774205}.

While there are many interesting developments in the relation between the reflection equation and quantum group actions (\citelist{\cite{MR1198638}\cite{arXiv:math/991141}\cite{MR1958831}\cite{MR1992884}\cite{MR2304470}\cite{MR2565052}} to list a few), the main motivation for the present work is the interplay between the theory of KZ-equation of type B on the analytic side, and quantum symmetric pairs for semisimple Lie algebras on the algebraic side.

Let $\mfg$ be a complex semisimple Lie algebra, $\theta$ an involution of $\mfg$ with spectral decomposition $\mfg = \mfg^\theta \oplus \mfp$, where $\mfg^{\theta}$ is the fixed point Lie algebra. Choose a Cartan subalgebra $\mfh$ of $\mfg$ which is stable under $\theta$ and such that $\mfh \cap \mfp$ is a maximal $\ad_\mfg$-diagonalizable abelian subspace of $\mfp$. G.~Letzter~\citelist{\cite{MR1717368}\cite{MR1742961}} realizes a quantization of $\mfg^{\theta}$ as a coideal subalgebra of $U_q(\mfg)$ built using a root system of $(\mfg, \mfh)$, depending possibly on some extra parameters.  The categorical counterpart is given by the representation category of this coideal subalgebra as a module category over $\Rep_q(\mfg)$. Recent works of M.~Balagovi\'c and S.~Kolb \citelist{\cite{arXiv:1507.06276}\cite{arXiv:1705.04238}}, making use of the formalism of the \emph{universal $K$-matrix}~\cite{MR1992884}, show that this module category is a \emph{ribbon twist-braided module category} in the sense of A.~Brochier \cite{MR3248737}. Note that the theory of (universal) $K$-matrices forms the concrete connection with the theory of reflection equations. 

On the other hand, the theory of \emph{2-cyclotomic} KZ-equations (or, KZ-equations of \emph{type B})~\citelist{\cite{MR1052280}\cite{MR1289327}\cite{MR1940926}\cite{MR2383601}} makes it possible to construct,  for an arbitrary involution $\sigma$ of $\mfg$, an action of $\Rep(\mfg)$ with the KZ-associator on $\Rep(\mfg^{\sigma})$ and a non-trivial \emph{mixed} associator. Based on the results of \cite{MR2383601} and \cite{MR2892463}, one observes that this is actually again a ribbon twist-braided module category.

\smallskip
The goal of this paper is to perform the above constructions in the presence of a C$^*$-structure, corresponding to the compact form $\mfu$ of $\mfg$, and to state a precise conjecture concerning the equivalence of the resulting ribbon twist-braided module C$^*$-categories in case $\sigma = \theta$. We also present a third picture, conjecturally equivalent to the above two, where we consider involutions $\nu$ of $\mfu$ for which the prescribed Cartan subalgebra $\mft\subseteq \mfu$ is stable and $\mft^{\nu}$ is a Cartan subalgebra of $\mfu^{\nu}$. In this situation, the involution quantizes straightforwardly to an automorphism of $U_q(\mfu)$. Using a $\nu$-twisted version of the Drinfeld double construction, we construct a $*$-algebra with coaction by $U_q(\mfu)$ such that its category of admissible $*$-representations becomes a ribbon twist-braided module C$^*$-category over $\Rep_q(\mfu)$. Here the construction of the braiding on the module C$^*$-category is a variation on the construction of a universal $K$-matrix in \cite{MR1992884}.

\smallskip
The precise structure of the paper is as follows.

In Section~\ref{sec:rib-tw-mod} we recall the formalism of ribbon twist-braided module categories \citelist{\cite{MR3248737}\cite{arXiv:1705.04238}} in the setting of non-trivial associators and non-trivial autoequivalence of the acting tensor category. In Section~\ref{sec:symm-pair-lie-alg-symm-sp} we present an overview of the main elements of the theory of symmetric pairs for compact semisimple Lie algebras that we will make use of. In Section~\ref{ref:three-const-ribb-br-mod} we present in detail the three constructions of twist-braided module C$^*$-categories associated to compact symmetric pairs that were mentioned above. We then state our main conjecture regarding the equivalence of these ribbon twist-braided module C$^*$-categories in Section \ref{sec:main-conj}. We verify the conjecture in Section~\ref{sec:rk-one-case} for the rank one case.

In the Appendix we treat in more detail the following topics. In the first part we discuss the two extremal positions for a stable Cartan subalgebra with respect to an involution on a semisimple compact Lie algebra. In the second part we investigate compatibility of the $*$-structure of $U_q(\mfu)$ with the coideal symmetric pairs of \cite{MR1717368}, complementing the discussion in \cite{MR1913438} concerning twisted $*$-invariance. In the third part we provide a link between the choice of parameters for a coideal symmetric pair, and the character theory of the associated non-parameter coideal symmetric pair. In the fourth part we discuss
a relation between ribbon braids and cylinder twists.

\smallskip
\emph{Acknowledgements}: We thank Weiqiang Wang for bringing our attention to~\cite{arXiv:1610.09271}.

\subsection*{Conventions}

We assume in the following, unless stated otherwise, that all categories are $\C$-linear and semisimple---in particular morphism spaces are finite dimensional. We also assume that all (bi)functors are $\C$-(bi)linear. We identify identity morphisms with their objects, $\id_X = X$. For $\sigma$ an endomorphism on a real vector space $V$, we will denote by the same symbol its extension to a complex-linear endomorphism of $V^{\C} = V \otimes_\R \C$.

We use the following notation for $q$-integers, $q$-factorials, and $q$-binomial coefficients:
\begin{align*}
[n]_q &= \frac{q^{-n} - q^n}{q^{-1} - q},&
[n]_q! &= [n]_q[n-1]_q \cdots [2]_q[1]_q,&
\qbin{m}{n}{q} &= \frac{[m]_q!}{[n]_q![m-n]_q!}.
\end{align*}

\section{Ribbon twist-braided module \texorpdfstring{C$^*$}{C*}-categories}
\label{sec:rib-tw-mod}

We mainly follow the conventions of~\cite{MR3204665} for tensor C$^*$-categories. Additional basic material on module categories can be found in, e.g., \citelist{\cite{MR2681261}\cite{MR3242743}}.

\begin{Def}
A \emph{tensor category} is given by:
\begin{itemize}
\item a category $\mcC$,
\item a bifunctor $\otimes\colon \mcC \times \mcC \to \mcC$,
\item a distinguished object $\un$ of $\mcC$,
\item a natural isomorphism $\Phi_{U, V, W} \colon (U \otimes V) \otimes W \to U \otimes (V \otimes W)$, and
\item natural isomorphisms $\lambda_U \colon \un \otimes U \to U$, $\rho_U \colon U \otimes \un \to U$
\end{itemize}
satisfying a standard set of axioms. We often use the symbol $\mcC$ to denote the above set of data. A \emph{tensor C$^*$-category} is a tensor category whose underlying category is a C$^*$-category and for which $\otimes$ is a $*$-bifunctor and $\Phi$, $\lambda$, and $\rho$  are \emph{unitary} natural transformations.
\end{Def}

\begin{Def}
Let $\mcC$ be a tensor category as above. A \emph{(right) module category} over $\mcC$  is given by:
\begin{itemize}
\item a category $\mcD$,
\item a bifunctor $\odot \colon \mcD  \times \mcC \rightarrow \mcD$, and
\item natural isomorphisms $r_X\colon X \odot \un \to X$ and
\[
\Psi_{X, U, V}\colon (X \odot U) \odot V \rightarrow X \odot (U \otimes V)
\]
\end{itemize}
for which the following diagrams commute (unit relations and the \emph{mixed pentagon relation}):

\begin{figure}[h]

\begin{tikzcd}[column sep=tiny]
(X\odot \un)\odot U \ar[rr, "r_X\odot U"] \ar[rd, "\Psi_{X,\un,U}"'] & & X\odot U  \\
& X\odot (\un \odot U) \ar[ru, "X\odot \lambda_U"']
\end{tikzcd}
\qquad
\begin{tikzcd}[column sep=tiny]
(X\odot U)\odot \un \ar[rr, "r_{X\odot U}"] \ar[rd, "\Psi_{X,U,\un}"'] & & X\odot U  \\
& X\odot (U \odot \un) \ar[ru, "X\odot \rho_U"']
\end{tikzcd}
\begin{tikzpicture}[commutative diagrams/every diagram]
\begin{scope}[yscale=0.7,xscale=1.2,evaluate={\radius=3;}]
\path (0,0) (0,\radius+1); 
  \node (P0) at (0,\radius) {$((X\odot U)\odot V) \odot W $};
  \node (P1) at ({\radius*cos(90+72)-1},{\radius*sin(90+72)}) {$(X\odot (U\otimes V))\odot W$};
  \node (P2) at ({\radius*cos(90+2*72)-1.3},{\radius*sin(90+2*72)}) {$X\odot ((U\otimes V)\otimes W)$};
  \node (P3) at ({\radius*cos(90+3*72)+1.3},{\radius*sin(90+3*72)}) {$X\odot ((U\otimes V)\otimes W)$};
  \node (P4) at ({\radius*cos(90+4*72)+1},{\radius*sin(90+4*72)}) {$(X\odot U) \odot (V\otimes W)$};
  \path[commutative diagrams/.cd, every arrow, every label]
    (P0) edge node[swap] {$\Psi_{X,U,V}\odot W$} (P1)
    (P1) edge node[swap,pos=0.3] {$\Psi_{X,U\otimes V,W}$} (P2)
    (P2) edge node[swap] {$X\odot \Phi_{U,V,W}$} (P3)
    (P4) edge node[pos=0.3] {$\Psi_{X,U,V\otimes W}$} (P3)
    (P0) edge node {$\Psi_{X\odot U,V,W}$} (P4);
\end{scope}
\end{tikzpicture}
\end{figure}
Again, we often denote the whole structure by $\mcD$.

When $\mcC$ is a tensor C$^*$-category, we call $\mcD$ a module C$^*$-category if $\mcD$ is a C$^*$-category, $\odot$ is a $*$-bifunctor, and $\Psi,r$ are unitary.
\end{Def}

In the following we will assume that the unit $\un$ of $\mcC$ is strict, and that it acts as a strict unit on $\mcD$, so $X \odot \un = X$ and $r_X= X$. This will be satisfied in all our examples.

\begin{Def}
A \emph{braided} tensor category is a tensor category $\mcC$ endowed with natural isomorphisms $\beta_{U, V}\colon U \otimes V \to V \otimes U$ satisfying the \emph{hexagon equation}
\[
\begin{tikzcd}[column sep=tiny]
& (V\otimes U)\otimes W  \ar[rrrrr,"\Phi_{V,U,W}"] &&&&& V\otimes (U\otimes W) \ar[dr,"V\otimes \beta_{U,W}"]&\\
(U\otimes V)\otimes W \ar[ur,"\beta_{U,V}\otimes W"] \ar[dr,"\Phi_{U,V,W}"']  &&&&&&&  V\otimes (W\otimes U)\\
& U\otimes (V\otimes W) \ar[rrrrr,"\beta_{U,V\otimes W}"']  &&&&& (V\otimes W)\otimes U \ar[ru,"\Phi_{V,W,U}"']&
\end{tikzcd}
\]
and the similar relation with $\beta_{U,V}$ replaced by $\beta_{V,U}^{-1}$. We also assume $\beta_{\un,U} = \beta_{U,\un} = U$.

A \emph{monoidal autoequivalence} of $\mcC$ consists of an equivalence $\sigma \colon \mcC \to \mcC$ and natural isomorphisms
\[
(\sigma_2)_{U, V} \colon \sigma(U) \otimes \sigma(V) \to \sigma(U \otimes V)
\]
satisfying natural compatibility conditions with respect to the associator $\Phi$. A \emph{braided autoequivalence} of $\mcC$ is a monoidal autoequivalence $\sigma$ for which the following diagram commutes:
\[
\begin{tikzcd}[column sep=large]
\sigma(U)\otimes \sigma(V) \ar[r,"\beta_{\sigma(U),\sigma(V)}"] \ar[d,"(\sigma_2)_{U,V}"']& \sigma(V)\otimes \sigma(U) \ar[d,"(\sigma_2)_{V,U}"]\\ \sigma(U\otimes V) \ar[r,"\sigma(\beta_{U,V})"] & \sigma(V\otimes U)
\end{tikzcd}
\]
\end{Def}

The following definition was introduced in \cite{MR3248737}*{Section 5}, see also \cite{arXiv:1705.04238}*{Remark 3.15}. However, we use a slightly different convention than the latter reference in the presence of a non-trivial braided autoequivalence (apart from introducing explicitly the necessary associators in the non-strict setting).

\begin{Def}
Let $\mcC$ be a braided tensor category, and $\sigma\colon\mcC\to\mcC$ a braided monoidal autoequivalence. A right module category $\mcD$ over $\mcC$ is said to be $\sigma$-\emph{braided} when $\mcD$ is equipped with a natural isomorphism $\eta_{X, U}\colon X \odot \sigma(U) \to X \odot U$, called the \emph{$\sigma$-braid} of $\mcD$, which satisfies the $\sigma$-\emph{octagon equation}
\begin{equation}\label{EqOct2}
\begin{tikzcd}[column sep=large,cramped] 
(X \odot U) \odot \sigma(V) \ar[d,"\eta_{X\odot U,V}"'] \ar[r,"\Psi_{X,U,\sigma(V)}"] & X\odot (U\otimes \sigma(V)) \ar[r,"X\odot \beta_{U,\sigma(V)}"] & X\odot (\sigma(V)\otimes U)  \ar[r,"\Psi_{X,\sigma(V),U}^{-1}"] & (X\odot \sigma(V))\odot U \ar[d,"\eta_{X,V}\odot U"] \\
(X\odot U)\odot V &  \ar[l,"\Psi_{X,U,V}^{-1}"] X\odot (U\otimes V) & \ar[l,"X\odot \beta_{V,U}"]X\odot (V\otimes U) & \ar[l,"\Psi_{X,V,U}"] (X\odot V)\odot U
\end{tikzcd}
\end{equation}
We say that $\eta$ is a \emph{ribbon $\sigma$-braid} if, moreover, the \emph{ribbon $\sigma$-twist} equation is satisfied:
\begin{equation}\label{EqRB2}
\begin{tikzcd}[column sep=5em]
X\odot \sigma(U\otimes V) \ar[r,"X\odot (\sigma_2)_{U,V}^{-1}"] \ar[d,"\eta_{X,U\otimes V}"']& X\odot (\sigma(U)\otimes \sigma(V)) \ar[r,"\Psi_{X,\sigma(U),\sigma(V)}^{-1}"] & (X\odot \sigma(U))\odot \sigma(V) \ar[d,"\eta_{X,U}\odot \sigma(V)"]\\
X\odot (U\otimes V) & \ar[l,"\Psi_{X,U,V}"] (X\odot U)\odot V & \ar[l,"\eta_{X\odot U,V}"] (X\odot U)\odot \sigma(V)
\end{tikzcd}
\end{equation}
We say that $\mcD$ is \emph{(ribbon) twist-braided} if~$\sigma$ is implicit, and that $\mcD$ is \emph{(ribbon) braided} if $\sigma$ is the identity.
\end{Def}

\begin{Rems}
\begin{enumerate}
\item If all our structures including $\sigma$ are strict, our definition will give a (ribbon) $\sigma^{-1}$-braid in the sense of \cite{arXiv:1705.04238}*{Remark 3.15}, taking $e_{M,V} = \eta_{M,\sigma^{-1}(V)}$.
\item In the C$^*$-setting, besides the condition of module C$^*$-category, we just assume that $\sigma$ is a $*$-autoequivalence and $\mcD$ is a $\mcC$-module C$^*$-category. We \emph{do not} assume that $\beta$ and $\eta$ are unitary, as this would be too restrictive.
\end{enumerate}
\end{Rems}

The following definition clarifies the notion of equivalence we will use between different (ribbon) twist-braided module C$^*$-categories.

\begin{Def}
Let $\mcC$ be a braided tensor category with braided monoidal autoequivalences $\sigma$ and $\sigma'$, and let $(\mcD,\odot,\Psi,\eta)$, resp.~$(\mcD',\odot',\Psi',\eta')$ be a $\sigma$-braided, resp.~$\sigma'$-braided module category over $\mcC$. We say that $\mcD$ and $\mcD'$ are \emph{equivalent} as twist-braided module categories, written $\mcD \cong \mcD'$, if there exist an isomorphism of monoidal functors $F\colon \sigma' \rightarrow \sigma$, and an equivalence $G\colon \mcD \rightarrow \mcD'$ of categories together with natural isomorphisms
$$
(G_2)_{X, U}\colon G(X) \odot' U \rightarrow G(X \odot U)
$$
such that the diagrams
$$
\begin{tikzpicture}[commutative diagrams/every diagram]
\begin{scope}[yscale=0.7,xscale=1.2,evaluate={\radius=3;}]
  \node (P0) at (0,\radius) {$(G(X) \odot' U)\odot' V$};
  \node (P1) at ({\radius*cos(90+72)},{\radius*sin(90+72)})         {$G(X\odot U)\odot'V$};
  \node (P2) at ({\radius*cos(90+2*72)-0.7},{\radius*sin(90+2*72)+0.1}) {$G((X\odot U)\odot V)$};
  \node (P3) at ({\radius*cos(90+3*72)+0.7},{\radius*sin(90+3*72)+0.1}) {$G(X\odot (U\otimes V))$};
  \node (P4) at ({\radius*cos(90+4*72)},{\radius*sin(90+4*72)})     {$G(X) \odot' (U\otimes V)$};
  \path[commutative diagrams/.cd, every arrow, every label]
    (P0) edge node[swap] {$(G_2)_{X,U}\odot' V$} (P1)
    (P1) edge node[swap,pos=0.3] {$(G_2)_{X\odot U,V}$} (P2)
    (P2) edge node[swap] {$G(\Psi_{X,U,V})$} (P3)
    (P4) edge node[pos=0.3] {$(G_2)_{X,U\otimes V}$} (P3)
    (P0) edge node {$\Psi'_{G(X),U,V}$} (P4);
\end{scope}
\begin{scope}[shift={(11.5,0)},yscale=0.7,xscale=1.2,evaluate={\radius=3;}]
  \node (P0) at (0,\radius) {$G(X) \odot' \sigma'(U)$};
  \node (P1) at ({\radius*cos(90+72)},{\radius*sin(90+72)})         {$G(X)\odot' \sigma(U)$};
  \node (P2) at ({\radius*cos(90+2*72)-0.7},{\radius*sin(90+2*72)+0.1}) {$G(X\odot \sigma(U))$};
  \node (P3) at ({\radius*cos(90+3*72)+0.7},{\radius*sin(90+3*72)+0.1}) {$G(X\odot U)$};
  \node (P4) at ({\radius*cos(90+4*72)},{\radius*sin(90+4*72)})     {$G(X) \odot' U$};
  \path[commutative diagrams/.cd, every arrow, every label]
    (P0) edge node[swap] {$G(X) \odot' F_U$} (P1)
    (P1) edge node[swap,pos=0.3] {$(G_2)_{X,\sigma(U)}$} (P2)
    (P2) edge node[swap] {$G(\eta_{X,U})$} (P3)
    (P4) edge node[pos=0.3] {$(G_2)_{X,U}$} (P3)
    (P0) edge node {$\eta'_{X,U}$} (P4);
\end{scope}
\end{tikzpicture}
$$
commute. In the C$^*$-setting we ask that $G$ is a $*$-functor, and that $F$ and $G_2$ are unitary.
\end{Def}
Note that $F$ is uniquely determined if it exists.

For completeness, let us briefly recall the graphical presentation of (ribbon) twist-braids, following~\cite{MR3248737}. We represent composition of morphisms by vertical stacking, $\xi \eta$ being represented by the part for $\xi$ on top of that for $\eta$. The monoidal product is represented by horizontal juxtaposition. As usual, we represent the braiding of $\mcC$ by the tangle \ref{fig:br}. A twist-braid can be represented by the diagram \ref{fig:tw-br}, so that the twist octagon relation can be presented as \ref{fig:tw-oct-rel} (suppressing the module category associator $\Psi$).
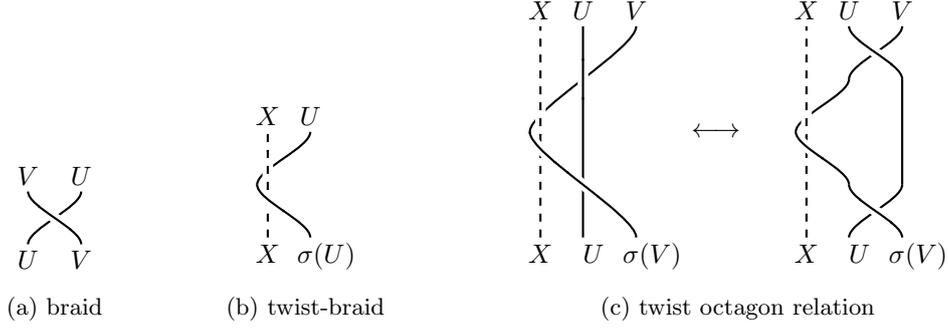
\begin{figure}[h]
\begin{subfigure}[b]{.2\linewidth}
\centering
\leavevmode
\begin{tikzpicture}
\begin{knot}[
]
\braidgen{0}{0}
\node at (0,-0.3) {$U$};
\node at (1,-0.3) {$V$};
\node at (1,1.3) {$U$};
\node at (0,1.3) {$V$};
\flipcrossings{1}
\end{knot}
\end{tikzpicture}
\caption{braid}
\label{fig:br}
\end{subfigure}
\begin{subfigure}[b]{.2\linewidth}
\centering
\leavevmode
\begin{tikzpicture}
\begin{knot}[
]
\stwbr{0}
\strand[only when rendering/.style={dashed}] (0.2,0) -- (0.2,2);
\node at (0.2,-0.3) {$X$};
\node at (1.3,-0.35) {$\sigma(U)$};
\node at (0.2,2.3) {$X$};
\node at (1,2.3) {$U$};
\flipcrossings{1}
\end{knot}
\end{tikzpicture}
\caption{twist-braid}
\label{fig:tw-br}
\end{subfigure}
\begin{subfigure}[b]{.5\linewidth}
\centering
\leavevmode
\begin{tikzpicture}
\begin{knot}[
]
\ctwbr{2}{0}
\strand[only when rendering/.style={dashed}] (0.2,0) -- (0.2,4);
\strand[] (1,0) -- (1,4);
\node at (0.2,-0.3) {$X$};
\node at (1.2, -0.3) {$U$};
\node at (2.3,-0.35) {$\sigma(V)$};
\node at (0.2,4.3) {$X$};
\node at (1,4.3) {$U$};
\node at (2,4.3) {$V$};
\flipcrossings{1,2}
\end{knot}
\draw node at (3.5,2) {$\longleftrightarrow$};
\begin{scope}[shift={(5,0)}]
\begin{knot}[
]
\braidgen{1}{0}
\stwbr{1}
\braidgen{1}{3}
\strand[] (2,1) -- (2,3);
\strand[only when rendering/.style={dashed}] (0.2,0) -- (0.2,4);
\node at (0.2,-0.3) {$X$};
\node at (1.2, -0.3) {$U$};
\node at (2.3,-0.35) {$\sigma(V)$};
\node at (0.2,4.3) {$X$};
\node at (1,4.3) {$U$};
\node at (2,4.3) {$V$};
\flipcrossings{1,2,4}
\end{knot}
\end{scope}
\end{tikzpicture}
\caption{twist octagon relation}
\label{fig:tw-oct-rel}
\end{subfigure}
\caption{twist-braids}
\end{figure}

On the other hand, ribbon twist-braids need to be represented by \ref{fig:rib-tw-br}, so that regular isotopy of strands in \ref{fig:rib-tw-eq} represents the ribbon twist equation.\\
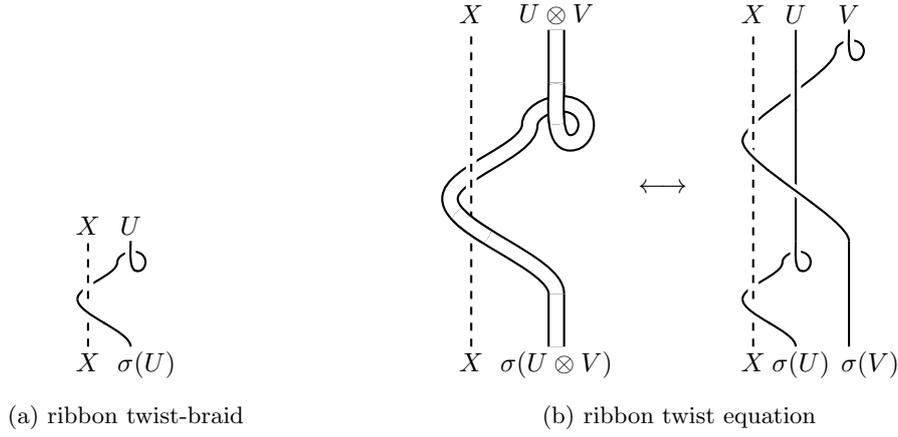
\begin{figure}[h]
\begin{subfigure}[b]{.35\linewidth}
\centering
\leavevmode
\begin{tikzpicture}
\begin{knot}[
]
\sribtwbr{0}
\strand[only when rendering/.style={dashed}] (0.2,0) -- (0.2,2);
\node at (0.2,-0.3) {$X$};
\node at (0.2,2.3) {$X$};
\node at (1.3,-0.35) {$\sigma(U)$};
\node at (1,2.3) {$U$};
\flipcrossings{2}
\end{knot}
\end{tikzpicture}
\caption{ribbon twist-braid}
\label{fig:rib-tw-br}
\end{subfigure}
\begin{subfigure}[b]{.55\linewidth}
\centering
\leavevmode
\begin{tikzpicture}
\begin{scope}[scale=2]
\begin{knot}[
]
\sribtwbr[only when rendering/.style={
    double distance=5pt,
  }]{0.5}
\strand[only when rendering/.style={
    double distance=5pt,
  }] (1,0) -- (1,0.5);
\strand[only when rendering/.style={
    double distance=5pt,
  }] (1,2.5) -- (1,3);
\strand[only when rendering/.style={dashed}] (0.2,0) -- (0.2,3);
\flipcrossings{2}
\end{knot}
\end{scope}
\node at (0.4,-0.3) {$X$};
\node at (0.4,6.3) {$X$};
\node at (2,-0.35) {$\sigma(U \otimes V)$};
\node at (2,6.3) {$U \otimes V$};
\node at (4,3) {$\longleftrightarrow$};
\begin{scope}[shift={(5.5,0)}]
\begin{knot}[
]
\sribtwbr{0} \strand[] (2,0) -- (2,2);
\cribtwbr{2}{2}  \strand[] (1,2) -- (1,6);
\strand[only when rendering/.style={dashed}] (0.2,0) -- (0.2,6);
\node at (0.2,-0.3) {$X$};
\node at (0.2,6.3) {$X$};
\node at (1.1,-0.35) {$\sigma(U)$};
\node at (1,6.3) {$U$};
\node at (2.4,-0.35) {$\sigma(V)$};
\node at (2,6.3) {$V$};
\flipcrossings{2,5,6}
\end{knot}
\end{scope}
\end{tikzpicture}
\caption{ribbon twist equation}
\label{fig:rib-tw-eq}
\end{subfigure}
\caption{ribbon twist-braids}
\end{figure}

\begin{Rem}
The extra twist can be gotten rid of by slightly changing the axiomatics of the ribbon twist-braids. This change of axiomatics does not affect anything as long as the acting tensor category has a ribbon twist, see \cite{arXiv:1606.04769}*{Remark~3.6}.
\end{Rem}

One can obtain module categories from algebraic structures in the following way. We momentarily neglect the semisimplicity condition on our categories.

\begin{Def}[\citelist{\cite{MR1047964}\cite{MR1321145}*{Chapter XV}}]
A (unital and counital) \emph{quasi-bialgebra} is given by
\begin{itemize}
\item a unital algebra $A$,
\item unital algebra homomorphisms $\Delta \colon A \to A \otimes A$ and $\varepsilon \colon A \to \C$, and
\item an invertible element $\phi \in A \otimes A \otimes A$
\end{itemize}
satisfying
\begin{gather*}
\begin{aligned}
(\id \otimes \Delta) \Delta (a) &= \phi ((\Delta \otimes \id) \Delta(a)) \phi^{-1},&
(\id \otimes \id \otimes \Delta)(\phi) (\Delta \otimes \id \otimes \id)(\phi) &= \phi_{234} (\id \otimes \Delta \otimes \id)(\phi) \phi_{123},
\end{aligned}\\
(\varepsilon \otimes \id) \Delta(a) = a = (\id \otimes \varepsilon) \Delta (a),\quad (\id \otimes \varepsilon \otimes \id)(\phi) = 1 \otimes 1.
\end{gather*}
\end{Def}

From the above conditions we also have $(\id \otimes \id \otimes \varepsilon)(\phi) = 1 = (\varepsilon \otimes \id \otimes \id)(\phi)$. The category $\Mod{A}$ of left $A$-modules becomes a tensor category as follows. When $U, V, W$ are $A$-modules, $U \otimes V$ is the usual tensor product over $\C$ with the left $A$-module structure $a\cdot (u\otimes v) = \Delta(a)(u\otimes v)$, and the associator is given by
\[
\Phi_{U,V,W}(u\otimes v\otimes w) = \phi (u\otimes v\otimes w).
\]
In this context the rigidity of $\Mod{A}$ can be encoded by a generalization of antipode, and quasi-bialgebras with such antipodes are called \emph{quasi-Hopf algebras}.

\begin{Def}
Let $A$ be a quasi-bialgebra, and $B$ a unital algebra. A \emph{(normalized right) quasi-coaction} of $A$ on $B$ is given by a unital homomorphism $\alpha\colon B \rightarrow B \otimes A$ and an invertible element $\psi \in B \otimes A \otimes A$ satisfying
\begin{gather*}
(\id\otimes \varepsilon\otimes \id)(\psi) = (\id\otimes \id\otimes \varepsilon)(\psi) = 1\otimes 1,\\
(1\otimes \phi)(\id\otimes \Delta\otimes \id)(\psi)(\psi\otimes 1) = (\id\otimes \id\otimes \Delta)(\psi)(\alpha\otimes \id\otimes \id)(\psi),\\
(\id\otimes \Delta)\alpha(b) = \psi (\alpha\otimes \id)\alpha(b) \psi^{-1}, \quad (\id\otimes \varepsilon)\alpha(b) = b.
\end{gather*}
\end{Def}

The following lemma is immediate.

\begin{Lem}
In the above setting, $\Mod{B}$ becomes an $(\Mod{A})$-module category as follows. When $X$ is a $B$-module and $U, V$ are $A$-modules, $X \otimes U$ is a $B$-module by $b\cdot (x\otimes u) = \alpha(b)(x\otimes u)$, and the associator is given by
\[
\Psi_{X,U,V}(x\otimes v\otimes w) = \psi (x\otimes u\otimes v).
\]
\end{Lem}

It is also easy to see that this way we get a one-to-one correspondence between module category structures on $\Mod{B}$ with strict unit conditions and quasi-coactions on $B$.

Additional braided structure can be implemented analogously. Recall that a \emph{quasi-triangular quasi-bialgebra} is given by a quasi-bialgebra $A$ together with an invertible element $\msR\in A\otimes A$ satisfying
\begin{align*}
\msR\Delta(a)\msR^{-1} &= \Delta(a)_{21},&
(\Delta\otimes \id)(\msR) &= \phi_{312}\msR_{13}\phi_{132}^{-1}\msR_{23}\phi,&
(\id\otimes \Delta)(\msR) &= \phi_{231}^{-1}\msR_{13}\phi_{213}\msR_{12}\phi^{-1}.
\end{align*}
Then $\Mod{A}$ becomes a braided tensor category with
$$
\beta_{U,V}(u \otimes v) = \Sigma \msR (u \otimes v),
$$
where $\Sigma\colon U \otimes V \to V \otimes U$ is the flip map $u\otimes v \mapsto v\otimes u$. Let $\sigma$ be an automorphism of $(A,\Delta,\phi,\msR)$, so in particular
\begin{align*}
(\sigma \otimes \sigma \otimes \sigma)(\phi) &= \phi,& (\sigma \otimes \sigma)(\msR) &= \msR.
\end{align*}
It then defines a (strict) braided autoequivalence of $\Mod{A}$, which we denote by the same symbol $\sigma$: if $U$ is a left $A$-module, then $\sigma(U)=U$ as a vector space, while the module structure is given by $a \cdot_{\sigma} x=\sigma(a) x$.

\begin{Def}
Let $(B,\alpha,\psi)$ be a quasi-coaction for a quasi-triangular quasi-bialgebra $(A,\Delta,\phi,\msR)$. A \emph{$\sigma$-braid} for $B$ is an invertible element $\msE\in B\otimes A$ such that $\msE(\id\otimes\sigma)\alpha(b)=\alpha(b)\msE$ for all $b\in B$ and $\msE$ satisfies the $\sigma$-octagon relation
$$
(\alpha\otimes \id)(\msE) = \psi^{-1} \msR_{21} \psi_{021} \msE_{02}(\id\otimes\id\otimes\sigma)(\psi_{021}^{-1} \msR_{12}\psi).
$$
We say that $\msE$ is a \emph{ribbon $\sigma$-braid} if moreover the ribbon $\sigma$-braid equation holds:
$$
(\id\otimes \Delta)(\msE) = \msR_{21} \psi_{021} \msE_{02} (\id\otimes\id\otimes\sigma)(\psi_{021}^{-1} \msR_{12} \psi) \msE_{01}(\id\otimes\sigma\otimes\sigma)(\psi^{-1}).
$$
\end{Def}

Note that in the above equations $B$ is considered at place $0$.

The following lemma is now immediate.

\begin{Lem}
Let $\msE$ be a $\sigma$-braid (resp.~ribbon $\sigma$-braid) for $(B,\alpha,\psi)$. Then $\Mod{B}$ becomes a $\sigma$-braided (resp.~ribbon $\sigma$-braided) $\Mod{A}$-module category by
$$
\eta_{X,U}(x \otimes u) =\msE(x \otimes u).
$$
\end{Lem}

\begin{Rem}\label{rem:sigma-change}
Assume that $(B,\alpha,\psi)$ is a quasi-coaction for a quasi-triangular quasi-bialgebra $(A,\Delta,\phi,R)$, $\msE$ is a (ribbon) $\sigma$-braid and $g\in A$ is a group-like element. Then $\msE(1\otimes g^{-1})$ is a (ribbon) $(\Ad g)\sigma$-braid, and the associated twist-braided module categories are equivalent. Therefore, if $\sigma$ is inner and implemented by a group-like element, we can pass from the $\sigma$-braided setting to the usual one. It follows that by replacing $A$ by the crossed product $A\rtimes_\sigma\Z$ we can always get (ribbon) braids from (ribbon) $\sigma$-braids. Note also that if $\sigma$ is of finite order $n$, then instead of $\Z$ we can take $\Z/n\Z$.
\end{Rem}

In the setting we will be interested in, the elements $\phi,\psi,\msR,\msE$ will live in a completed tensor product. More precisely, assume that we have a dense unital $\Delta$-preserving embedding $A \subseteq M(\msA)$, where $\msA=(\msA,\Delta)$ is a discrete multiplier quasi-bialgebra \cite{MR2832264}, so as an algebra
\[
\msA \cong \bigoplus_{\lambda \in \Lambda} \End(V_{\lambda})
\]
for some index set $\Lambda$ and finite-dimensional vector spaces $(V_{\lambda})_{\lambda \in \Lambda}$. Here $M(\msA)$ is the \emph{multiplier algebra} of the (non-unital) algebra $\msA$, which we can identify with
\[
M(\msA) \cong \prod_{\lambda \in \Lambda} \End(V_{\lambda}).
\]
The embedding of $A \subseteq M(\msA)$ is \emph{dense} in the sense that $xAy = x\msA y$ for all elements $x,y$ in $\msA$. For $B$ a unital algebra, we write
\[
B \hotimes A \hotimes \ldots \hotimes A = M(B\otimes \msA \otimes \ldots \otimes \msA) = \prod_{\lambda,\ldots,\mu} B\otimes \End(V_{\lambda}\otimes \ldots \otimes V_{\mu}),
\]
and we make use of the natural inclusions of the form
\[
A\otimes \cdots \otimes A \subseteq A \hotimes \cdots \hotimes A.
\]

We then call \emph{generalized quasi-coaction} of $A$ on $B$ any homomorphism $\alpha\colon B \rightarrow B\hotimes A$ together with an invertible element $\psi \in B \hotimes A \hotimes A$ satisfying
\[
\psi (\alpha\otimes \id)\alpha(b)\psi^{-1} = (\id\otimes \Delta)\alpha(b),
\]
interpreting the identity in $B \hotimes A \hotimes A$. Similarly the notion of (ribbon) $\sigma$-braid can be modified by only requiring $\msE \in B \hotimes A$.

\section{Symmetric pairs of Lie algebras and symmetric spaces}
\label{sec:symm-pair-lie-alg-symm-sp}

In this section we recall some elements of the theory of symmetric pairs and symmetric spaces.

\subsection{Symmetric pairs and symmetric spaces}

By a \emph{Lie $*$-algebra} we mean a complex Lie algebra $\mfg$ endowed with an involutive conjugate-linear anti-homomorphism $*\colon \mfg \to \mfg$. Let $\mfg_0$ be a real Lie algebra with complexification $\mfg = \mfg_0^{\C} = \mfg_0 \otimes_{\R}\C$. We can recover $\mfg_0$ from $\mfg$ as $\mfg_0 = \{x\in \mfg\mid x^* =-x\}$ if we endow $\mfg$ with the Lie $*$-algebra structure
\[
*\colon \mfg \rightarrow \mfg,\quad x+i y \mapsto (x+iy)^* = -x + i y,\qquad (x,y\in \mfg_0).
\]
Every Lie $*$-algebra $(\mfg,*)$ is obtained from a real Lie algebra $\mfg_0$ as above, and we call $\mfg_0\subseteq \mfg$ the associated \emph{real form} of $(\mfg,*)$. In particular, when $\mfg$ is semisimple with the Killing form $B(x,y) = \Tr(\ad_\mfg(x)\ad_\mfg(y))$, we reserve the notation $*$ for the $*$-structure defined by a fixed compact real form $\mfu \subseteq \mfg$. Note that the associated sesqui-linear form $\langle x,y \rangle = B(x,y^*)$ on $\mfg$ will be positive definite.

In the following we fix a compact semisimple Lie algebra $\mfu$.
\begin{Def}
Let $\mfk \subsetneq \mfu$ be a proper real Lie subalgebra. The pair $\mfk \subseteq \mfu$ is called a \emph{compact symmetric pair} if there exists  an involutive Lie algebra automorphism $\sigma$ of $\mfu$ such that $\mfk = \mfu^{\sigma}$, the fixed point Lie algebra of $\sigma$.
\end{Def}
Note that the nontrivial involutive automorphisms of $\mfu$ are in one-to-one correspondence with the $*$-preserving Lie algebra involutions\footnote{By involution we will always mean a \emph{proper} involution, i.e., not the identity.} of $\mfg = \mfu^{\C}$. Since the Killing form is $\sigma$-invariant, $\sigma$ is unitary with respect to the Hermitian scalar product on $\mfg$, and it follows that $\sigma$ is uniquely determined by $\mfk$.

Let $U$ be a simply connected compact Lie group (unique up to isomorphism) integrating~$\mfu$. Let $\sigma$ be an involution on $\mfu$. Then $\sigma$ integrates uniquely to an involutive Lie group automorphism of $U$, again denoted by $\sigma$. Let $K = U^{\sigma}$ be the compact Lie subgroup of $\sigma$-fixed points. We have that $K$ is connected \cite{MR1834454}*{Theorem VII.8.2}, with Lie algebra $\mfk$.

\begin{Def}
With $K\subseteq U$ as above, we call $U/K$ the \emph{standard symmetric space} associated to $\sigma$.
\end{Def}

We will in the following always consider $U/K$ as a $U$-space by the natural action $U \curvearrowright U/K$.
By isomorphisms of symmetric spaces we mean $U$-equivariant homeomorphisms.

\begin{Def}
Two involutions $\sigma_1,\sigma_2$ of $\mfu$, or the corresponding symmetric pairs, are called \emph{equivalent} if there exists an automorphism of $\mfu$ carrying $\mfk_1 = \mfu^{\sigma_1}$ to $\mfk_2 = \mfu^{\sigma_2}$. We call $\sigma_1$ and $\sigma_2$ \emph{inner equivalent} if we can choose the automorphism of the form $\Ad(u)$ for some $u\in U$.
\end{Def}

It is clear that $\sigma_1,\sigma_2$ are inner equivalent if and only if $U/K_1$ and $U/K_2$ are isomorphic as $U$-spaces.

There are two standard ways to classify involutions of $\mfu$: one is by means of \emph{Satake diagrams}, the other by means of \emph{Vogan diagrams}. In the following, we will recall the main ingredients of these classifications. Both classifications depend explicitly on the choice of a fixed Cartan subalgebra $\mft \subseteq \mfu$ together with a positive root system $\Delta^+ \subseteq \mfh^*$, where $\mfh = \mft^{\C}\subseteq \mfg$. We write $W$ for the Weyl group and $I =\{\alpha_1,\ldots,\alpha_l\} \subseteq \Delta^+$ for the set of simple positive roots. When considering the associated Dynkin diagrams, we explicitly take $I$ as the set of vertices. Moreover, we fix as well Chevalley generators $\{e_r,f_r,h_r\mid r\in I\}$ for $\mfg$ such that, under the $*$-structure on $\mfg$ defining~$\mfu$, one has $e_r^* = f_r$ and $h_r^* =h_r$. In this case we say that the generators are \emph{compatible with the compact form}~$\mfu$ of~$\mfg$, or that they are \emph{$*$-compatible}. Note that with such a choice of generators the Chevalley automorphism
\begin{align*}
\omega(h_r)&=-h_r,&
\omega(e_r)&=-f_r,&
\omega(f_r)&=-e_r.
\end{align*}
is a $*$-automorphism of $\mfg$.

Denote by $T\subset U$ the maximal torus defined by $\mft$.  For $\tau$ an automorphism of the Dynkin diagram, we denote as well by $\tau$ the corresponding $*$-preserving automorphism of $\mfg$ determined by
\begin{align*}
h_r &\mapsto h_{\tau(r)},&
e_r &\mapsto e_{\tau(r)},&
f_r &\mapsto f_{\tau(r)}.
\end{align*}
For $X\subseteq I$, let $W_X$ be the subgroup of $W$ generated by the simple reflections $s_r$ corresponding to the roots $\alpha_r$, $r\in X$. We also view the $s_r$ as automorphisms of $\mfh$ by duality, and we denote the canonical lift of the longest element~$w_X$ to the normalizer $N_U(T)$ by~$m_X$,~so
\[
m_X =m_{r_1}\ldots m_{r_k},\quad m_r = \exp(e_r)\exp(-f_r)\exp(e_r),
\]
if $w_X = s_{r_1}\ldots s_{r_k}$, as we have $s_r = \Ad(m_r)$.

\begin{Def}
We say that the involution $\theta$ is in \emph{Satake form} (with respect to the chosen Chevalley presentation) if there exist an involutive automorphism $\tau$ of the Dynkin diagram, a globally $\tau$-invariant subset $X\subseteq I$ and unimodular numbers $(z_r)_{r \in I}$ such that
\begin{align}
\label{eq:theta-formula-rel-to-zr-and-mx}
\theta(h_r) &= -(\Ad m_X)(h_{\tau(r)}),&
\theta(e_r) &= -\bar z_{\tau(r)}(\Ad m_X)(f_{\tau(r)}),&
\theta(f_r) &= -z_{\tau(r)}(\Ad m_X)(e_{\tau(r)}),
\end{align}
where we assume that the following conditions are satisfied:
\begin{gather}
\text{the action of $\tau$ on $X$ coincides with the action of $-w_X$};\label{cond:admis1}\\
z_r=1\text{ for }r\in X.\label{cond:z}
\end{gather}
\end{Def}

In other words, we can write
$$
\theta=(\Ad z)\circ(\Ad m_X)\circ\tau\circ\omega,
$$
where $z\in T$ denotes any element such that $z(\alpha_r)=z_r$. If one defines a map $\theta$ by the formula \eqref{eq:theta-formula-rel-to-zr-and-mx}, $\theta^2=\id$ is equivalent to
\begin{equation}\label{eq:z2}
z_r\bar z_{\tau(r)}=(-1)^{2(\alpha_r,\rho_X^\vee)},
\end{equation}
where $\rho_X^\vee$ is half the sum of the positive coroots of the root system generated by $X$, see~\cite{MR3269184}*{Section~2.4}. In particular, we must have
\begin{equation}\label{cond_admis2}
(\alpha_r,\rho_X^\vee)\in\Z\text{ if }\tau(r)=r.
\end{equation}

\begin{Def}[\cite{MR3269184}*{Definition 2.3}]
A pair $(X,\tau)$, consisting of an involutive automorphism $\tau$ of the Dynkin diagram and a $\tau$-invariant subset $X\subsetneq I$ satisfying conditions~\eqref{cond:admis1} and~\eqref{cond_admis2} is called \emph{admissible}.\footnote{In \cite{MR3269184} condition \eqref{cond_admis2} is only stated for $i\in I\setminus X$, but for $i\in X$ one automatically has $(\rho_X^\vee,\alpha_r) = 1$. We also exclude the case $X = I$ as it corresponds to the identity automorphism.}
\end{Def}

From the above discussion, an involution in Satake form gives an admissible pair. Conversely, starting with an admissible pair we can choose $*$-compatible Chevalley generators $(e_r, f_r, h_r)_{r \in I}$ and numbers $z_r \in \T$ satisfying conditions~\eqref{cond:z} and~\eqref{eq:z2}, and define an involution in Satake form by \eqref{eq:theta-formula-rel-to-zr-and-mx}. How exactly we choose $z_r$ is not important, since the involutions corresponding to different choices are conjugate by inner automorphisms defined by elements of $T$, or equivalently, we can obtain Satake forms of the same involution with different scalars $z_r$ by rescaling the generators $e_r$ and $f_r$. In particular, in the choice of $z_r$, besides satisfying~\eqref{cond:z} and~\eqref{eq:z2}, we can always set $z_r=1$ for all $r$ fixed by $\tau$. For every admissible pair $(X,\tau)$ we make such a choice and denote the corresponding element $z$ by $s(X,\tau)$ and the involution $(\Ad s(X,\tau))\circ(\Ad m_X)\circ\tau\circ\omega$ by $\theta(X,\tau)$.

Note that an involution in Satake form preserves in particular the Cartan subalgebra $\mfh$. Let $\Theta\colon \mfh^* \rightarrow \mfh^*$ define the involutive transformation dual to $\theta_{\mid \mfh}$. The \emph{Satake diagram} of $\theta$ is obtained from the Dynkin diagram as follows:
\begin{itemize}
\item the vertices corresponding to the simple roots in $X$ are painted black, and
\item two distinct simple roots~$\alpha_r$ and~$\alpha_s$ in $I\setminus X$ such that $\Theta(\alpha_r)+\alpha_s\in\Z X$ are joined by an arrow.
\end{itemize}
Repeatedly using $\beta - s_{r} \beta = (\alpha_r^\vee, \beta) \alpha_r$, we obtain $\beta - w_X(\beta) \in \Z X$ for any $\beta$ in the root lattice $Q$. Combining this for $\beta = \Theta(\alpha_r)$ with the identity $-(w_X\circ\Theta)(\alpha_r)=\alpha_{\tau(r)}$, we see that the vertices $r$ and $s$ are joined by an arrow if and only if $r=\tau(s)$. Therefore the admissible pairs and the Satake diagrams contain literally the same information. Hence we will use these terminologies interchangeably.

Modulo the explicit statement about $*$-compatibility, the following result is the standard classification of symmetric pairs in terms of Satake diagrams as formulated, e.g., in~\cite{MR3269184}*{Theorem~2.7}. For the benefit of the reader, we provide some further information on its proof in Appendix \ref{Ap1}.

\begin{Theorem}\label{TheoSat}
Let $\mfu$ be a compact semisimple Lie algebra with fixed $*$-compatible Chevalley generators in~$\mfg=\mfu^\C$. If $\sigma$ is an involution of $\mfu$, there exists a Satake diagram $(X,\tau)$ such that $\sigma$ is inner equivalent with $\theta(X,\tau)$ for some, and hence any, choice of $s(X,\tau)$. Furthermore, up to an automorphism of the Dynkin diagram, $(X,\tau)$ depends only on the equivalence class of $\sigma$.
\end{Theorem}




\bigskip
Next let us review Vogan diagrams.

\begin{Def}
We say that an involutive automorphism $\nu$ of $\mfu$ is in \emph{Vogan form} (with respect to a given set of Chevalley generators) if there exists an involutive automorphism $\mu$ of the Dynkin diagram such that
\begin{align*}
\nu(h_r) &= h_{\mu(r)},&
\nu(e_r) &= \epsilon_r e_{\mu(r)},&
\nu(f_r) &= \epsilon_r f_{\mu(r)}
\end{align*}
for some $\epsilon_r =\pm1$ with $\epsilon_r=1$ if $\mu(r)\ne r$. We denote by $Y$ the set of simple roots fixed by $\mu$ such that $\epsilon_r=-1$.
\end{Def}

A pair $(Y,\mu)$ as above, consisting of an involutive automorphism $\mu$ of the Dynkin diagram and a set $Y\subseteq I$ pointwise fixed by $\mu$, can be encoded as a \emph{Vogan diagram} by making the vertices in $Y$ black and joining vertices $r,s$ by an arrow if $r= \mu(s)$, see~\cite{MR1920389}*{Section~VI.8}. Therefore, we will refer to $(Y,\mu)$ itself as a Vogan diagram. We let $\nu(Y,\mu)$ be the involution associated to the Vogan diagram $(Y,\mu)$ (and a fixed set of Chevalley generators). We exclude the case where $\mu = \id$ and $Y = \emptyset$, which corresponds to the identity map.

Again, the proof of the following theorem is provided in Appendix \ref{Ap1}.

\begin{Theorem}\label{TheoVog}
Let $\mfu$ be a compact semisimple Lie algebra with fixed $*$-compatible Chevalley generators in $\mfg=\mfu^\C$. If $\sigma$ is an involution of $\mfu$, there exists a Vogan diagram $(Y,\mu)$ such that $\sigma$ is inner equivalent with $\nu(Y,\mu)$.
\end{Theorem}

\begin{Rem}
As opposed to Satake diagrams, the equivalence class of an involution can be described by non-isomorphic Vogan diagrams. We say that $(Y,\mu)$ is \emph{standard}, and then that the corresponding involution $\nu=\nu(Y,\mu)$ is in \emph{standard Vogan form}, if the following conditions are satisfied:
\begin{itemize}
\item every connected component of the Dynkin diagram contains at most one element of $Y$,
\item if $r\in Y$ and $\mu=\id$ on the connected component containing $r$, then $(\varpi_r-\varpi_s,\varpi_s)\le0$ for all $s$,
\end{itemize}
where $\varpi_r$ denote the fundamental weights. It can be shown that every involution $\nu$ of $\mfu$ is inner equivalent with $\nu(Y,\mu)$ for a standard Vogan diagram~$(Y,\nu)$. For this one starts with a Vogan form of $\nu$ and then makes a new, more careful choice of simple roots, see~\cite{MR1920389}*{Section~VI.10}. Standard Vogan diagrams are much closer to being uniquely (up to conjugacy) associated to the conjugacy classes of involutions. There is, however, still some ambiguity, corresponding to the obvious isomorphisms $\mathfrak{so}(p,q)\cong\mathfrak{so}(q,p)$ and $\mathfrak{sp}(p,q)\cong\mathfrak{sp}(q,p)$ of real simple Lie algebras.
\end{Rem}

\section{Three constructions of ribbon braided module \texorpdfstring{C$^*$}{C*}-categories from symmetric pairs}
\label{ref:three-const-ribb-br-mod}

Throughout this section $\mfu$ will be a fixed compact semisimple Lie algebra with complexification $\mfg= \mfu^{\C}$ and a fixed set of Chevalley generators $\{e_r,f_r,h_r\mid r\in I\}$ compatible with $\mfu$. We let $\mfh \subseteq \mfg$ be the associated Cartan subalgebra with root and weight lattices $Q$ and $P$, respectively, with positive part $Q^+$ and $P^+$. We write the simple positive roots as $\alpha_r$, and the associated fundamental weights in $P^+$ as $\varpi_r$. We fix also a non-degenerate $\mfg$-invariant form $(-,-)_{\mfg}$ on $\mfg$ for which the induced form $(-,-)$ on $\mfh^*$ makes the short roots of each simple summand have square length $2$. We let $A = (a_{rs})_{rs}$ be the associated Cartan matrix of $\mfg$, with entries $a_{rs} = ({\alpha}^\vee_r,\alpha_s) = 2 (\alpha_r, \alpha_s) / (\alpha_r, \alpha_r)$, and write $d_A$ for the determinant of $A$ (for our purposes we may as well take the least common multiple of the determinants of Cartan matrices of irreducible components), which is also equal to the index of $Q$ in $P$. By our normalization, $(-, -)_{\mfg}$ takes integral values on $Q \times Q$, and on $P \times P$ its values belong to $d_A^{-1} \Z$. 

\subsection{\texorpdfstring{$q$}{q}-deformed tensor \texorpdfstring{C$^*$}{C*}-categories}

Let $\mbq$ be an indeterminate variable. We form the Hopf algebra $U_{\mbq}(\mfg)$  over the field $\C(\mbq^{1/d_A})$ of rational functions in $\mbq^{1/d_A}$ using the same conventions as in \cite{MR3269184}, but taking the Cartan part to be labeled by elements of $P$, as follows. Let $U_\mbq(\mfn^+)$ be the algebra over $\C(\mbq^{1/d_A})$ generated by the elements $(E_r)_{r\in I}$ subject to the relations
\[
\sum_{n=0}^{1-a_{rs}}(-1)^n \qbin{1-a_{rs}}{n}{\mbq_r} E_r^{1-a_{rs}-n} E_s E_r^n = 0 \qquad (r \neq s),
\]
with $\mbq_{r} = \mbq^{\frac{1}{2}(\alpha_r,\alpha_r)}$. Let $U_\mbq(\mfn^-)$ be an isomorphic copy of $U_\mbq(\mfn^+)$ with generators $F_r$, and let $U_\mbq(\mfh)$ be the group algebra of $P$, spanned by the elements $(K_{\omega})_{\omega \in P}$. Then $U_\mbq(\mfg)$ is defined as the algebra generated by $U_\mbq(\mfn^+),U_\mbq(\mfn^-)$ and $U_\mbq(\mfh)$ with the following interchange relations:
\begin{align*}
K_{\omega} E_r &= \mbq^{(\omega,\alpha_r)}E_rK_{\omega},&
K_{\omega} F_r &= \mbq^{-(\omega,\alpha_r)}F_rK_{\omega},&
[E_r,F_s] &= \delta_{rs}\frac{K_{r} - K_{r}^{-1}}{\mbq_r-\mbq_r^{-1}},
\end{align*}
where $K_r  =K_{\alpha_r}$.  Multiplication gives a vector space isomorphism
\[
U_\mbq(\mfn^-)\otimes U_\mbq(\mfh)\otimes U_\mbq(\mfn^+)\cong U_\mbq(\mfg),
\]
called the \emph{triangular decomposition}, and this still holds with any of the factors on the left interchanged. One can endow $U_\mbq(\mfg)$ with a unique Hopf algebra structure such that
\begin{align*}
\Delta(K_{\omega}) &= K_{\omega}\otimes K_{\omega},&
\Delta(E_r) &= E_r\otimes 1+ K_{r}\otimes  E_r,&
\Delta(F_r) &= F_r\otimes K_{r}^{-1} + 1\otimes F_r.
\end{align*}
We denote by $U_\mbq(\mfb^{\pm})$ the Hopf subalgebras generated by $U_\mbq(\mfn^{\pm})$ and $U_\mbq(\mfh)$.

By specializing $\mbq$ at $q \in \C^\times$ not a root of unity, we obtain the Hopf algebra $U_q(\mfg)$ over $\C$. In case $0<q<1$, we can moreover endow $U_q(\mfg)$ with a good Hopf $*$-algebra structure
\begin{align*}
E_r^* &= F_r K_r,&
F_r^* &= K_r^{-1}E_r,&
K_{\omega}^* &= K_{\omega}.
\end{align*}
The resulting Hopf $*$-algebra will be denoted $U_q(\mfu)$.\footnote{The theory can also be developed for $q>1$, but this case would need a slight modification in later sections, so we restrict to $0<q<1$ from the outset. Note that the $*$-structure is also well-defined for $q<0$, but the resulting $*$-algebra does not have a good $*$-representation theory.} 

We call a $U_{\mbq}(\mfg)$-module \emph{admissible} if its restriction to $U_\mbq(\mfh)$ is a direct sum of one-dimensional modules of the form
\[
K_{\omega} \mapsto \mbq^{(\omega,\chi)},\qquad \chi \in P.
\]
It is called \emph{finite} if it is finite-dimensional. 

A similar definition can be made for $U_q(\mfg)$. For $U_q(\mfu)$ we moreover assume that the module is equipped with a pre-Hilbert space structure compatible with the $*$-structure, in which case we call it a \emph{$*$-representation}. Finite admissible $*$-representations are then in one-to-one correspondence with finite $*$-representations under which each $K_{\omega}$ becomes a positive operator.

\begin{Def}
For $0<q<1$, we denote by $\Rep_q(\mfu)$ the tensor C$^*$-category of finite admissible $*$-representations of the Hopf $*$-algebra $U_q(\mfu)$.
\end{Def}

The equivalence classes of irreducible objects in $\Rep_q(\mfu)$ can be labeled by $\varpi \in P^+$, the set of dominant integral weights, and for each $\varpi\in P^+$ one can canonically construct an irreducible $*$-representation $V_{\varpi}^q$ generated by a highest weight vector $\xi_{\varpi}$ vanishing under the $E_r$ and with
$$
K_{\chi}\xi_{\varpi} = q^{(\chi,\varpi)}\xi_{\varpi}.
$$
We can then make the (non-unital) direct sum $*$-algebra
$$
\msU_q(\mfu) = \bigoplus_{\varpi\in P^+} B(V_{\varpi}^q)
$$
into a discrete multiplier $*$-bialgebra $(\msU_q(\mfu) ,\Delta_q)$ with dense, coproduct-compatible embedding \[U_q(\mfu) \subseteq M(\msU_q(\mfu)).\] There is a unique quasi-triangular structure $\msR_q \in U_q(\mfu)\hotimes U_q(\mfu)$ such that\footnote{Our $\msR_q$ corresponds to $\msR_{21}^{-1}$ in \cite{MR2832264} as to have the same conventions for the $R$-matrix as in \cite{MR3269184}.} for all $\varpi,\chi \in P^+$
\[
\msR_q(\xi_{\varpi}\otimes \eta_{w_0\chi}) = q^{-(\varpi,w_0\chi)} \xi_{\varpi} \otimes \eta_{w_0\chi},
\]
where $\eta_{w_0\chi}$ is the lowest weight vector in $V_{\chi}^q$ of weight $w_0\chi$ and $w_0$ is the longest element in the Weyl group of $\mfg$. Correspondingly, $\Rep_q(\mfu)$ becomes a braided tensor C$^*$-category.

An equivalent braided tensor C$^*$-category can be obtained from monodromy of Knizhnik--Zamolod\-chikov equations. Specifically, let $\{X_i\}_{i \in \mcI}$ be an orthonormal basis of $\mfg$ with respect to the Hermitian inner product $\langle X,Y\rangle_{\mfg} = (X,Y^*)_{\mfg}$, and put
$$
t = \sum_{i \in \mcI} X_i^* \otimes X_i \in\Sym^2(\mfg)^{\mfg}.
$$
Let $V_1,\ldots,V_n$ be finite-dimensional $\mfu$-representations, and $\hbar\in i\R$. On the configuration space of $n$ distinct complex numbers
$$
\Omega_n = \Bigl\{z \in \C^n \mid \prod_{i\neq j}(z_i-z_j) \neq 0\Bigr\}
$$
we consider the following system of differential equations on $V_1\otimes\cdots \otimes V_n$-valued functions, called \emph{KZ$_n$-equations} (of type A),
$$
\frac{\partial v}{\partial z_i} = \hbar \sum_{j\neq i} \frac{t_{ij}}{z_i-z_j}v \quad (i = 1, \ldots, n),
$$
where $t_{ij}$ means $t$ acting on the $i$-th and $j$-th tensors. It is well-known that the appropriately normalized monodromy of KZ$_3$ from the region $\absv{z_2-z_1} \ll \absv{z_3-z_1}$ to the region $\absv{z_3-z_2} \ll \absv{z_3-z_1}$ then gives a family of unitary operators
$$
\Phi_{\hbar}\colon (V \otimes W) \otimes Z \rightarrow V \otimes (W \otimes Z),
$$
providing a non-trivial unitary associator on the tensor C$^*$-category $\Rep(\mfu)$ of finite-dimensional unitary representations of $\mfu$. This becomes a braided tensor C$^*$-category with respect to the braiding\footnote{Again, we choose the opposite braiding of \cite{MR2832264}.} defined by $\Sigma e^{-\pi i \hbar t}$, where $\Sigma$ is the flip map. Correspondingly, if we write $U(\mfu)$ for the enveloping Hopf algebra of $\mfg$  with the $*$-structure induced by $\mfu$, together with its dense embedding
$$
U(\mfu) \subseteq M(\msU(\mfu)),\qquad \msU(\mfu) = \bigoplus_{\varpi\in P^+} B(V_{\varpi}),
$$
where again $V_{\varpi}$ is constructed with respect to the fixed Chevalley generators, then $U(\mfu)$ becomes a quasi-triangular quasi-Hopf algebra with respect to $\phi_{\hbar} \in U(\mfu)^{\hat{\otimes}3}$ defined by $\Phi_{\hbar}$ and $\msR_{\hbar} = e^{-\pi i \hbar t} \in U(\mfu)^{\hat{\otimes} 2}$. We denote by $\Rep_{\hbar}(\mfu)$ the corresponding braided tensor C$^*$-category.

One has the following result.

\begin{Theorem}[\cite{MR2832264}, cf.~\citelist{\cite{MR1047964}\cite{MR1239506}\cite{MR1239507}}] Let $0<q<1$, and let $\hbar\in i\R_{>0}$ be such that $q = e^{\pi i \hbar}$. Then there exists an equivalence of braided tensor C$^*$-categories
\[
F\colon \Rep_{\hbar}(\mfu) \rightarrow \Rep_q(\mfu)
\]
such that $F(V_{\varpi}) = V_{\varpi}^q$.
\end{Theorem}

The above equivalence is, moreover, unique up to a natural unitary monoidal isomorphism~\citelist{\cite{MR2782190}\cite{MR2959039}}. In particular, the Tannaka reconstruction applied to the forgetful functor on
$\Rep_q(\mfu)$ composed with any $F$ as above will produce a discrete multiplier Hopf $*$-algebra with a distinguished isomorphism to $(\msU_q(\mfu),\Delta_q)$.

\subsection{Ribbon braided module \texorpdfstring{C$^*$}{C*}-categories from KZ equations of type B}
\label{sec:ribbon-br-mod-from-KZ}

Let $\sigma$ be an involution on $\mfu$ (recall that we assume involutions to be nontrivial, so $\sigma \neq \id$). In this subsection, we will construct from $\sigma$ a ribbon $\sigma$-braided module C$^*$-category over $\Rep_{\hbar}(\mfu)$.

Let $\mfg_{\pm}$ be the $\pm1$-eigenspaces in $\mfg$ for $\sigma$, and write $\mfu_{\pm} = \mfg_{\pm}\cap \mfu$. We will also write
$$
\mfu_+ = \mfk,\quad \mfu_- = \mfm,\quad \text{so} \quad \mfg_+ = \mfk^{\C},\quad \mfg_- = \mfm^{\C}.
$$
Let $\{X_i\}_{i\in \mcI_{\pm}}$ be an orthonormal basis of $\mfg_{\pm}$ for the restriction of Hermitian inner product on $\mfg$ as in the previous section, and put
$$
t^+  = t^{\mfk} =  \sum_{i\in \mcI_{+}} X_i^* \otimes X_i \in \mfk^{\C}\otimes \mfk^{\C},\qquad t^-  = t^{\mfm} =  \sum_{i\in \mcI_{-}} X_i^* \otimes X_i \in \mfm^{\C}\otimes \mfm^{\C}
$$
so that $t = t^{+}+t^{-}$. Finally, let us denote the Casimir element of $\mfk$ as
$$
C^{+} = C^{\mfk} =  \sum_{i\in \mcI_+} X_i^*X_i \in U(\mfk^{\C}).
$$

The following definition introduces the Knizhnik--Zamolodchikov equations associated to the hyperplane arrangement of the Coxeter group of type B, as constructed from $\sigma$ in~\cite{MR2892463}*{Section 1.3}.

\begin{Def}[\citelist{\cite{MR1017085}\cite{MR1289327}\cite{MR1940926}}]
Let $\hbar \in \C$,
$$
\Omega_n' = \{w\in \C^n \mid \prod_{\pm, i\neq j} w_i(w_i\pm w_j)\neq 0\},
$$
$V_0$ be a finite-dimensional representation of $\mfk$, and $V_1,\ldots,V_n$ finite-dimensional representations of $\mfu$. The system of \emph{2-cyclotomic KZ$_n$-equations} is the following system of differential equations on $V_0 \otimes \cdots \otimes V_n$-valued functions on $\Omega_n'$:
\begin{equation}\label{eq:KZ2}
\frac{\partial v}{\partial w_i} = \hbar\left( \frac{2 t^{\mfk}_{0,i} + C^{\mfk}_i}{w_i} + \sum_{\pm,j\neq i} \frac{t_{i,j}^{\mfk} \pm t_{i,j}^{\mfm}}{w_i\mp w_j}\right)v, \qquad (i = 1, \ldots, n).
\end{equation}
\end{Def}

Note that the coefficients of the above system indeed lie in $U(\mfk^{\C}) \otimes U(\mfg)^{\otimes n}$, so that the equations are meaningful. Let us check explicitly that this is a compatible system of equations.

\begin{Lem}
The $2$-cyclotomic KZ$_n$-system \eqref{eq:KZ2} is flat for any $n\ge2$.
\end{Lem}
\begin{proof}
We verify that the conditions 1--4) of \cite{MR1289327}*{Proposition 1.2} are satisfied with
\begin{align*}
\tau_{ij} &= t_{ij},&
\mu_{ij} &= t_{ij}^{\mfk}-t_{ij}^{\mfm} = (\id\otimes \sigma)(t)_{ij},&
\nu_i &= 2t_{0i}^{\mfk}+C_i^{\mfk}.
\end{align*}
Equations 1) and 2.a) follow from $\lbrack t,\Delta(X)\rbrack = 0$ for all $X\in \mfg$. Equation 2.b) follows from
\begin{equation*}
\begin{split}
\lbrack \mu_{ik},\tau_{ij}+\mu_{jk}\rbrack &= \lbrack (\id\otimes \sigma)(t)_{ik},t_{ji} + (\id\otimes \sigma)(t)_{jk}\rbrack \\
&= \sigma_k(\lbrack t_{ik},t_{ji}+t_{jk}\rbrack) \\
&= 0,
\end{split}
\end{equation*}
where we used that $\sigma$ is a Lie algebra automorphism.

For 3.a), we use that $\Delta(C^{\mfk}) = (C^{\mfk}\otimes 1) + (1\otimes C^{\mfk}) + 2t^{\mfk}$ to compute that
\begin{equation*}
\begin{split}
\lbrack \tau_{ij} + \nu_i+\nu_j,\mu_{ij}\rbrack &= \lbrack t_{ij} + C_i^{\mfk}+C_j^{\mfk} +2t_{0,i}^{\mfk} + 2t_{0,j}^{\mfk},t_{ij}^{\mfk} -t_{ij}^{\mfm}\rbrack \\
&= \lbrack -(t_{ij}^{\mfk}-t_{ij}^{\mfm}) +\Delta(C^{\mfk})_{ij} +2 (\id\otimes \Delta)(t^{\mfk})_{0ij},t_{ij}^{\mfk}-t_{ij}^{\mfm}\rbrack \\
&= 0,
\end{split}
\end{equation*}
using also that $\lbrack t^{\mfk},\Delta(X)\rbrack = 0$ and hence $\lbrack t^{\mfm},\Delta(X)\rbrack = 0$  for all $X\in \mfk^{\C}$. Then 3.b) follows from applying $\sigma$ to the $j$-th leg.

For 3.c), we compute
\begin{equation*}
\begin{split}
\lbrack \tau_{ij}+\nu_i + \mu_{ij},\nu_j \rbrack &= \lbrack 2t_{ij}^{\mfk} + 2t_{0i}^{\mfk}+C_i^{\mfk},2t_{0j}^{\mfk}+C_j^{\mfk}\rbrack \\
&= \lbrack 2 (\Delta\otimes \id)(t^{\mfk})_{0,j,i} + C_i^{\mfk},2 t_{0j}^{\mfk}+C_j^{\mfk}\rbrack \\
&= 0.
\end{split}
\end{equation*}
Equations of 4) are obvious.
\end{proof}

We will be particularly interested in the $2$-cyclotomic KZ$_2$-equation. Analogously to the KZ$_3$-equation of type A, we can construct a modified $2$-cyclotomic KZ$_2$-equation \cite{MR2383601} on $\C \setminus \{0,\pm1\}$, with values in $V_0\otimes V_1\otimes V_2$:
\begin{equation}\label{EqKZB}
H'(w) = \hbar\left(\frac{B_-}{w+1} + \frac{2A}{w}+ \frac{B_+}{w-1}\right)H(w),
\end{equation}
where
\begin{align}\label{EqAB}
A&=t^{\mfk}_{01}+\frac{1}{2}C^{\mfk}_{1},&
B_- &= t^{\mfk}_{12} - t^{\mfm}_{12},&
B_+ &= t^{\mfk}_{12} + t^{\mfm}_{12}= t_{12}.
\end{align}
If $H$ is a solution of the modified $2$-cyclotomic KZ$_2$-equation, we see by using Lemma~\ref{LemComm} below that, with
$$
d=\hbar(2 t^{\mfk}_{01}+ 2t^{\mfk}_{02} + 2t^{\mfk}_{12} + C_1^{\mfk} + C_2^{\mfk}),
$$
the function
\[
v(w_1,w_2) = w_2^{d} H(w_1/w_2)
\]
is a solution of the $2$-cyclotomic KZ$_2$-equation.

\begin{Rem}
The modified KZ$_2$-equation can be written in a more concise form
by considering rather
$$
G(z) = H(\sqrt{z})
$$
on $\C \setminus \R^-$, which satisfies the equation
\begin{equation}
\label{eq:mKZ}
G'(z) = \hbar \left(\frac{A}{z}+\frac{B(z)}{z-1}\right)G(z), \quad \text{with } B(z) = t^{\mfk}_{12} + \frac{t^{\mfm}_{12}}{\sqrt{z}}.
\end{equation}
The second term is precisely $X(z)/z(z-1)$ in the notation of \cite{MR2126485}, but we have swapped the first and the last leg and started indexing at $0$ as to have right module category actions later on. We also chose a different cut for the $\theta_k$ than the set $\mathbf{D}$ in \cite{MR2126485}, as it allows for a more concise form of their term $X(z)$.
\end{Rem}

Following \cite{MR2383601} we define $\Psi$ as a normalized monodromy  of \eqref{EqKZB} from $0$ to~$1$.\footnote{It is also possible to define $\Psi$ as a normalized monodromy of \eqref{eq:mKZ} from $0$ to $1$, but then we have to be careful with normalization, since the natural normalization, as in~\cite{MR2126485}, leads to an operator which differs from $\Psi$ by the factor $2^{\hbar t_{12}}$.} Namely, consider general operators $a$, $b_+$ and $b_-$ on a finite dimensional vector space and the equation
\begin{equation}\label{EqKZB2}
H'(w) =\left(\frac{b_-}{w+1} + \frac{a}{w}+ \frac{b_+}{w-1}\right)H(w).
\end{equation}
Then from the theory of differential equations with singularities it is known that if the operators~$a$ and~$b_+$ have no eigenvalues that differ by a nonzero integer, then there exist unique operator valued solutions~$H_0$ and~$H_1$ of this equation on $(0,1)$ such that
\begin{itemize}
\item $H_0(w) w^{-a}$ is analytic in a neighborhood of $0$ and equal to $\id$ at $0$,
\item $H_1(w)(1-w)^{- b_+}$ is analytic in a neighborhood of $1$ and is equal to $\id$ at $1$.
\end{itemize}
We then put
$$
\Psi(a,b_+,b_-)=H_1(w)^{-1}H_0(w),
$$
which is independent of $w\in(0,1)$. Note that if $\hbar$ is purely imaginary and $a,b_+,b_-$ are Hermitian, the associated monodromy $\Psi(a,b_+,b_-)$ will be unitary.

Assume now that $\hbar\in i\R$. Let $V,W$ be finite dimensional unitary representations of $\mfu$, and $X$ a finite dimensional unitary representation of $\mfk$. Then the skew-adjoint operators $2\hbar A$ and $\hbar B_{\pm}$ from \eqref{EqAB}, considered as operators on $X \otimes V \otimes W$, have no eigenvalues that differ by a nonzero integer, so we get a unitary
$$
\Psi_{X,V,W}=\Psi(2\hbar A|_{X\otimes V\otimes W}, \hbar B_+|_{X\otimes V\otimes W},\hbar B_-|_{X\otimes V\otimes W}) \in B(X\otimes V\otimes W).
$$
By slightly abusing notation, we can write the family $(\Psi_{X,V,W})_{X,V,W}$ as
$$
\Psi=\Psi\left(\hbar(2t^{\mfk}_{01}+C^{\mfk}_{1}), \hbar t_{12},\hbar (t^{\mfk}_{12}-t^{\mfm}_{12})\right).
$$

Write $\Rep(\mfk)$ for the category of finite dimensional unitary representations of $\mfk$. We have the obvious functor
$$
\odot\colon \Rep(\mfk)\times \Rep(\mfu)\to\Rep(\mfk)
$$
of tensor product of representations of $\mfk$ with the restrictions of representations of $\mfu$ to $\mfk$. Then the morphisms $\Psi_{X,V,W}$ can be regarded as natural isomorphisms
$$
\Psi_{X, V, W}\colon (X \odot V) \odot W \rightarrow X \odot (V \otimes W),
$$
making $\Rep(\mfk)$ into a right $\Rep_{\hbar}(\mfu)$-module C$^*$-category which we will denote $\Rep_{\hbar}(\mfk)$, see \citelist{\cite{MR2126485}*{Theorem~4.6}\cite{MR2383601}*{Proposition~~2.1}}.\footnote{Note once again that the normalizations of monodromy operators in these papers are different. We use the one of the latter.}

In \cite{MR2383601}*{Section~4.5}, it is essentially proved  that $\Rep_\hbar(\mfk)$ is a braided module category for the braid given by the operator $e^{-\pi i\hbar(2t^{\mfk}_{01} + C_1^{\mfk})}$. Since Enriquez works with crossed products (see our Remark~\ref{rem:sigma-change}) and does not consider the ribbon braid relation, let us provide some details.

\begin{Prop}
For any operators $a$, $b_+$, $b_-$ on a finite dimensional vector space we have the following identity, whenever the left hand side of it is well-defined:
\begin{equation} \label{eq:Eg}
\Psi(a,b_+,b_-)^{-1}e^{\pi i b_+}\Psi(c,b_+,b_-)e^{\pi i c}\Psi(c,b_-,b_+)^{-1}e^{\pi i b_-}\Psi(a,b_-,b_+)e^{\pi i a}=1,
\end{equation}
where $c = -a - b_+ - b_-$.
\end{Prop}

\begin{proof}
This is a particular case of \cite{MR2383601}*{Proposition~2.1}. To be pedantic, Enriquez works with formal power series, but whenever all the terms in the above identity can be defined analytically, all his equalities can be specialized.
\end{proof}

We will apply this proposition to
\begin{align*}
a &= 2\hbar A = \hbar (2t^{\mfk}_{01}+C^{\mfk}_{1}),&
b_+ &= \hbar B_+= \hbar t_{12},&
b_- &= \hbar B_-=\hbar (t^{\mfk}_{12}-t^{\mfm}_{12})
\end{align*}
acting on $X\otimes V\otimes W$.

\begin{Lem}\label{LemComm}
Define
$$
d =\hbar(2 t^{\mfk}_{01}+ 2t^{\mfk}_{02} + 2t^{\mfk}_{12} + C_1^{\mfk} + C_2^{\mfk}).
$$
Then $d$ commutes with $a$, $b_+$ and $b_-$. Furthermore, with $c = -a - b_+ - b_-$, we have the following identities:
\begin{align} \label{eq:var}
d+c &= a_{02}=\hbar (2t^{\mfk}_{02}+C^{\mfk}_{2}),&
d-a &= (\Delta\otimes\id)(a),&
d &= (\id\otimes\Delta)(a).
\end{align}
\end{Lem}

\begin{proof}
As $t_{01}^{\mfk} + t_{02}^{\mfk}  = (\id\otimes \Delta)(t^{\mfk})$ and $t_{02}^{\mfk} + t_{12}^{\mfk} = (\Delta\otimes \id)(t^{\mfk})$, and $t^{\mfk}$ is $\mfk^{\C}$-invariant, it follows that $d$ commutes with $t^{\mfk}_{12}$ and $t^{\mfk}_{01}$. Clearly $d$ also commutes with $C_1^{\mfk}$. It remains to check that $d$ commutes with~$t_{12}$. The summand $t_{01}^{\mfk} + t_{02}^{\mfk}$ commutes with $t_{12}$ as $t$ is $\mfg$-invariant. For the same reason, $t_{12}$ commutes with $2t^{\mfk}_{12} + C_1^{\mfk} + C_2^{\mfk} = \Delta(C^{\mfk})$. Identities \eqref{eq:var} are straightforward from the equalities we have indicated in this proof.
\end{proof}

\begin{Theorem}\label{thm:EKZ}
For $\hbar\in i\R$, the triple  $\Rep_{\hbar}(\mfk) = (\Rep(\mfk^{\C}),\odot,\Psi,e^{-\pi i\hbar(2t^{\mfk}_{01} + C_1^{\mfk})})$ is a ribbon $\sigma$-braided right module C$^*$-category over  $\Rep_{\hbar}(\mfu) = (\Rep(\mfu),\otimes,\Phi,\Sigma e^{-\pi i\hbar t})$.
\end{Theorem}

\begin{proof}
Let $H(w)$ be a solution of the equation \eqref{EqKZB2}. If $D$ is an operator commuting with all $a,b_-,b_+$, one easily sees that $H_D(w) = H(w)z^{D}$ is a solution of \eqref{EqKZB2} with $a$ replaced by $a+D$. In particular, $\Psi(a,b_+,b_-) = \Psi(a+D,b_+,b_-)$.  Hence, as $c=a_{02}-d$, with $d$ commuting with $c,b_+,b_-$, it follows that we can write \eqref{eq:Eg} as
$$
\Psi(a,b_+,b_-)^{-1}e^{\pi i b_+}\Psi(a_{02},b_+,b_-)e^{\pi i a_{02}}\Psi(a_{02},b_-,b_+)^{-1}e^{\pi i b_-}\Psi(a,b_-,b_+)e^{\pi i a}=e^{\pi i d}.
$$
We have $\Psi(a,b_+,b_-)=\Psi$ and $\Psi(a_{02},b_+,b_-)=\Psi_{021}$. As $(\id\otimes\id\otimes\sigma)(b_+)=b_-$, we also have
$$
\Psi(a_{02},b_-,b_+)=(\id\otimes\id\otimes\sigma)(\Psi_{021})\ \ \text{and}\ \ \Psi(a,b_-,b_+)=(\id\otimes\id\otimes\sigma)(\Psi).
$$
Hence we get
\begin{equation} \label{eq:RTKZ}
\Psi^{-1}e^{\pi i b_+}\Psi_{021}e^{\pi i a_{02}}(\id\otimes\id\otimes\sigma)(\Psi_{021}^{-1}e^{\pi i b_+}\Psi) e^{\pi i a}=e^{\pi i d}.
\end{equation}

Let us now verify that this entails that $\msE = e^{-\pi i a} = e^{-\pi i\hbar (2t^{\mfk}_{01}+C^{\mfk}_1)}$ is a ribbon $\sigma$-braid with respect to $\msR = e^{-\pi i b_+} = e^{-\pi i \hbar t}$. Clearly $\msE$ commutes with $\Delta(X) = (\id\otimes \sigma)\Delta(X)$ for all $X\in U(\mfk)$.
If in \eqref{eq:RTKZ} we move $e^{\pi i a}$ to the right, take inverses and apply $\sigma$ to the last leg, we obtain (using the identity $d-a=(\Delta\otimes\id)(a)$) the $\sigma$-octagon relation
$$
\Psi^{-1} \msR_{21} \Psi_{021} \msE_{02}(\id \otimes \id \otimes \sigma)(\Psi_{021}^{-1} \msR_{12} \Psi) = (\Delta \otimes \id)(\msE).
$$
Using the relations of Lemma \ref{LemComm} and that $d$ commutes with $a$, we then obtain
$$
\Psi^{-1} \msR_{21} \Psi_{021} \msE_{02}(\id \otimes \id \otimes \sigma)(\Psi_{021}^{-1} \msR_{12} \Psi) \msE_{01} = (\id \otimes \Delta)(\msE),
$$
which is equivalent to the ribbon $\sigma$-braid equation as $\Psi = (\id\otimes \sigma\otimes \sigma)(\Psi)$ commutes with $(\id\otimes \Delta)(\msE)$.
\end{proof}

Up till now, the choice of a fixed set of Chevalley generators has not played any r\^{o}le. However, recall that this choice determined the particular form of the braided tensor equivalence
\[
\Rep_\hbar(\mfu) \rightarrow \Rep_q(\mfu),\qquad q = e^{\pi i \hbar}.
\]
By means of this equivalence, any braided autoequivalence of $\Rep_\hbar(\mfu)$ can be transported to a braided autoequivalence of $\Rep_q(\mfu)$, and any ribbon twist-braided module C$^*$-category for $\Rep_\hbar(\mfu)$ correspondingly defines one for $\Rep_q(\mfu)$. In particular, we can interpret in this way $\Rep_\hbar(\mfk)$ as a ribbon $\sigma$-braided module C$^*$-category for $\Rep_q(\mfu)$.

\begin{Rem} \label{RemEqInn}
Note that if $\sigma$ and $\sigma'$ are inner equivalent, there is an obvious isomorphism between the associated twist-braided module C$^*$-categories.
\end{Rem}

\subsection{Ribbon braided module \texorpdfstring{C$^*$}{C*}-categories from coideal quantum symmetric pairs}\label{SecCoid}

Fix again a compact semisimple Lie algebra $\mfu$ together with compatible Chevalley generators $\{e_r,f_r,h_r\mid r\in I\}$ of $\mfg = \mfu^{\C}$. Fix also an admissible pair $(X,\tau)$ with associated involution $\theta=\theta(X,\tau)$ and the fixed point Lie algebra $\mfk = \mfu^{\theta}$. Then $\theta$ has a quantum analogue~\cite{MR3269184}*{Definition 4.3}
$$
\theta_\mbq = \theta_\mbq(X,\tau) =  (\Ad s(X,\tau)) \circ T_{w_X} \circ \psi \circ \tau \circ \omega \colon U_\mbq(\mfg)\rightarrow U_\mbq(\mfg),
$$
where the different maps in the composition are all $\C(\mbq^{1/d})$-algebra automorphisms given by the following formulas:
\begin{itemize}
\item $\omega$ is the \emph{quantum Chevalley} automorphism
\begin{align*}
\omega(E_r) &= -F_r,&
\omega(F_r) &= -E_r,&
\omega(K_{\chi}) &= K_{-\chi},
\end{align*}
\item $\tau$ is the automorphism induced by the Dynkin diagram automorphism,
\begin{align*}
\tau(E_r) &= E_{\tau(r)},&
\tau(F_r) &= F_{\tau(r)},&
\tau(K_{\varpi_r}) &= K_{\varpi_{\tau(r)}},
\end{align*}
\item $\psi$ is the automorphism
\begin{align*}
\psi(E_r) &= E_rK_r,&
\psi(F_r) &= K_r^{-1}F_r,&
\psi(K_{\chi}) &= K_{\chi},
\end{align*}
\item $T_{w_X}$ is the \emph{Lusztig braid operator} (following the conventions of \cite{MR1359532}*{Section 8}, see also Remark in loc.~cit.~Section 8.6) corresponding to the longest element in the Weyl group associated to $X$: with a reduced expression
\[
w_X = s_{r_1}\ldots s_{r_M},
\]
we have
\[
T_{w_X} = T_{r_1}\ldots T_{r_M}
\]
where each $T_r$ is the algebra automorphisms determined by
\begin{align*}
T_r(E_r) &= -F_rK_r,&
T_r(F_r) &= -K_r^{-1}E_r,&
T_r(K_{\chi}) &= K_{s_r(\chi)}
\end{align*}
and for $r\neq s$
\[
T_r(E_s) = \sum_{m+n= -a_{rs}}
\frac{(-\mbq_r)^{-m}}{[m]_{\mbq_r}![n]_{\mbq_r}!}
E_r^{n}E_sE_r^{m},\quad
T_r(F_s) = \sum_{m+n = -a_{rs}}
\frac{(-\mbq_r)^{m}}{[m]_{\mbq_r}![n]_{\mbq_r}!}
F_r^{m}F_sF_r^{n},
\]
\item $\Ad(s(X,\tau))$ is given by
\[
E_r \mapsto z_rE_r,\qquad F_r \mapsto z_r^{-1}F_r,\qquad
K_{\chi}\mapsto K_{\chi}
\]
where $z_r = s(X,\tau)_r$.
\end{itemize}
The same construction applies verbatim to construct an automorphism $\theta_q$ of $U_q(\mfg)$ with $q \in \C \setminus \{ 0 \}$ not a root of unity. Note however that $\theta_{\mbq}$ is not involutive and does not preserve the comultiplication. Furthermore, $\theta_q$ does not preserve the $*$-structure when $0<q<1$.

Associated to $(X,\tau)$ one has a (family of) \emph{right coideal subalgebras}
\[
B \subseteq U_\mbq(\mfg),\qquad \Delta(B) \subseteq B\otimes U_\mbq(\mfg).
\]
To introduce them, we further borrow notation from \cite{MR3269184}. Let us first recall the parameter sets $\mcC_{\mbq}$ and $\mcS_\mbq$. Writing
\[
I_{\mcC} = \{r \in I \setminus X \mid \tau(r)\neq  r\textrm{ and }(\alpha_r,\Theta(\alpha_r)) = 0\},
\]
we put
\[
\mcC_{\mbq}  = \{(c_r(\mbq))_{r \in I \setminus X}\mid c_r(\mbq) \in \C(\mbq^{1/d})^\times \textrm{ and } c_r(\mbq) = c_{\tau(r)}(\mbq) \textrm{ for }r \in I_{\mcC}\}.
\]
(Recall that $\Theta$ denotes the involutive transformation dual to $\theta_{\mid \mfh}$.)
Note that by \cite{MR1834454}*{Chapter X, Exercises and further results, F.3-5} when $\tau(r)\neq r$ we have
\[
(\alpha_r,\Theta(\alpha_r)) = 0 \Leftrightarrow \mfg_{2\overline{\alpha_r}} = 0,
\]
where $\overline{\alpha}$ is the restriction of $\alpha$ to $\mfa_{\theta} = \{x \in i\mft \mid \theta(x) = -x \}$ and $\mfg_{\psi}$ for $\psi \in \mfa^*$ is the corresponding $\mfa_{\theta}$-eigenspace. In terms of the Satake diagram, this means $r\in I_{\mcC}$ if and only if $\tau(r) \neq r$ and there is at least one white vertex on the path between $r$ and $\tau(r)$.

Let us further write
$$
I_\ns = \{r \in I\setminus X\mid \tau(r) = r \textrm{ and
}(\alpha_r,\alpha_s)=0 \textrm{ for all } s \in X\},
$$
which correspond to the white vertices not connected to arrows or any black vertices in the Satake diagram. Next, we define
$$
I_{\mcS}=\{r\in I_\ns\mid a_{sr}\in 2\Z\textrm{ for all }s\in I_\ns\}. \footnote{This corrects a misprint in the definition of the set $\mcS$ in \cite{MR3269184}*{(5.11)}, where the $a_{ij}$ should read $a_{ji}$.}
$$
In other words, $I_{\mcS}$ consists of vertices $r \in I_\ns$ such that, when $s \in I_\ns$ is different from $r$, either there is no edge between $r$ and $s$ or there is an arrow of multiplicity two from $r$ to $s$ (when $\mfg$ is simple, the latter possibility happens only for the type CI case corresponding to $\mfk^\C \cong \mathfrak{gl}_l(\C) \subset \mathfrak{sp}_{2l}(\C) \cong \mfg$). Then we define
\[
\mcS_{\mbq} = \{(s_r(\mbq))_{r\in I\setminus X} \mid s_r(\mbq) \in \C(\mbq^{1/d}) \textrm{ and } s_r(\mbq) = 0 \textrm{ for }r\notin I_{\mcS}\}.
\]

We will also use the shorthand notation
\[
\mcT_{\mbq} = \mcC_{\mbq} \times \mcS_{\mbq} \subseteq \C(\mbq^{1/d})^{I\setminus X}\times \C(\mbq^{1/d})^{I\setminus X},
\]
and write $\mbt = (\mbc,\mbs)$ for a typical element in $\mcT_{\mbq}$.

Let $\mfg_X$ be the Lie algebra generated by the elements $e_r,f_r,h_r$ with $r\in X$, and let $U_\mbq(\mfg_X)$ be the $\C(\mbq^{1/d})$-algebra generated by the elements $E_r,F_r,K_r$ with $r\in X$. Moreover, let $U_\mbq(\mfh^{\theta})$ denote the algebra generated by the elements $K_{\omega}$ with $\Theta(\omega) = \omega$.

\begin{Def}
For $\mbt = (\mbc,\mbs) \in \mcT_{\mbq}$, we define $U_\mbq^{\mbt}(\mfk^{\C}) = B_{\mbc,\mbs}$ to be the algebra generated by $U_\mbq(\mfg_X)$, $U_\mbq(\mfh^{\theta})$ and the elements
\[
B_r = F_r + c_r \theta_\mbq(F_r K_r)K_r^{-1} + s_r K_r^{-1},\qquad r\in I\setminus X.
\]
\end{Def}

From \cite{MR3269184}*{Proposition 5.2}, it follows that $U_{\mbq}^{\mbt}(\mfk^{\C})$ is a right coideal subalgebra of $U_\mbq(\mfg)$. Again, the same constructions apply verbatim to the Hopf algebra $U_q(\mfg)$ for $q\ne0$ not a root of unity, in which case $\mcC_q$, $\mcS_q$ are interpreted as subsets of $\C^{I\setminus X}$.

For $0<q<1$ and $U_q(\mfu)$ as before, it has been observed in \cite{MR1913438}*{discussion before Theorem 7.6} that $U_q^{\mbt}(\mfk^{\C})$ should always be `twisted' $*$-invariant, see also the discussion in \cite{MR1742961}*{Section 3}. For us, it will be important to know when $U_q^{\mbt}(\mfk^{\C})$ is genuinely $*$-invariant. We therefore refine the corresponding statement of \cite{MR1913438} in the theorem below. We thank  Weiqiang Wang for pointing out to us that \cite{arXiv:1610.09271}*{Proposition 4.6} considers a similar treatment of $*$-invariance. As however we are using different conventions and normalisations, we spell out some of the details. 

Choose a set $I^*$ of representatives for $\tau$-orbits in $I\setminus X$. Note that by the proof of \cite{MR3269184}*{Proposition 9.2}, we can conjugate $U_q^{\mbt}(\mfk^{\C})$ with a unitary element in the (completed) Cartan part as to have $c_r >0$ when $\tau(r)=r$ or $r\in I_{\mcC}$ or $r\notin I^*$, without altering the invariance under $*$. We will hence assume this as an extra condition in the following.\footnote{Although the claim for $r\notin I^*$ is not substantiated in \cite{MR3269184}*{Proposition 9.2} a proof is easily provided, cf.~\cite{MR1913438}*{Section 7, Variation 1}.}

\begin{Theorem}\label{TheoStarInv}
The algebra $U_q^{\mbt}(\mfk^{\C})$ is $*$-invariant if $s_r \in i\R$ and
\begin{equation}\label{EqC}
c_{\tau(r)}c_r = q^{(\Theta(\alpha_r)-\alpha_{r},\alpha_{\tau(r)})}
\end{equation}
for all $r\in I\setminus X$.
\end{Theorem}

We will prove this theorem in Appendix \ref{Ap2}. Note that this implies that $c_r>0$ for all $r$ under the above convention on $c_r$.

In the following, we will write
\[
\mcC_{q,c} = \{\mbc \in \mcC_q \mid c_r >0, c_{\tau(r)}c_r = q^{(\Theta(\alpha_r)-\alpha_{\tau(r)},\alpha_r)}\},\quad \mcS_{q,c} = \{\mbs\in \mcS_q \mid s_r \in i\R\},\qquad  \mcT_{q,c}= \mcC_{q,c}\times \mcS_{q,c}.
\]
We then write $U_q^{\mbt}(\mfk)$ for the corresponding $*$-invariant coideal subalgebra of $U_q(\mfu)$. When $s_r = 0$ and $0 < c_r = c_{\tau(r)}$ satisfies \eqref{EqC} for all $r$ (fixing all $c_r$), we will refer to this as the \emph{no-parameter-case}, and write $\mbt = (\mbc,\mbs) = 0$ and
\[
U_q^{\theta}(\mfk) = U_q^0(\mfk).
\]

\begin{Def}
By a finite \emph{admissible} $*$-representation of $U_q^{\mbt}(\mfk)$ we mean any $*$-representation $(V,\pi)$ of $U_q^{\mbt}(\mfk)$ on a finite-dimensional pre-Hilbert space $V$ such that the operators $\pi(K_{\omega})$ for $K_\omega \in U_q(\mfh^{\theta})$ have positive spectrum.
\end{Def}

The following is now obvious.

\begin{Theorem}
Let $\Rep_q(\mfu)$ be the tensor C$^*$-category of finite admissible  $*$-representations of $U_q(\mfu)$. Let $ \Rep_{q}^{\mbt}(\mfk)$ be the C$^*$-category of finite admissible $*$-representations of $U_q^{\mbt}(\mfk)$. Then $\Rep_{q}^{\mbt}(\mfk)$ is a (strict) module C$^*$-category over $\Rep_q(\mfu)$ by means of \[\pi \odot \pi' = (\pi\otimes \pi')\circ \Delta_{\mid U_q^{\mbt}(\mfk)},\qquad f\odot g = f\otimes g.\]
We write $\Rep_q^{\theta}(\mfk) = \Rep_q^0(\mfk)$ for the no-parameter case.
\end{Theorem}

As the next theorem shows, the value of $\mbt$ is, in fact, irrelevant.

\begin{Theorem}
The $\Rep_q(\mfu)$-module C$^*$-categories $\Rep_q^{\mbt}(\mfk)$ are all equivalent.
\end{Theorem}
\begin{proof}
We will show in Appendix \ref{SecCharCoid} that for $\mbt,\mbt' \in \mcT_{q,c}$ there exist $^*$-isomorphisms
\[
\pi_{\mbt,\mbt'}\colon U_q^{\mbt}(\mfk) \rightarrow U_q^{\mbt'}(\mfk)
\text{ such that }
(\pi_{\mbt,\mbt'}\otimes \id)\Delta = \Delta \pi_{\mbt,\mbt'}
\]
on $U_q^{\mbt}(\mfk)$, and such that $\pi_{\mbt,\mbt'}(K_{\omega})$ is a \emph{positive} scalar multiple of $K_{\omega}$ for each $K_{\omega} \in U_q(\mfh^{\theta})$. This clearly implies the theorem.
\end{proof}

Hence in the following we can restrict to the no-parameter case. In order to turn $\Rep_q^{\theta}(\mfk)$ into a ribbon (twist-)braided module category, we will make use of the results from \citelist{\cite{arXiv:1507.06276}\cite{arXiv:1705.04238}}. We again use freely the notation used there, which is compatible with the notation from \cite{MR3269184}.

There is a slight adaption to be made as the results of \cite{arXiv:1705.04238} require some extra assumptions on~$\mbc,\mbs$ which are incompatible with the assumptions needed for our setting. To resolve this, let us for the moment return to the case of a formal variable $\mbq$. In \cite{arXiv:1705.04238}*{(3.13), (3.15) and (3.20)}, it is assumed that $\mbc\in \mcC_{\mbq}$ and $\mbs\in \mcS_{\mbq}$ satisfy the extra relations
$$
c_{\tau(r)}(\mbq) = \mbq^{(\alpha_r,\Theta(\alpha_r)-2\rho_X)}c_r(\mbq^{-1}),\quad c_{\tau(r)}(\mbq) = c_{\tau_0(r)}(\mbq),\quad s_r(\mbq) = s_r(\mbq^{-1}),
$$
where $\tau_0$ is the diagram automorphism (nontrivial only for the A$_n$, D$_{2n+1}$, and E$_6$ types) characterized by $\alpha_{\tau_0(r)} = -w_0\alpha_r$  with $w_0$ the longest element in the Weyl group of $\mfg$. As a case-by-case investigation of the Satake diagrams shows, the set $X$ is always preserved by $\tau_0$, and $\tau_0$ and $\tau$ commute \cite{arXiv:1507.06276}*{Remark 7.2}. Upon adding roots of $\mbq$ if necessary, it follows that the special solution
\[
c_{r}'(\mbq) = \mbq^{\frac{1}{2}(\alpha_r,\Theta(\alpha_r)-2\rho_X)},\qquad s_r'(\mbq) = 0,\qquad \mbt' = (\mbc',\mbs')
\]
always satisfies the extra relations. Let now $\omega_0 \in \mfh^*$ be such that
\[
(\omega_0,\alpha_r) = 0,\quad r\in X,\qquad (\omega_0,\alpha_r) = \frac{1}{4}(\Theta(\alpha_{\tau(r)}) - \alpha_{\tau(r)} -\Theta(\alpha_r) +2\rho_X,\alpha_r),\quad r\in I\setminus X.
\]
Since $\tau\omega_0$ satisfies the same conditions, $\omega_0$ is $\tau$-invariant. Moreover, since $\omega_0$ vanishes on the roots in $X$, it follows that $\Theta(\omega_0) = -\tau(\omega_0) = -\omega_0$, so
\[
(\omega_0,\alpha_r -\Theta(\alpha_r)) = 2(\omega_0,\alpha_r) =  \frac{1}{2}(\Theta(\alpha_{\tau(r)}) - \alpha_{\tau(r)} -\Theta(\alpha_r) +2\rho_X,\alpha_r)\quad r\in I\setminus X.
\]
Let $\gamma$ be the Hopf algebra automorphism $\Ad(K_{\omega_0})$ of $U_\mbq(\mfg)$, which is characterized by
\begin{align*}
E_r &\mapsto \mbq^{(\omega_0,\alpha_r)}E_r, &
F_r &\mapsto \mbq^{-(\omega_0,\alpha_r)}F_r,&
K_{\omega} &\mapsto K_{\omega}.
\end{align*}
Then, using
\[
\Ad(K_{\omega_0})(F_r + c_r'\theta_\mbq(F_rK_r)K_r^{-1}) = \mbq^{-(\omega_0,\alpha_r)}(F_r + \mbq^{(\omega_0,\alpha_r - \Theta(\alpha_r))}c_r'\theta_\mbq(F_rK_r)K_r^{-1}),
\]
we obtain
\[
\gamma(U_\mbq^{\mbt'}(\mfk^{\C})) = U_\mbq^{\theta}(\mfk^{\C}),
\]
where we interpret our non-parameter case in the setting of the indeterminate variable $\mbq$.

Denote by $\msU_\mbq(\mfg) = \bigoplus_{\varpi \in P^+} \End_{\C(\mbq)}(V_\varpi^{\mbq})$ the discrete multiplier bi-algebra over $\C(\mbq)$ associated to $U_\mbq(\mfg)$, so in particular $U_\mbq(\mfg) \subseteq M(\msU_\mbq(\mfg)) = \prod_{\varpi \in P^+} \End_{\C(\mbq)}(V_\varpi^{\mbq})$. Let
\[
\msK' \in U_\mbq^{\mbt'}(\mfk^{\C})\hotimes U_\mbq(\mfg)
\]
be the universal $K$-matrix for $U_q^{\mbt'}(\mfk^{\C})$ \citelist{\cite{arXiv:1507.06276}\cite{arXiv:1705.04238}*{Theorem 3.11}}, and write
\begin{equation}
\label{EqK1}
\wmsK = (\gamma\otimes \gamma)(\msK') \in U_\mbq^{\theta}(\mfk^{\C})\hotimes U_\mbq(\mfg)
\end{equation}
Then by $\tau\tau_0$-invariance of $\omega_0$,
\begin{equation}\label{EqK2}
\wmsK \Delta(b) = (\id\otimes \tau\tau_0)(\Delta(b))\wmsK,\qquad b\in U_\mbq^{\theta}(\mfk^{\C}).
\end{equation}
Moreover, since the $R$-matrix $\msR \in U_\mbq(\mfg)\hotimes U_\mbq(\mfg)$ is invariant under $\gamma\otimes \gamma$, we also retain
\begin{equation}\label{EqK3}
(\Delta\otimes \id)(\wmsK) = (\id \otimes \tau\tau_0 \otimes \id)(\msR_{32}) \wmsK_{13} \msR_{23},\qquad (\id\otimes \Delta)(\wmsK) = \msR_{32} \wmsK_{13} (\id \otimes \tau \tau_0 \otimes \id)(\msR_{23}) \wmsK_{12}.
\end{equation}

All of the above can be performed as well in the setting of $0<q<1$ a real number, and in particular \eqref{EqK1},\eqref{EqK2} and \eqref{EqK3} continue to hold. (For example, all constructions, including those involving the bar involution, can be performed over the $\C$-algebra $\C[\mbq]/(\mbq^2 -(q+q^{-1})\mbq + 1)$, from which one can specialize to $\mbq = q$.) Let
\[
\msK = (\id\otimes \tau\tau_0)(\wmsK) \in U_q^{\theta}(\mfk) \hotimes U_q(\mfu).
\]
Then \eqref{EqK2} becomes
\[
\msK((\id\otimes \tau\tau_0)(\Delta(b)) = \Delta(b) \msK,\qquad \forall b\in U_q^{\theta}(\mfk).
\]
We also have $(\tau\tau_0\otimes \tau\tau_0)(\msR) = \msR$, since any automorphism of the Dynkin diagram will leave the universal $R$-matrix invariant, see \cite{MR1358358}*{Corollary 8.3.12} (or, just to check this equality in the multiplier algebra one may use the fact that $\msR$ is determined by its action on the tensor product of lowest and highest weight vectors). It follows that
\[
(\Delta\otimes \id)(\msK) = \msR_{32} \msK_{13} ((\id\otimes \tau\tau_0)(\msR))_{23},\qquad (\id\otimes \Delta)(\msK) = (\id\otimes \Delta)(\msK)\msK_{12}.
\]
We hence obtain the following theorem.

\begin{Theorem}
With respect to the braiding
\[
\eta_{X,U} = (\pi_X\otimes \pi_U)(\msK),
\]
$\Rep_{q}^{\theta}(\mfk)$  is a (strict) ribbon $\tau\tau_0$-braided module C$^*$-category over the braided tensor C$^*$-category $\Rep_q(\mfu)$, where the latter is equipped with the braiding coming from the universal $R$-matrix $\msR$.
\end{Theorem}

\subsection{Ribbon braided module \texorpdfstring{C$^*$}{C*}-categories from quantum group automorphisms}\label{sec:Vogan}

Fix again a compact semisimple Lie algebra $\mfu$ with Chevalley generators $\{e_r,f_r,h_r\mid r\in I\}$ in $\mfg = \mfu^{\C}$, and let $\mfh = \mft \oplus \mfa$ be the complex Cartan subalgebra with vector part  $\mfa = \oplus_r \R h_r$ and compact part $\mft = i\mfa$.

Let $(Y,\mu)$ be a Vogan diagram with associated involution $\nu = \nu(Y,\mu)$. For $0<q<1$, we then obtain an involutive Hopf $*$-algebra automorphism $\nu_q$ of $U_q(\mfu)$ by
\begin{align*}
\nu_q(E_r) &= \epsilon_rE_{\mu(r)},& \nu_q(F_r) &= \epsilon_rF_{\mu(r)},& \nu_q(K_{\alpha}) &= K_{\Nu(\alpha)},
\end{align*}
where $N$ denotes the involutive transformation dual to $\nu_{\mid \mfh}$.
Note that $\nu_q$ also preserves the Borel subalgebras~$U_q(\mfb^{\pm})$. In the following, it will be convenient to work for a while with versions of $U_q(\mfu)$ and~$U_q(\mfg)$  with an extra Cartan part added.

Consider independent copies of the Hopf algebras $U_q(\mfb^{\pm})=(U_q(\mfb^{\pm}),\Delta,\varepsilon,S)$ with generators now denoted by $X_r^{\pm}, L_{\omega}^{\pm}$, so in particular
\begin{align*}
L_{\omega}^{\pm} X_r^{\pm} &= q^{\pm (\omega,\alpha_r)}X_r^{\pm} L_{\omega}^{\pm},& \Delta(X_r^+) &= X_r^+\otimes 1 + L_r^+ \otimes X_r^+,& \Delta(X_r^-) &= X_r^- \otimes (L_r^{-})^{-1} + 1\otimes X_r^-.
\end{align*}
Consider the skew-pairing between $U_q(\mfb^-)$ and $U_q(\mfb^+)$ characterized by
 \begin{align*}
(X_r^-,X_s^+) &= \delta_{rs} (q_r^{-1}-q_r)^{-1},& (L_{\alpha}^-,L_{\beta}^+) &= q^{-(\alpha,\beta)},
\end{align*}
and being zero between the other generators. Here the skew-pairing property means we have
\begin{align*}
(XY,Z) &= (X\otimes Y,\Delta(Z)),& (X,YZ) &= (\Delta^{\op}(X),Y\otimes Z).
\end{align*}
We can then make a new skew-pairing by twisting with $\nu_q$,
\[
(X,Y)_{\nu} = (\nu_q(X),Y).
\]
Write $(X, Y)_0 = (X, Y)$ and $(X, Y)_1 = (X, Y)_{\nu}$. Define for $k,l\in \{0,1\}$ the unital algebra $U_q^{kl}(\mfg\oplus \mfh)$ generated by the algebras $U_q(\mfb^{\pm})$ with interchange relation
\begin{equation}
\label{eq:interchange-relation}
YX = (X_{(1)},Y_{(1)})_k X_{(2)}Y_{(2)} (S(X_{(3)}),Y_{(3)})_l,\qquad Y\in U_q(\mfb^+), X \in U_q(\mfb^-).
\end{equation}
Then it is easily verified that we have as universal relations between the generators those of $U_q(\mfb^{\pm})$ together with
\begin{gather*}
L_{\omega}^+ L_{\chi}^- = q^{(\omega,\Nu^l(\chi)-\Nu^k(\chi))}L_{\chi}^-L_{\omega}^+,\\
L_{\omega}^- X_r^+ = q^{(\alpha_r,\Nu^k(\omega))}X_r^+L_{\omega}^-,\qquad L_{\omega}^+X_r^- =  q^{-(\alpha_r,\Nu^l(\omega))} X_r^- L_{\omega}^+,\\
\lbrack X_r^+,X_s^-\rbrack = \frac{\delta_{r,\mu^l(s)} \epsilon_s^l L_{r}^+ - \delta_{r,\mu^k(s)} \epsilon_s^{k} (L_{s}^-)^{-1}}{q_r-q_r^{-1}},
\end{gather*}
and that moreover there are no `extra relations' in the sense that the multiplication map gives isomorphisms
\[
U_q(\mfb^+) \otimes U_q(\mfb^-) \cong U_q^{kl}(\mfg\oplus \mfh) \cong U_q(\mfb^-) \otimes U_q(\mfb^+).
\]
We can endow $U_q^{kl}(\mfg\oplus \mfh)$ with the unique $*$-structure such that
\begin{align*}
(X_r^+)^* &= X_r^-L_r^-,& (L_r^+)^* &= L_r^-,
\end{align*}
and we denote this $*$-algebra by $U_q^{kl}(\mfu\oplus \mfa)$. The $*$-algebra $U_{q}^{00}(\mfu\oplus \mfa)$ is known to be a cover of the Hopf $*$-algebra $U_q(\mfu)$ under
\begin{align*}
X_r^+ &\mapsto E_r,&
X_r^- &\mapsto F_r,&
L_{\alpha}^{\pm} &\mapsto K_{\alpha}.
\end{align*}
It is also immediate from the construction that the family $(U_q^{kl}(\mfu\oplus \mfa))_{k,l}$ form a Hopf--Galois system in the sense of \cite{MR3285340}*{Definition 2.4} with $*$-preserving comultiplications
\[
\Delta = \Delta_{kl}^m\colon U_q^{kl}(\mfu\oplus \mfa)\rightarrow U_q^{km}(\mfu\oplus \mfa)\otimes U_q^{ml}(\mfu\oplus \mfa)
\]
given by the usual comultiplication on $U_q(\mfb^{\pm})$, and with the antipodes
\[
S = S^{kl}\colon U_q^{kl}(\mfu\oplus \mfa) \rightarrow U_q^{lk}(\mfu\oplus \mfa)
\]
determined by the ordinary antipodes on $U_q(\mfb^{\pm})$. In the following, we denote by
\[
\walpha\colon U_q^{10}(\mfu\oplus \mfa) \rightarrow U_q^{10}(\mfu\oplus \mfa)\otimes U_q(\mfu)
\]
the composition of $\Delta_{10}^0$ with the projection $U_q^{00}(\mfu\oplus \mfa) \rightarrow U_q(\mfu)$.

Let $U_q^{10}(\mfureg\oplus \mfa)$ be the $*$-subalgebra of $U_q^{10}(\mfu\oplus \mfa)$ generated by
\[
X_r^-,\quad (X_r^-)^* \quad (r \in I), \qquad L_{\omega}^{\pm} \quad (\omega \in P, N(\omega)=\omega),\qquad L_{\omega}^+ L_{-\omega}^- \quad (\omega\in P).
\]
Then $U_q^{10}(\mfureg\oplus \mfa)$ is precisely the $*$-algebra of elements commuting with all $L_{\omega}^+L_{-\omega}^-$, and is stable under the coaction $\walpha$ of $U_q(\mfu)$. Moreover, the ensuing coaction of $U_q(\mfu)$ has $L_{\omega}^+L_{-\omega}^-$ in its space of coinvariants. Let us write $U_q^{10}(\mfureg)$ for the $*$-algebra obtained by taking the quotient of $U_q^{10}(\mfureg\oplus \mfa)$ with the ideal generated by the $1-L_{\omega}^+L_{-\omega}^-$. Then we obtain a coaction $\alpha$ of $U_q(\mfu)$ on $U_q^{10}(\mfureg)$. It is, in fact, not hard to verify that $U_q^{10}(\mfureg)$ is the universal $*$-algebra generated (as a $*$-algebra) by the elements $F_r$ for $r\in I$ and the selfadjoint elements $K_{\omega}$ for $\alpha \in P^{N}$ (obtained as images of $X_r^-$ and $L_{\omega}^+$ respectively) such that $K_{\omega},F_r$ satisfy the relations of $U_q(\mfb^-)$ and such that
\[
F_r^*F_s -q^{(\alpha_r,\alpha_{s})}F_sF_r^* = \frac{\delta_{r,s}-\delta_{r,\mu(s)}\epsilon_s q^{-\frac{1}{2}(\alpha_r,\alpha_s-\alpha_{\mu(r)})} K_{\alpha_{\mu(r)}+\alpha_r}^{-1}}{q_r-q_r^{-1}},
\]
the coaction $\alpha$ by $U_q(\mfu)$ being given by the ordinary formulas. Note that if $\mfk\subseteq \mfu$ is of equal rank (i.e., $\nu$ acts trivially on $\mfh$), we can identify $U_q^{10}(\mfureg)$ with $U_q(\mfu)$ as a vector space by the natural triangular decomposition. We will denote the $*$-algebra generated by the elements $K_{\omega}$ with $\Nu(\omega)= \omega$ by $U_q(\mft^{\nu})$.

Next let us review the representation theory of $U_q^{10}(\mfureg)$, cf.~\cite{MR3208147}.

\begin{Def}
We call (resp.~\emph{finite}) \emph{admissible} $U_q^{10}(\mfureg)$-module any (resp.~finitely generated) left $U_q^{10}(\mfureg)$-module for which $U_q(\mfn^-)$ acts locally finitely and for which the $K_{\omega}\in U_q(\mft^{\nu})$ act as semisimple transformations with positive eigenvalues. We call (resp.~\emph{finite}) \emph{admissible $*$-representation} any (resp.~finite) admissible module $M$ endowed with a pre-Hilbert space structure for which
$$
\langle v, X w \rangle = \langle X^*v, w \rangle,\qquad (v,w \in M, X \in U_q^{10}(\mfureg)).
$$
\end{Def}

Note that in general, the finite admissible $*$-representations will \emph{not} be on finite-dimensional Hilbert spaces, even when they are irreducible.

For $\omega \in \mfa^{\nu}$, we call a vector $\xi$ in a $U_q^{10}(\mfureg)$-module a weight vector of weight $\omega$ if
\[
K_{\chi}\xi = q^{(\omega,\chi)}\xi.
\]
We call a vector $\eta$ in a $U_q^{10}(\mfureg)$-module a \emph{lowest weight} vector (of weight $\omega$)  if it is a weight vector (of weight~$\omega$) vanishing under the action of the $F_r$. We call a $U_q^{10}(\mfureg)$-module a \emph{lowest weight}-module (of weight~$\omega$) if it is generated by a lowest weight vector (of weight~$\omega$).

\begin{Lem}\label{LemIrLow}
Any finite admissible $*$-representation of $U_q^{10}(\mfureg)$ is a finite direct sum of irreducible admissible $*$-representations. Moreover, an admissible $*$-representation is lowest weight if and only if it is irreducible, and in this case the space of its self-intertwiners consists of the scalars.
\end{Lem}

\begin{proof}
In a finite admissible $*$-representation the eigenspaces of $K_{2\rho}$, where $\rho$ is the ($N$-invariant) half-sum of the positive roots, must be finite-dimensional. Since any subrepresentation of a finite admissible $*$-representation is again finite admissible, any admissible $*$-representation is a direct sum of a finite number of irreducible admissible $*$-representations as we can take orthogonal complements in each $K_{2\rho}$-eigenspace.

If now $M$ is admissible and irreducible, we can obtain a lowest weight vector by repeatedly applying $F_r$'s to a weight vector. By irreducibility, the module spanned by the lowest weight vector must be~$M$. Conversely, if $M$ is admissible and of lowest weight $\omega_0$, the space of $\omega_0$-weight vectors must be one-dimensional and can hence be contained in only one irreducible component.

The final statement is immediate from Schur's lemma.
\end{proof}

\begin{Lem}
A lowest weight $*$-representation $M$ is completely determined up to isomorphism by its lowest weight.
\end{Lem}

\begin{proof}
Let $\omega_0$ be the lowest weight. By inducing from a one-dimensional module for the Borel part, we can turn $U_q(\mfn^-)^*$ into a $U_q^{10}(\mfureg)$-module with $1$ as a lowest weight vector at weight $\omega_0$. This module has an invariant sesqui-linear form by the triangular decomposition. By multiplying with a scalar if necessary, the natural map  $U_q(\mfn^-)^* \rightarrow M$ must then preserve this form, again by the triangular decomposition. Hence the kernel of the above map must lie in the kernel of the form on $U_q(\mfn^-)^*$, and must then be equal to this kernel. This completely determines $M$.
\end{proof}

In the following, we say that a $\nu$-invariant weight $\omega_0$ is \emph{$U_q^{10}(\mfureg)$-adapted} if it arises as the lowest weight in a lowest weight admissible $*$-representation, and we write the associated lowest weight $*$-representation as $M_{\omega_0}$. The modules $M_{\omega_0}$, with $\omega_0$ ranging over the $U_q^{10}(\mfureg)$-adapted weights, then form a maximal collection of pairwise non-isomorphic irreducible admissible $U_q^{10}(\mfureg)$-representations.

\begin{Def}
We write $\Rep_q^{\nu}(\mfu^{\nu})$ for the C$^*$-category of finite admissible $*$-representations of $U_q^{10}(\mfureg)$.
\end{Def}

From the above discussion we see that the morphism spaces between finite admissible $*$-representations are finite-dimensional, and consist of adjointable maps. We now have the following result.

\begin{Theorem}
If $M$ is a finite admissible $*$-representation of $U_q^{10}(\mfureg)$, and $V$ a finite admissible representation of $U_q(\mfu)$, then $M\otimes V$ with the representation $(\pi_M\otimes \pi_V)\circ \alpha$ is a finite admissible $*$-representation of $U_q^{10}(\mfureg)$, making $\Rep_q^{\nu}(\mfu^{\nu})$ a module C$^*$-category over $\Rep_q(\mfu)$.
\end{Theorem}

We now show that $\Rep_q^{\nu}(\mfu^{\nu})$ is a ribbon $\nu_q$-braided module C$^*$-category, where $\nu_q$ also denotes the (strict) braided autoequivalence of $\Rep_q(\mfu)$ induced by the involution $\nu_q$ on $U_q(\mfu)$.

Let again
\[
\msR \in U_q(\mfu)\hotimes U_q(\mfu) = \prod_{\varpi,\chi} B(V_{\varpi})\otimes B(V_{\chi})
\]
be the universal $R$-matrix of $U_q(\mfu)$. We have $(\nu_q \otimes \nu_q)(\msR) = \msR$, as can be seen from a concrete formula which contains $E_r$ and $F_r$ in pairs and is stable under diagram automorphism~\cite{MR1358358}*{Section 8.3}. Recall also that it can be characterized by a pairing of $U_q(\mfb^-)$ and $U_q(\mfb^+)$, in the sense that for $Y\in U_q(\mfb^-)$, when $\omega$ is a functional on $U_q(\mfu)$ factoring through a finite admissible $*$-representation, one has
\[
Z = (\id\otimes \omega)(\msR)\in U_q(\mfb^+)\qquad \textrm{and}\qquad (Y,Z) = \omega(Y).
\]
Similarly, if $Z\in U_q(\mfb^+)$ we have then
\[
Y = (\omega\otimes \id)(\msR)\in U_q(\mfb^-)\qquad \textrm{and}\qquad (Y,Z) = \omega(Z).
\]
The first assertions in the above two lines will be written as
\[
\msR \in U_q(\mfb^+)\hotimes U_q(\mfb^-).
\]
We can then interpret $\msR$ as an element
\[
\wmsR \in U_q^{10}(\mfu\oplus \mfa)\hotimes U_q(\mfb^-) \subseteq U_q^{10}(\mfu\oplus \mfa)\hotimes U_q(\mfu).
\]
Similarly, $\msR_{21}$ can be interpreted as an element $\wmsR_{21}$ in $U_q^{10}(\mfu\oplus \mfa)\hotimes U_q(\mfb^+)$ (strictly speaking we should write $\widetilde{\msR_{21}}$, which is however a bit more awkward notation).

\begin{Lem}\label{LemComR}
For $X \in U_q(\mfb^+)$ and $Y\in U_q(\mfb^-)$, we have the following identities in $U_q^{10}(\mfu\oplus \mfa)\hotimes U_q(\mfu)$:
\begin{align*}
\wmsR\Delta(X) &= \Delta^{\op}(X)\wmsR,&
\wmsR\Delta(Y) &= (\id\otimes \nu_q)\Delta^{\op}(Y)\wmsR,\\
\wmsR_{21}((\id\otimes \nu_q)\Delta^{\op}(X)) &= \Delta(X)\wmsR_{21},&
\wmsR_{21}\Delta^{\op}(Y) &= \Delta(Y) \wmsR_{21}.
\end{align*}
\end{Lem}

\begin{proof}
Since the first legs of $\Delta(X)$ and $\wmsR$ both lie in $U_q(\mfb^+)$, the first identity can be interpreted inside $U_q(\mfb^+) \hotimes U_q(\mfg)$, and thus holds from the ordinary $R$-matrix relations. For the second identity, we take a functional $\omega$ on $U_q(\mfu)$ factorizing through a finite admissible $*$-representation. Then, we have
\[
(\id\otimes \omega)(\wmsR) Y = Y_{(2)} (\id\otimes \omega)((1\otimes\nu_q(Y_{(1)}))\wmsR(1\otimes S(Y_{(3)}))),
\]
from \eqref{eq:interchange-relation} and the fact that $(Y,(\id\otimes \omega)(\wmsR)) = \omega(Y)$, hence
\[
\wmsR \Delta(Y) = (Y_{(2)} \otimes \nu_q(Y_{(1)}))\wmsR(1\otimes S(Y_{(3)})Y_{(4)}) = (\id\otimes \nu_q)(\Delta^{\op}(Y))\wmsR.
\]

The remaining two identities are obtained in a similar way.
\end{proof}

Let $v = e^{\pi i \hbar C}$ be the natural ribbon element, where $C$ is the quantum Casimir element in the center of $M(\msU_q(\mfu))$ acting on $V_\varpi^q$ by the same scalar as the classical Casimir element of $\mfu$ on $V_\varpi$. Recall that
\[
\msR_{21}\msR\Delta(v) = v\otimes v.
\]

\begin{Def}
We define
\[
\wmsE = \wmsR_{21}(\id\otimes \nu_q)(\wmsR)(1\otimes v^{-1}) \in U_q^{10}(\mfu\oplus \mfa) \hotimes U_q(\mfu).
\]
\end{Def}

\begin{Prop}
The element $\wmsE$ satisfies the following properties:
\begin{gather*}
\wmsE(\id\otimes \nu_q)(\widetilde{\alpha}(x)) = \widetilde{\alpha}(x)\wmsE\ \ \text{for all}\ \ x\in U_q^{10}(\mfu\oplus \mfa),\\
\begin{aligned}
(\widetilde{\alpha} \otimes \id)(\wmsE) &= \msR_{32}\wmsE_{13}(\id\otimes \nu_q)(\msR)_{23},&
(\id\otimes \Delta)\wmsE &= (\widetilde{\alpha}\otimes \id)(\wmsE)\wmsE_{12}.
\end{aligned}
\end{gather*}
\end{Prop}

\begin{proof}
The first identity follows from Lemma \ref{LemComR} and that $\nu_q$ is involutive, by considering separately elements in $U_q(\mfb^+)$ and $U_q(\mfb^-)$. The second identity is also immediate from the properties of the $R$-matrix.

Writing out the last equality and using the defining property of $v$, we see that it is equivalent to
\begin{equation}\label{EqIdRibb}
(\id\otimes \Delta)(\wmsR_{21}(\id\otimes \nu_q)(\wmsR))\msR_{32}\msR_{23} = \msR_{32} \wmsR_{31}(\id\otimes \nu_q)(\wmsR)_{13}(\id\otimes \nu_q)(\msR)_{23}\wmsR_{21}(\id\otimes \nu_q)(\wmsR)_{12}.
\end{equation}
Noting
\begin{equation*}
\begin{split}
(\id\otimes \Delta)(\wmsR_{21}(\id\otimes \nu_q)(\wmsR))\msR_{32}\msR_{23}  &=  \msR_{32}(\id\otimes \Delta^{\op})(\wmsR_{21}(\id\otimes \nu_q)(\wmsR))\msR_{23}  \\
&= \msR_{32} \wmsR_{31}\wmsR_{21} (\id\otimes \Delta^{\op})((\id\otimes \nu_q)(\wmsR))\msR_{23} \\
&= \msR_{32} \wmsR_{31}\wmsR_{21} \msR_{23}(\id\otimes \Delta)((\id\otimes \nu_q)(\wmsR))\\
&= \msR_{32} \wmsR_{31}\wmsR_{21} \msR_{23}(\id\otimes \nu_q)(\wmsR)_{13}(\id\otimes \nu_q)(\wmsR)_{12},
\end{split}
\end{equation*}
\eqref{EqIdRibb} will follow from
\[
(\id\otimes \nu_q)(\wmsR)_{13}(\id\otimes \nu_q)(\msR)_{23}\wmsR_{21} = \wmsR_{21} \msR_{23}(\id\otimes \nu_q)(\wmsR)_{13}.
\]
This is a consequence of the commutation relations in Lemma \ref{LemComR}:
\begin{equation*}
\begin{split}
(\id\otimes \nu_q)(\wmsR)_{13}(\id\otimes \nu_q)(\msR)_{23}\wmsR_{21} &= (\Delta\otimes \id)((\id\otimes \nu_q)\wmsR) \wmsR_{21} \\
& = \wmsR_{21} (\id\otimes \nu_q \otimes \nu_q)(\Delta^{\op}\otimes \id)(\wmsR) \\
&= \wmsR_{21} \msR_{23}(\id\otimes \nu_q)(\wmsR)_{13},
\end{split}
\end{equation*}
where in the last step we used $(\nu_q\otimes \nu_q)(\msR) = \msR$, as the form $(-,-)$ is $\nu_q$-invariant.
\end{proof}

It follows in particular from the first identity above that $\wmsE$ commutes with the $L_{\omega}^+L_{-\omega}^-\otimes 1$, hence
\[
\wmsE \in U_q^{10}(\mfureg\oplus \mfa) \hotimes U_q(\mfu).
\]
We write $\msE$ for the projection of $\wmsE$ inside $U_q^{10}(\mfureg)\hotimes U_q(\mfu)$. The following theorem is now
immediate.

\begin{Theorem}
The $\Rep_q(\mfu)$-module C$^*$-category $\Rep_q^{\nu}(\mfu^{\nu})$ is ribbon $\nu_q$-braided with respect to the $\nu_q$-braiding
\[
\eta_{M,U}(\xi\otimes \eta) = \msE(\xi\otimes \eta),\qquad \xi \in M,\eta \in U,
\]
where $M \in \Rep_q^{\nu}(\mfu^{\nu})$ and $U \in \Rep_q(\mfu)$.
\end{Theorem}

\section{Statement of the main conjecture}
\label{sec:main-conj}

We conjecture that the three ribbon twist-braided $\Rep_q(\mfu)$-module C$^*$-categories constructed in the previous section are equivalent if the associated involutions of $\mfu$ are inner equivalent.

\begin{Conj}\label{ConjSat}
Let $\mfu$ be a compact semisimple Lie algebra and $\{e_r,f_r,h_r\mid r\in I\}$ be $*$-compatible Chevalley generators of $\mfg=\mfu^\C$. Let
\begin{itemize}
\item $\sigma$ be an involution of $\mfu$ and $\mfk_{\sigma} = \mfu^{\sigma}$,
\item $(X,\tau)$ be a Satake diagram with associated involution $\theta = \theta(X,\tau)$ and $\mfk_{\theta} = \mfu^{\theta}$.
\item $(Y,\mu)$ be a Vogan diagram with associated involution $\nu = \nu(Y,\mu)$ and $\mfk_{\nu} = \mfu^{\nu}$.
\end{itemize}

Let $0<q<1$, and let $\hbar\in i\R$ be such that $e^{\pi i \hbar} = q$.

Then if $\sigma,\theta$ and $\nu$ are inner equivalent, we have
\[
\Rep_{\hbar}(\mfk_{\sigma})\cong  \Rep_q^{\theta}(\mfk_{\theta}) \cong \Rep_q^{\nu}(\mfk_{\nu})
\]
as twist-braided $\Rep_q(\mfu)$-module C$^*$-categories.
\end{Conj}

\section{The rank one case}
\label{sec:rk-one-case}

In this section we will verify Conjecture \ref{ConjSat} for the simplest case $\mfu = \mfsu_2$. We first introduce some general terminology.

\begin{Def}
Let $\mcC$ be a tensor C$^*$-category, and $\mcD$ a $\mcC$-module C$^*$-category. We call $\mcD$ \emph{connected} if for any nonzero objects $X,Y$ in $\mcD$ there exists $U$ in $\mcC$ with $\Mor(X\odot U,Y)\neq \{0\}$. For $X\in \mcD$ a simple object, we write $\mcD_X$ for the smallest C$^*$-subcategory of $\mcD$ which is full, isomorphism-closed, subobject complete and closed under finite direct sums and $\odot$.
\end{Def}

Later, when we consider $\mcD=\Rep_q^{\mbt}(\mfk)$, we will write
$\Rep_q^{\mbt}(\mfk;X)$ for $\mcD_X$.

The category $\mcD_X$ is itself a $\mcC$-module C$^*$-category, and will be called the $\mcC$-module C$^*$-category \emph{generated} by $X$. If $\mcC$ is rigid, with $\overline{U}$ a dual object of $U$, it is clear by Frobenius reciprocity
\[
\Mor(X\odot U,Y) \cong \Mor(X,Y\odot \overline{U})
\]
that each $\mcD_X$ is equal to the full subcategory consisting of all objects $Y$ with $\Mor(X\odot U,Y)\neq \{0\}$ for some $U$. Moreover, we then have $\mcD_X = \mcD_Y$ for each simple object $Y$ in $\mcD_X$. If $\mcD$ is twist-braided, this structure passes by restriction to any $\mcD_X$.

The following lemma is immediate.

\begin{Lem}\label{LemIsoComp}
Let $\mcC$ be a rigid (braided) tensor C$^*$-category, and let $\mcD,\mcD'$ be two (twist-braided) $\mcC$-module C$^*$-categories. Assume that $\{X_i\mid i\in \mcI\}$ and $\{Y_i\mid i\in \mcI\}$ are maximal collections of mutually non-isomorphic simple objects in respectively $\mcD$ and $\mcD'$. If $\mcD_{X_i} \cong \mcD'_{Y_i}$ for each $i\in \mcI$, then $\mcD\cong \mcD'$.
\end{Lem}

Let us now turn to the specific case of $\mfu = \mfsu_2$. We use the standard generators $e,f,h$ for $\mfsl_2(\C)$, \[\lbrack h,e\rbrack= 2e,\qquad \lbrack h,f\rbrack = -2f,\qquad \lbrack e,f\rbrack = h,\qquad e^* = f,\qquad h^* = h.\] In this case, there is up to inner equivalence only one non-trivial involution of $\mfsu_2$, with Satake form and Vogan form determined by \[\theta(h) = -h,\quad \theta(e) = -f,\quad \theta(f) = -e,\qquad \nu(h) = h,\quad \nu(e) = -e,\quad \nu(f) = -f.\] The respective fixed point subalgebras are
\[
\mfk_{\theta} =\mft_\theta = \R(f-e),\qquad \mfk_{\nu} = \mft_\nu = i \R h,
\]
with complexifications $\mfh_\theta = \C(f-e)$ and $\mfh_\nu = \mfh = \C h$.

Let us fix $0<q<1$ and $\hbar\in i\R$ with $q= e^{\pi i \hbar}$. Then $U_q(\mfsu_2)$ is the Hopf $*$-algebra generated by $K,E,F$ with
\begin{gather*}
\begin{aligned}
KE &= q^2EK,& KF &= q^{-2}FK,& \lbrack E,F\rbrack &= \frac{K-K^{-1}}{q-q^{-1}},& K^* &= K,& E^* &= FK,
\end{aligned}\\
\begin{aligned}
\Delta(K) &= K\otimes K,& \Delta(E) &= E\otimes 1 + K\otimes E,& \Delta(F) &= F\otimes K^{-1} + 1\otimes F.
\end{aligned}
\end{gather*}
(Strictly speaking, we should also add a square root of $K$ by the conventions for $U_q(\mfu)$ used in this paper, but this will be inessential in what follows.)

Let $V$ be the fundamental generating 2-dimensional $*$-representation of $U_q(\mfsu_2)$ on $\C^2$ with the standard orthonormal basis $e_+ = (1,0), e_- = (0,1)$ such that, with $X_V = \pi_V(X)$,
\[
E_V = \begin{pmatrix} 0 & q^{1/2} \\ 0 & 0 \end{pmatrix},\qquad F_V = \begin{pmatrix} 0 & 0 \\ q^{-1/2} & 0 \end{pmatrix},\qquad K_V = \begin{pmatrix} q & 0 \\ 0 & q^{-1}\end{pmatrix}.
\]
Then writing $e_{ab} = e_a\otimes e_b$, we have with respect to the basis $\{e_{++},e_{+-},e_{-+},e_{--}\}$ that the $R$-matrix and braiding at $V\otimes V$ are given by
\[
R=\msR_{V,V} = q^{1/2} \begin{pmatrix} q^{-1} & 0 & 0 & 0 \\ 0 & 1 & q^{-1} - q & 0 \\ 0 & 0 & 1 & 0\\ 0 & 0 & 0& q^{-1}\end{pmatrix},\quad \beta= \beta_{V,V} = \Sigma R =  q^{1/2} \begin{pmatrix} q^{-1} & 0 & 0 & 0 \\ 0 & 0 & 1 & 0 \\ 0 & 1 & q^{-1}-q & 0\\ 0 & 0 & 0& q^{-1}\end{pmatrix}.
\]

Let us also introduce the following terminology.

\begin{Def}
Let $\mcD$ be a $\Rep_q(\mfsu_2)$-module C$^*$-category. We say that $\mcD$ has $\Z$-fusion rules if the simple objects of $\mcD$ (up to isomorphism) can be labeled $\{W_n\mid n\in\Z\}$ such that
\[
W_n \otimes V \cong W_{n+1}\oplus W_{n-1}.
\]
\end{Def}

Our first goal will now be to study the braided module C$^*$-category structure on $\Rep_q^{\theta}(\mft_{\theta})$. Note that in this case we have
\[
\mcC_{q,c} = \{q^{-2}\},\qquad \mcS_{q,c} = i\R.
\]
In the following, we identify $\mcT_{q,c} = \mcC_{q,c} \times \mcS_{q,c}$ with $\R$ by
\[
(q^{-2},it) \leftrightarrow t.
\]
The coideal $U_q^t(\mft_{\theta})$ is generated by
\[
B_t = F-q^{-2}EK^{-1} + itK^{-1},\ \ \text{and then}\ \ B_{t,V} = \begin{pmatrix} itq^{-1} & -q^{-1/2} \\ q^{-1/2} & itq\end{pmatrix}.
\]
Note that $B_t^* =-B_t$. In particular, the simple objects in $\Rep_q^{\theta}(\mft_{\theta})$ are the evaluations
\[
\ev_{it}\colon U_q^0(\mft_{\theta}) \rightarrow \C,\quad B_0 \mapsto it
\]
for $t\in \R$, and we then obtain equivalences of module C$^*$-categories with fixed generating object
\[
\Rep_q^{\theta}(\mfk_{\theta};\ev_{it}) \cong \Rep_q^t(\mfk_{\theta};\varepsilon),
\]
where $\varepsilon$ is the counit of $U_q(\mfsu_2)$ (restricted to $U_q^t(\mft_{\theta})$).

For $t\in \R$, let us write $\lambda = \lambda_t \in \R$ for the unique real number such that
\[
t = q^{-1/2} \frac{q^{-\lambda} - q^{\lambda}}{q^{-1}-q}.
\]

\begin{Lem}
Let $\C_n$ be the one-dimensional $*$-representation of $U_q^{t}(\mft_{\theta})$ given by the $*$-character
\[
\chi_n\colon U_q^t(\mft_{\theta}) \rightarrow \C,\quad B_t \mapsto iq^{-1/2} \frac{q^{-\lambda -n} - q^{\lambda +n}}{q^{-1}-q}.
\]
Then the $\C_n$ exhaust up to isomorphism the irreducible objects in $\Rep_q^t(\mft_{\theta};\varepsilon)$, and
\[
\C_n \otimes V \cong \C_{n+1} \oplus \C_{n-1},
\]
that is, $\Rep_q^t(\mft_{\theta};\varepsilon)$ has $\Z$-fusion rules.
\end{Lem}
\begin{proof}
Clearly $\C_0$ corresponds to the trivial representation given by the counit. Then the fusion rules above follow from an easy spectral computation for the $2$-by-$2$-matrices $(\chi_n\otimes \pi_V)\Delta(B_t)$. Since $\C_0$ is generating, these exhaust all irreducible objects.
\end{proof}

\begin{Prop}\label{PropUnBraid}
Let $t\in \R$. Up to isomorphism, there is a unique ribbon braided $\Rep_q(\mfsu_2)$-module C$^*$-category structure $\eta$ on $\Rep_q^t(\mft_{\theta};\varepsilon)$ such that $C = \eta_{\varepsilon,V}$ is not a scalar multiple of the identity. In that case, the singular values of $C$ are $\{q^{\pm \lambda - 1/2}\}$.
\end{Prop}
\begin{proof}
From the defining relations \eqref{EqOct2} and \eqref{EqRB2} (in the case of trivial braided monoidal autoequivalence), it is clear that $\eta$ is completely determined by $C = \eta_{\C,V}$. Moreover, as $\eta$ is a natural transformation,
\[
\eta_{\varepsilon,V\otimes V} =  (C\otimes 1)\beta(C\otimes 1)\beta = \beta(C\otimes 1)\beta (C\otimes 1).
\]
It is immediate then from a direct computation, cf.~\cite{MR1644317}*{Section 5}, that, if $C$ is not diagonal, $C$ must be of the form\footnote{The specific form differs from that in \cite{MR1644317} as a different convention for the $R$-matrix is used.}
\[
C = \begin{pmatrix} a & b \\ c & 0 \end{pmatrix}
\]
with $bc = -q^{-1}$. But as $C$ must commute with $B_{t,V}$, either $C$ is a scalar multiple of the identity (which is excluded by assumption) or
\[
C = \pm \begin{pmatrix} it(q^{-1}-q) & -q^{-1/2} \\ q^{-1/2} & 0  \end{pmatrix}.
\]
Hence the singular values of $C$ are indeed $\{q^{\pm \lambda - 1/2}\}$. The sign of $C$ is irrelevant as it can be changed by means of the natural automorphism $F$ of the identity functor on $\Rep_q(\mfsu_2)$ characterized by
\[
F_{U_{n/2}} = (-1)^{n} \id_{U_{n/2}}
\]
where $U_{n/2}$ is the unique (up to isomorphism) irreducible representation of dimension $n+1$.
\end{proof}

\begin{Theorem}\label{TheoUniBraid}
Let $(\mcD,\odot,\Psi,\eta)$ be a ribbon braided $(\Rep_q(\mfsu_2),\otimes,\beta)$-module C$^*$-category. Assume that $\mcD$ has $\Z$-fusion rules by means of the simple objects $\{W_n\mid n\in \Z\}$.  Assume $C = \eta_{W_0,V}$ non-scalar and let $t\in \R$ with
\[
\Tr(C^*C) = q^{-1}(q^{2\lambda} + q^{-2\lambda}),\qquad \lambda = \lambda_t.
\]
Then $\mcD$ and $\Rep_q^t(\mft_{\theta};\varepsilon)$ are equivalent as ribbon twist-braided $\Rep_q(\mfsu_2)$-module categories by an equivalence sending $W_0$ to $\varepsilon = C_0$.
\end{Theorem}
Note that there are two possibilities for $t$. Indeed, it is not hard to check that the two choices are related by an isomorphism sending $\C_n$ to $\C_{-n}$.

\begin{proof}
It follows from \cite{MR3420332} that any module C$^*$-category with $\Z$-fusion rules must be isomorphic to some $\Rep_q^t(\mft_{\theta};\varepsilon)$. Since any isomorphism of the identity of $\Rep_q(\mfsu_2)$ changes the braiding to another braiding satisfying the same non-scalar condition, the allowed values of $t$ are then immediate from Proposition \ref{PropUnBraid}.
\end{proof}

The following theorem now settles one part of Conjecture \ref{ConjSat} in this case.

\begin{Theorem} Let $\sigma$ be an involution of $\mfsu_2$. Then $\Rep_{\hbar}(\mfk_{\sigma})$ and $\Rep_q^{\theta}(\mfk_{\theta})$ are equivalent as twist-braided module C$^*$-categories.
\end{Theorem}
\begin{proof}
By Remark \ref{RemEqInn} and the fact that all involutions of $\mfsu_2$ are inner, we may take $\sigma = \theta = \Ad(g)$ with
$$
g = \begin{pmatrix} 0 & 1 \\ -1 & 0\end{pmatrix}.
$$

Clearly, we can take as a maximal family of simple objects in $\Rep_{\hbar}(\mfk_{\theta})$ the family of characters
\[
\chi_{\lambda}\colon \mfh_{\theta} \rightarrow \C,\quad f-e \mapsto
i\lambda,
\]
for $\lambda\in \R$. By Lemma \ref{LemIsoComp}, it is now sufficient to show that
\[
\Rep_{\hbar}(\mft_{\theta};\chi_{\lambda}) \cong
\Rep_q^{t}(\mft_{\theta};\varepsilon)\ \ \text{for}\ \ \lambda =
\lambda_t.
\]
However, $\Rep_{\hbar}(\mft_{\theta};\chi_\lambda)$ is equivalent to a ribbon $\Id$-braided category with associated non-scalar braiding
\begin{multline*}
\widetilde{\eta}_{\chi_\lambda,V} = e^{i\pi \hbar (\chi_{\lambda}(F-E)(F- E) - \frac{1}{2})} g =e^{-\pi i \hbar /2} \begin{pmatrix} i\sinh(\pi i\hbar \lambda)  & \cosh(\pi i \hbar \lambda) \\ -\cosh(\pi i\hbar \lambda) & i\sinh(\pi i \hbar \lambda)  \end{pmatrix} \\
= q^{-1/2} \begin{pmatrix} \frac{i}{2}(q^\lambda -q^{-\lambda}) & \frac{1}{2}(q^\lambda + q^{-\lambda})\\ - \frac{1}{2}(q^\lambda + q^{-\lambda}) &\frac{i}{2}(q^\lambda -q^{-\lambda}) \end{pmatrix} .
\end{multline*}
Since $\Rep_{\hbar}(\mfk_{\theta};\chi_{\lambda})$ has $\Z$-fusion rules, we conclude that $\Rep_{\hbar}(\mft_{\theta};\chi_{\lambda}) \cong \Rep_q^{t}(\mft_{\theta};\varepsilon)$ as ribbon twist-braided $\Rep_q(\mfsu_2)$-module C$^*$-categories by Proposition \ref{PropUnBraid}.
\end{proof}

Let us now prove the second part of Conjecture \ref{ConjSat}. Note first that in this case, $U_q^{10}(\mfsu_2)$ is given by the universal relations \[KF = q^{-2} FK,\quad KF^* = q^2 F^*K,\quad F^*F - q^2FF^* = \frac{1 + K^{-2}}{q-q^{-1}}.\] The following lemma is straightforward.

\begin{Lem} The weights
\[
\omega_r(K) = q^{-r},\qquad r\in \R
\]
exhaust the family of admissible weights for $U_q^{10}((\mfsu_2)_{\mathrm{r}})$.
\end{Lem}

\begin{proof} Let $M_r = \C[\N]$ be a pre-Hilbert space (with $\{e_n\}_{n \in \N}$ being an orthonormal basis) endowed with the $*$-representation $\pi_r$ of $U_q^{10}((\mfsu_2)_{\mathrm{r}})$ characterized by
\[
\pi_r(K)e_n = q^{-r+2n}e_n,\qquad q^{1/2}(q^{-1}-q)\pi_r(F) e_n = q^{-n}\left((1-q^{2n})(1+q^{2r+2-2n})\right)^{1/2}e_{n-1}.
\]
It can be concretely verified that these exhaust the isomorphism classes of irreducible $*$-representations of  $U_q^{10}((\mfsu_2)_{\mathrm{r}})$.
\end{proof}

We will keep the notation $(M_r,\pi_r)$ from the proof of the previous lemma.

\begin{Lem}\label{LemFusVog} We have $M_r \otimes V \cong M_{r+1} \oplus M_{r-1}$.
\end{Lem}
\begin{proof}
Let $\eta_r$ be the lowest weight vector in $M_r$. Then clearly $M_r \otimes V$ is generated by the lowest weight vector $\eta_r\otimes e_-$ (of weight $\omega_{r+1}$) and $\eta_r\otimes e_+$. But $\eta_r\otimes e_+$ is a lowest weight vector in the quotient module $M_r\otimes V/ U_q^{\nu}(\mfsu_2)(\eta_r\otimes e_-)$, of weight $\omega_{r-1}$.
\end{proof}

The following theorem now settles the second part of Conjecture \ref{ConjSat}.

\begin{Theorem} The twist-braided $\Rep_q(\mfsu_2)$-module C$^*$-categories $\Rep_{q}^{\theta}(\mfk_{\theta})$ and $\Rep_q^{\nu}(\mfk_{\nu})$ are equivalent.
\end{Theorem}

\begin{proof}
By Lemma \ref{LemFusVog} and Lemma \ref{LemIsoComp}, it suffices to show that  the $\nu_q$-braided $\Rep_q(\mfsu_2)$-module C$^*$-category $\Rep_q^{\nu}(\mfk_{\nu};M_r)$ is equivalent to $\Rep_q^t(\mfk_{\theta};\varepsilon)$ for $\lambda_t = 2r+2$.

Now the ribbon element for $U_q(\mfsu_2)$ satisfies
\[
v_V = q^{3/2}.
\]
On the other hand, we can formally write $K= q^{H}$ and
\[
(\id\otimes \nu_q)(\wmsR) = (1\otimes 1+(q-q^{-1})KF^*\otimes F + \ldots)q^{-\frac{1}{2}H\otimes H},\qquad  \wmsR_{21} =  (1\otimes 1-(q-q^{-1})F\otimes E + \ldots)q^{-\frac{1}{2}H\otimes H}
\]
interpreted in $U_q^{\nu}(\mfsu_2)\otimes U_q(\mfsu_2)$ (as the extra Cartan part of $U_q^{10}(\mfu \oplus \mfa)$ is central in this case, it can be ignored in the discussion), hence
\[
(\id\otimes \nu_q)(\wmsR)(\eta_r \otimes e_-) = q^{-\frac{r}{2}} \eta_r\otimes e_-,\quad (\id\otimes \nu_q)(\wmsR)(\eta_r \otimes e_+) = q^{\frac{r}{2}}(\eta_r\otimes e_+  +(q-q^{-1}) q^{\frac{3}{2}-r} F^*\eta_r \otimes e_-).
\]
It follows that with $\msE = \msR_{21}(\id\otimes \nu_q)\msR(1\otimes v^{-1})$ the $\nu_q$-braiding, we have
\begin{align*}
\msE(\eta_r\otimes e_-) &=  q^{-r-\frac{3}{2}}\eta_r\otimes e_-,\\
\msE(\eta_r\otimes e_+) &= q^{\frac{r-3}{2}}(((q-q^{-1})q^{1-\frac{3}{2}r} + q^{2+\frac{1}{2}r}) \eta_r \otimes e_+ + (q-q^{-1})q^{\frac{5-3r}{2}}F^*\eta_r\otimes e_- ).
\end{align*}

Since in the quotient module of $M_r\otimes V$ we can identify $F^* \eta_r\otimes e_-$ and $-q^{-3/2} \eta_r \otimes e_+$, it follows that~$\msE$ has eigenvalues
\[
q^{-r-\frac{3}{2}},\quad  q^{r+\frac{1}{2}}.
\]
Now, we can write
\[
\nu_q = \Ad(K_\chi),\ \ \text{where}\ \ \chi(H) = \frac{\pi i}{2\ln(q)}.
\]
Since the element $\msE K_{\chi}$ is non-scalar, we deduce from Theorem \ref{TheoUniBraid} that $\Rep_q^{\nu}(\mfsu_2;M_r)$ is equivalent to $\Rep_q^{t}(\mft_{\theta};\varepsilon)$ for $\lambda_{t}$ such that $\lambda_{t} = 2r+2$, by an equivalence sending $M_r$ to $\C_0$.
\end{proof}

\appendix

\section{Symmetric pairs}\label{Ap1}

\subsection{Proof of Theorem \ref{TheoSat}}

Let $\theta$ be any involution of $\mfu$. Since any automorphism of $\mfu$ is the composition of an inner automorphism and the automorphism induced from a Dynkin diagram automorphism, for the existence part it is enough to show that one can find $*$-compatible Chevalley generators of $\mfg$ with respect to which $\theta$ is in Satake form.

Denote the eigenspace decomposition of $\theta$ by $\mfu = \mfk \oplus \mfm$, with $\mfk=\mfu^\theta$, and consider the real form
$$
\mfg_\theta=\mfk\oplus i\mfm
$$
of $\mfg$. It is a real semisimple Lie algebra and $\theta$ is a Cartan involution on it. Let $\mft\subset\mfu$ be a $\theta$-invariant Cartan subalgebra such that, with $\mfh = \mft^\C$, the Cartan subalgebra $\mfh\cap\mfg_\theta$ of $\mfg_\theta$ is \emph{maximally split}, that is, it contains a maximal abelian subspace of $i\mfm$. Specifically, we start with a maximal abelian subspace $\mfa_{\theta}\subseteq i\mfm$, then take a maximal abelian subalgebra $\mft_{\theta}$ of the centralizer of $\mfa_{\theta}$ in $\mfk$, and then define $\mft=\mft_{\theta}\oplus i\mfa_{\theta}$, see~\cite{MR1920389}*{Proposition~6.47}.

The real vector space spanned by the root vectors can be identified with $(i\mft_{\theta})^*\oplus\mfa_{\theta}^*$. Consider the lexicographic order on this space defined by a basis $x_1,\dots, x_n$ in $i\mft_{\theta}\oplus\mfa_{\theta}$ such that $x_1,\dots,x_p$ is a basis in $\mfa_\theta$. Consider the corresponding positive roots $\alpha\in\Delta^+$ and simple roots $\alpha_r$, $r\in I$. This already fixes the generators $h_r$. For a root $\alpha$, denote by $\overline\alpha$ the restriction of~$\alpha$ to $\mfa_\theta$. Let
$$
X=\{r\in I\mid\overline{\alpha_r}=0\}.
$$

Recall that $\Theta\colon \mfh^* \rightarrow \mfh^*$ defines the involutive transformation dual to $\theta_{\mid \mfh}$. It acts trivially on the root system~$\Delta_X$ generated by $X$. From the definition of the order structure it is also clear that $\Theta$ maps every root $\alpha\in\Delta^+\setminus \Delta_X^+$ into a negative root. It follows that $-w_X\circ\Theta$ maps $\Delta^+$ onto itself, so
$$
-w_X\circ\Theta=\tau
$$
for an automorphism~$\tau$ of the Dynkin diagram such that its action on $X$ coincides with that of $-w_X$. In other words, if we fix generators $e_r$ and $f_r$ (such that $f_r=e_r^*$) and define the element $m_X$ using them, then the automorphism $\theta\circ(\Ad m_X)^{-1}\circ\tau^{-1}\circ\omega$ of $\mfg$ is trivial on $\mfh$. Hence
$$
\theta=(\Ad z)\circ\omega\circ\tau\circ(\Ad m_X)=(\Ad z)\circ(\Ad m_X)\circ\tau\circ\omega,
$$
for some $z\in T$, where we used that $\tau(m_X)=m_X$ by $\tau$-invariance of $X$ and that the automorphisms $\Ad m_X$ and $\omega$ commute~\cite{MR3269184}*{Proposition~2.2(3)}.

Next we want to show that $z_k=1$ for $k\in X$. Let us prove first that $\theta(e_k)=e_k$ for such $k$. Assume this is not the case for some $k\in X$. Since the root space $\mfg_{\alpha_k}$ is one-dimensional and $\theta$-invariant, it follows that $\theta(e_k)=-e_k$. But then the element $e_k+f_k$ lies in $i\mfm$ and commutes with~$\mfa_{\theta}$, which contradicts the maximality of $\mfa_\theta$.

Denote by $\tau_{X}$ the permutation of $X$ defined by $-w_X$. Then on the Lie subalgebra $\mfg_X\subset\mfg$ corresponding to~$X$ we have
$$
\Ad z=\theta\circ\omega\circ\tau^{-1}\circ(\Ad m_X)^{-1}=\omega\circ\tau_X\circ(\Ad m_X)^{-1}=\id,
$$
where the last equality follows, e.g., from~\cite{MR1314093}*{Lemme~4.9}. This shows that $z_k=1$ for $k\in X$.

Finally, it remains to show that $\tau$ is involutive. Since $\theta$ is trivial on $\mfg_X$, the automorphisms $\theta$ and $\Ad m_X$ commute. Hence the transformations $\Theta$ and $w_X$ commute. But then the equality $-w_X\circ\Theta=\tau$ implies that~$\tau$ is involutive.

As already explained in~\cite{MR3269184}, uniqueness of $(X,\tau)$ up to conjugation follows from results of Kac and Wang~\cite{MR1155464}.

\subsection{Proof of Theorem \ref{TheoVog}}

Given an involution $\nu$ of $\mfu$, it is sufficient to show that there exist Chevalley generators with respect to which $\nu$ is in Vogan form.

Again as above consider the eigenspace decomposition $\mfu=\mfk\oplus\mfm$ of $\nu$ and the real form $\mfg_\nu=\mfk\oplus i\mfm$ of~$\mfg$. Choose a $\nu$-invariant Cartan subalgebra $\mft$ of $\mfu$ with complexification $\mfh$ such that $\mfh\cap\mfg_\nu$ is \emph{maximally anisotropic} in $\mfg_\nu$, that is, it contains a Cartan subalgebra of $\mfk$. Specifically, we start with a Cartan subalgebra $\mft_\nu$ of $\mfk$, then take its centralizer in $\mfg_\nu$, which necessarily has the form $\mft_\nu\oplus\mfa_\nu$ for some $\mfa_\nu\subseteq i\mfm$, and put $\mft=\mft_\nu\oplus i\mfa_\nu$, see~\cite{MR1920389}*{Proposition~6.60}.

Every root $\alpha$ with respect to $\mft$ restricts to a nonzero functional on $\mft_{\nu}$. Indeed, assume that this is not the case for some $\alpha$. Then for the transformation $\Nu$ dual to $\nu_{\mid \mfh}$ we have $\Nu(\alpha)=-\alpha$, hence $\nu(\mfg_\alpha)=\mfg_{-\alpha}$. Since~$\nu$ is unitary with respect to the inner product $B(X,Y^*)$ on $\mfg$, for any nonzero $e_\alpha\in\mfg_\alpha$ we have $\nu(e_\alpha)=c_\alpha e_\alpha^*$ for some $c_\alpha\in\T$. By rescaling $e_\alpha$ we may assume that $c_\alpha=1$. But then $i(e_\alpha+e_\alpha^*)$ lies in~$\mfk$ and commutes with~$\mft_{\nu}$, which contradicts the maximality of $\mft_\nu$.

Now, consider the lexicographic order on $(i\mft)^*=(i\mft_\nu)^*\oplus\mfa_\nu^*$ defined by a basis $x_1,\dots, x_n$ in $i\mft_\nu\oplus\mfa_\nu$ such that $x_1,\dots,x_l$ is a basis in $i\mft_\nu$. Consider the corresponding system of simple roots. This fixes the generators~$h_r$. By the previous paragraph, the set of positive roots is invariant under $\Nu$, hence $\Nu$ is defined by an involutive automorphism $\mu$ of the Dynkin diagram.

Choose Chevalley generators $e_r$ and $f_r$ such that $e_r^*=f_r$. If $\alpha_r$ is identically zero on $\mfa_\nu$, then $\Nu(\alpha_r)=\alpha_r$ and the root space $\mfg_{\alpha_r}$ is $\nu$-invariant. Hence either $\nu(e_r)=e_r$ or $\nu(e_r)=-e_r$. On the other hand, if $\alpha_r$ restricts to a nonzero functional on $\mfa_\nu$, then $\alpha_{\mu(r)}=\Nu(\alpha_r)\ne\alpha_r$ and $\nu(e_r)=c_re_{\mu(r)}$ for some $c_r\in\T$. But then by choosing a representative $r$ in every $2$-point orbit of $\mu$ and by replacing $e_{\mu(r)}$ by $c_re_{\mu(r)}$, and correspondingly~$f_{\mu(r)}$ by~$\bar c_rf_{\mu(r)}$, we can arrange that $\nu(e_r)=e_{\mu(r)}$.

\section{\texorpdfstring{$*$}{*}-invariance of quantum coideal symmetric pairs}\label{Ap2}

In this section we will prove Theorem \ref{TheoStarInv}.

We will need some more information on the Lusztig braid operators. Following \cite{MR1359532}*{Section 8}, let us also write $T_r$ for the braid operators as acting on finite admissible $U_q(\mfg)$-modules $V$, so
\[
T_r(xv) = T_r(x)T_r(v),\quad (x\in U_q(\mfg),v\in V).
\]
Let again
\[
w_X = s_{r_1}\cdots s_{r_M}
\]
be a reduced expression of the longest element of the Weyl group of $\mfg_X$. Then we have $w_X = s_{r_M}\cdots s_{r_1}$,
\[
\beta_1 = \alpha_{r_1}, \quad \beta_2 = s_{r_1}(\alpha_{r_2}),\quad \beta_3 = s_{r_1}s_{r_2}(\alpha_{r_3}),\quad \ldots,\quad \beta_M = s_{r_1} \cdots s_{r_{M-1}}(\alpha_{r_M})
\]
is an enumeration of the positive roots of $\mfg_X$~\cite{bourbaki-lie-fr-4-6}*{Section V.1.6, Corollaire 2 \`a Proposition 17}, and $T_{w_X} = T_{r_1} \cdots T_{r_M}$ is independent of the choice of reduced expression~\citelist{\cite{MR1227098}*{Chapter 39}\cite{MR1359532}*{Section 8.18--8.23}}.

For $\varpi$ a dominant integral weight of $\mfg_X$, let $V_{\varpi}$ be the highest weight $U_q(\mfg_X)$-module with weight $\varpi$ and highest weight vector $\xi_{\varpi}$. Then by induction we see that $\xi_k = F_{r_k}^{(\varpi,\beta^{\vee}_{k})} \cdots F_{r_1}^{(\varpi,\beta^{\vee}_{1})} \xi_\varpi$ has the weight $s_{r_k} \cdots s_{r_1}(\varpi)$ for $1 \le k \le M$. Moreover, for each $0 \le k < M$, ~\cite{bourbaki-lie-fr-7-8}*{Section VIII.7.2, Proposition 3} implies that $E_{r_{k+1}} \xi_k = 0$ and that $(\varpi,\beta^{\vee}_{k})$ is the largest integer $n$ such that $F_{r_{k+1}}^n \xi_k \neq 0$ (we put $\xi_0 = \xi_\varpi$). Hence from the formulas in \cite{MR1359532}*{Section 8.6} we see that
\begin{equation}\label{EqFormT-}
T_{w_X}\xi_{\varpi} = q^{2(\varpi,\rho_X)}\frac{\prod_{k=1}^M (-1)^{(\varpi,\beta_{k}^{\vee})}}{\prod_{k=1}^M [(\varpi,\beta_{k}^{\vee})]_{q_{r_k}}!} Z_{\varpi}^- \xi_{\varpi}
\end{equation}
for the lowest weight vectors, with
\[
Z_{\varpi}^- = F_{r_M}^{(\varpi,\beta^{\vee}_{M})} \cdots F_{r_1}^{(\varpi,\beta^{\vee}_{1})}.
\]
In particular, the right hand side, hence $Z_{\varpi}^-\xi_{\varpi}$ itself as well, does not depend on the particular decomposition of~$w_X$. The same computations give
\begin{equation}\label{EqFormT-Inv}
T_{w_X}^{-1}\xi_{\varpi} = \frac{1}{\prod_{k=1}^M [(\varpi,\beta_{k}^{\vee})]_{q_{r_k}}!} Z_{\varpi}^- \xi_{\varpi}.
\end{equation}

Similarly, we put
\[
Z_{w_X\omega}^+ = E_{r_M}^{(\varpi,\beta^{\vee}_{M})} \cdots E_{r_1}^{(\varpi,\beta^{\vee}_{1})},
\]
so that $Z_{w_X\varpi}^+$ transports a lowest weight vector of $V_{\varpi}$ to a multiple of $\xi_{\varpi}$ (the multiple again being independent of the chosen reduced expression). Concretely,
\begin{equation}\label{EqFormT+}
\begin{split}
T_{w_X}(T_{w_X}\xi_{\varpi}) &=  \frac{1}{\prod_{k=1}^M [(\varpi,\beta_{k}^{\vee})]_{q_{r_k}}!} Z_{w_X\varpi}^+ (T_{w_X}\xi_{\varpi}),\\
T_{w_X}^{-1}(T_{w_X}\xi_{\varpi}) &=   q^{-2(\varpi,\rho_X)}\frac{\prod_{k=1}^M (-1)^{(\varpi,\beta_{k}^{\vee})}}{\prod_{k=1}^M [(\varpi,\beta_{k}^{\vee})]_{q_{r_k}}!} Z_{w_X\varpi}^+ (T_{w_X}\xi_{\varpi}).
\end{split}
\end{equation}
When $r \in I\setminus X$, the restriction of $w_X(\alpha_r)$ to $\mfh_X$ is a dominant weight~\cite{MR3269184}*{Section 3.3}. Interpreting $-\alpha_r$ and $w_X(\alpha_r)$ as dominant weights on $\mfg_X$, we have
\begin{align*}
Z_r^{-}(X) &= Z_{-\alpha_r}^-,& Z_r^+(X) &= Z_{\alpha_r}^+
\end{align*}
in the notation of \cite{MR3269184}*{Section 4.3}.

Let us denote now by $\Ad_q$ the \emph{adjoint action} of $U_q(\mfg)$ on itself, defined by
\[
\Ad_{q}(x)(y) = x_{(1)} y S(x_{(2)}),
\]
where $S$ is the antipode of $U_q(\mfg)$ and where we have used the Sweedler notation $\Delta(x) = x_{(1)}\otimes x_{(2)}$. Following still the notation of \cite{MR3269184}*{Section 4.3}, we write $a_r^+$ for the unique number such that
\[
T_{w_X}(E_r) = a_r^+ \Ad_q(Z_r^+(X))(E_r).
\]
It is easy to see that $a_r^+ = a_{\tau(r)}^+ \in \R$. Be careful, however, that $T_{w_X}(E_r)$ is considered as the automorphism $T_{w_X}$ applied to $E_r \in U_q(\mfg)$, which \emph{is not} a priori the same as $T_{w_X}$ applied to $E_r$ considered as an element of the $U_q(\mfg_X)$-module $\Ad_{q}(U_q(\mfg_X))(E_r)$ under the adjoint action.

\begin{Def} We write $e_{\varpi},d_{w_X\varpi}\in \R$ for the unique numbers such that in $V_{\varpi}$ we have
\begin{align*}
Z_{w_X\varpi}^+Z_{\varpi}^- \xi_{\varpi} &= e_{\varpi} \xi_{\varpi}, & Z_{\varpi}^-Z_{w_X\varpi}^+T_{w_X}\xi_{\varpi} &= d_{w_X\varpi}T_{w_X}\xi_{\varpi}.
\end{align*}
\end{Def}

The following lemma follows immediately from \eqref{EqFormT-} and \eqref{EqFormT+}.

\begin{Lem} We have
\[
e_{\varpi} = d_{w_X\varpi} =  \prod_{k=1}^M ([(\varpi,\beta_k^{\vee})]_{q_{r_k}}!)^2.
\]
\end{Lem}

\begin{Lem}\label{Forma} For $r\in I\setminus X$, we have $a_r^+ = d_{\alpha_r}^{-1/2}$.
\end{Lem}
\begin{proof} Let $\varpi_r$ be the fundamental weight of $\mfg$ dual to $\alpha_r$, and let $w_0$ be the longest element of the Weyl group of $\mfg$. Then $-w_0\varpi_r$ is a dominant integral weight for $\mfg$ such that the $U_q(\mfg)$-module $V_{-w_0\varpi_r}$ has a lowest weight vector $\eta_{-\varpi_r}$ at weight $-\varpi_r$. It follows that
\[
T_{w_X}(E_r)\eta_{-\varpi_r} = T_{w_X}(E_r)T_{w_X}(\eta_{-\varpi_r}) = T_{w_X}(E_r\eta_{-\varpi_r}).
\]
On the other hand, from a concrete formula of $S$ and that $\eta_{-\varpi_r}$ is a lowest weight vector, we have
\[
T_{w_X}(E_r)\eta_{-\varpi_r} = a_r^+ \Ad_q(Z_{\alpha_r}^+)(E_r)\eta_{-\varpi_r} = a_r^+ Z_{\alpha_r}^+ E_r \eta_{-\varpi_r}.
\]
Now $E_r\eta_{-\varpi_r}$ spans an $U_q(\mfg_X)$-module isomorphic to the highest weight-module $V_{w_X\alpha_r}$ of $U_q(\mfg_X)$, where under this correspondence $E_r\eta_{-\varpi_r}$ is sent to a lowest weight vector $Z_{w_X\alpha_r}^-\xi_{w_X\alpha_r}$.  Then it follows from \eqref{EqFormT-Inv} that
\[
T_{w_X}Z_{w_X\alpha_r}^- \xi_{w_X\alpha_i} = \prod_{k=1}^M  ([(w_X\alpha_r,\beta_k^{\vee})]_{q_{r_k}}!) \xi_{w_X\alpha_r},
\]
while by definition
\[
a_r^+Z_{\alpha_r}^+ Z_{w_X\alpha_r}^- \xi_{w_X\alpha_r} = a_r^+ e_{w_X\alpha_r}\xi_{w_X\alpha_r}.
\]
This finishes the proof.
\end{proof}

We are now ready to determine the $*$-invariance of $U_q^{\mbt}(\mfk^{\C})$. We start with two easy lemmas, the first of which is implicit in \cite{MR3269184}.

\begin{Lem} For all $r\in I$, we have
\[
K_{\tau(r)}K_r^{-1} \in U_q^{\mbt}(\mfk^{\C}).
\]
\end{Lem}

\begin{proof} Since $\Theta = -w_X \circ \tau$, we have
\[
\alpha_{\tau(r)}-\alpha_r  = (\alpha_{\tau(r)} - w_X\alpha_{\tau(r)}) - (\Theta(\alpha_r) + \alpha_r)
\]
with $\alpha_{\tau(r)} - w_X\alpha_{\tau(r)}\in \Z X$, while $\Theta(\alpha_r)+\alpha_r$ is $\Theta$-invariant.
\end{proof}

\begin{Lem}\label{LemAdPres}  If $x\in U_q^{\mbt}(\mfk^{\C}), r\in I$ and $y\in U_q(\mfg_X)$, then
\[
\Ad_q(y)(xK_{\tau(r)})K_r^{-1}\in U_q^{\mbt}(\mfk^{\C}).
\]
\end{Lem}

\begin{proof} We have
\[
\Ad_q(y)(xK_{\tau(r)})K_r^{-1} = \Ad_q(y_{(1)})(x)\Ad_q(y_{(2)})(K_{\tau(r)})K_r^{-1}.
\]
However, $\Ad_q(c)(x) \in U_q^{\mbt}(\mfk^{\C})$ for all $c\in U_q(\mfg_X)$, while
\[
\Ad_q(U_q(\mfg_X))(K_{\tau(r)})K_r^{-1}\in U_q(\mfg_X)K_{\tau(r)}U_q(\mfg_X)K_r^{-1} =  U_q(\mfg_X)K_{\tau(r)}K_r^{-1}  \subseteq U_q^{\mbt}(\mfk^{\C}).
\]
This proves the assertion.
\end{proof}

We are now ready to prove Theorem \ref{TheoStarInv}. In the proof of the following, we use again the shorthand
\[
z_r = s(X,\tau)(\alpha_r)
\]
under the notation of \cite{MR3269184}.

\begin{Theorem}
\label{thm:Star-invariance}
The algebra $U_q^{\mbt}(\mfk^{\C})$ is $*$-invariant if $s_r \in i\R$ and
$$
c_{\tau(r)}c_r = q^{(\Theta(\alpha_r)-\alpha_{r},\alpha_{\tau(r)})}
$$
for all $r\in I\setminus X$.
\end{Theorem}

\begin{proof} Clearly $U_q(\mfg_X)$ and $U(\mfh^{\theta})$ are $*$-invariant. We are to show that $B_r^* \in U_q^\mbt(\mfk^{\C})$ for $r\in I \setminus X$.

Recall first that $F_rK_r$ spans an irreducible $\Ad_q(U_q(\mfg_X))$-module \cite{MR3269184}*{Lemma 3.5}. We consider now two cases.

Suppose first that the $\Ad_q(U_q(\mfg_X))$-module spanned by $F_rK_r$ is one-dimensional, which is equivalent with $a_{rs} = 0$ for all $s\in X$. Then also the $\Ad_q(U_q(\mfg_X))$-module spanned by $E_r$ is one-dimensional. In particular, it follows that $T_{w_X}(E_r)= E_r$, and hence, since $z_r = 1$ as $a_{rs}=0$ for all $s\in X$,
\[
\theta_q(F_rK_r) = -E_{\tau(r)},
\]
so
\[
B_r = F_r -c_r E_{\tau(r)}K_r^{-1} + s_r K_r^{-1}.
\]
 Then
\[
B_r^* = K_r^{-1}E_r-\overline{c_r}K_r^{-1}F_{\tau(r)}K_{\tau(r)} + \overline{s_r}K_r^{-1}.
\]
On the other hand, since $K_r^{-1}K_{\tau(r)}\in U_{q}(\mfh^{\theta})U_q(\mfg_X)$, we have
\[
K_r^{-1}K_{\tau(r)}B_{\tau(r)} = K_r^{-1}K_{\tau(r)}F_{\tau(r)} - c_{\tau(r)}K_r^{-1}K_{\tau(r)}E_{r}K_{\tau(r)}^{-1} + s_{\tau(r)}K_r^{-1} \in U_q^{\mbt}(\mfk^{\C}).
\]
Hence $B_r^* \in U_q^{\mbt}(\mfk^{\C})$ if
$$
(1-\overline{c_r}c_{\tau(r)}q_r^2q^{(\alpha_r,\alpha_{\tau(r)})})K_r^{-1}E_r + (\overline{s_r} +q_r^2\overline{c_r}s_{\tau(r)})K_r^{-1}\in U_q^{\mbt}(\mfk^{\C}).
$$
Hence it is sufficient that
\begin{align*}
\overline{c_r} c_{\tau(r)}q^{(\alpha_r,\alpha_{\tau(r)}+\alpha_r)} &= 1,& \overline{s_r} &= - q_r^2 \overline{c_r} s_{\tau(r)}.
\end{align*}
Since our assumption implies $w_X\alpha_r = \alpha_r$, it is easily seen that $c_{\tau(r)}\overline{c_r}$ must have the form as in the statement of the proposition, since in this case $a_r^+  = 1$ and $\Theta(\alpha_r) = -w_X\tau(\alpha_r) = -\alpha_{\tau(r)}$. Moreover, if $s_r \neq 0$, then we are necessarily in the situation where $\tau(r)=r$. Since we assume that $c_r>0$ in this case, we obtain then $c_r = q_r^{-2}$, and it follows that $s_r$ is purely imaginary.

Let us now consider the case where $a_{rs}\neq 0$ for some $s\in X$. In particular, this automatically forces $s_r=0$.  Let us use the notation $v_r$ as in \cite{MR3269184}*{Theorem 4.4.(3)}, where we note that
\begin{equation}\label{Eqva}
v_r = s(X,\tau)(w_X\alpha_{\tau(r)})a_r^+ = \overline{z_r}a_r^+.
\end{equation}
An easy computation, using that $\Ad_q(y)(x)^* = \Ad_q(S(y)^*)(x^*)$, shows
\[
B_r^* = K_r^{-1}E_r - \overline{c_rv_r} K_r^{-1}\Ad_q(S(Z_{\tau(r)}^+)^*)(F_{\tau(r)}K_{\tau(r)}).
\]
On the other hand, it follows from Lemma \ref{LemAdPres} that
\[
\Ad_q(S(Z_{\tau(r)}^+(X))^*)(B_{\tau(r)}K_{\tau(r)})K_r^{-1} \in U_q^{\mbt}(\mfk^{\C}).
\]
Up to a scalar, this is the same element as
\begin{multline*}
K_r^{-1}\Ad_q(S(Z_{\tau(r)}^+)^*)(K_rB_{\tau(r)}K_{\tau(r)}K_r^{-1}) \\
= q^{-(\alpha_r,\alpha_{\tau(r)})}K_r^{-1} \Ad_q(S(Z_{\tau(r)}^+)^*)(F_{\tau(r)}K_{\tau(r)})  \\ + q^{-(\Theta(\alpha_r),\alpha_{\tau(r)})}c_{\tau(r)}K_r^{-1} \Ad_q(S(Z_{\tau(r)}^+)^*)(\theta_q(F_{\tau(r)}K_{\tau(r)})).
\end{multline*}
Now note that $S(Z_{\tau(r)}^+(X))^* = (-1)^{2(\rho_X^{\vee},\alpha_r)}Z_{\tau(r)}^-(X)$. Since
\[
\Ad_q(Z_{\tau(r)}^-(X)Z_{r}^+(X))(E_r) = \Ad_q(Z_{-\alpha_{\tau(r)}}^-Z_{\alpha_r}^+)(E_r)  =  \Ad_q(Z_{w_X\alpha_r}^-Z_{\alpha_r}^+)(E_r) =  d_{\alpha_r} E_r,
\]
we see that $B_r^* \in U_q^{\mbt}(\mfk^{\C})$ if
\[
c_{\tau(r)}\overline{c_r} v_{\tau(r)}\overline{v_r} =(-1)^{2(\rho^{\vee}_X,\alpha_r)} q^{(\Theta(\alpha_r) - \alpha_r,\alpha_{\tau(r)})}d_{\alpha_r}^{-1}.
\]
However, the conditions on an admissible pair imply that
\[
z_r \overline{z_{\tau(r)}} = z_r^2 = (-1)^{2(\rho_X^{\vee},\alpha_r)}.
\]
Using \eqref{Eqva}, we can then simplify the above identity, using Lemma \ref{Forma}, to the one in the statement of the proposition.
\end{proof}

\section{Characters of Letzter--Kolb coideals}\label{SecCharCoid}

Let $\theta = \theta(X,\tau)$ for $(X,\tau)$ an admissible pair, and put $\mfk =\mfu^{\theta}$. We assume that $\mfk \subseteq \mfu$ is an irreducible symmetric pair. One then has a non-trivial character on $\mfk = \mfu^{\theta}$ if and only if $\mfk \subseteq \mfu$ is a Hermitian symmetric pair. In particular, $\mfu$ is simple and the space of characters on $\mfk$ is one-dimensional. We will show that, analogously, the parameters $\mbc,\mbs$ can be non-trivial only in the Hermitian symmetric case, and then they are determined by a single parameter.

\begin{Def}
We say that a symmetric pair $\mfk = \mfu^{\theta}\subseteq \mfu$ is;
\begin{itemize}
\item of \emph{S-type} if $I_{\mcS} \neq \emptyset$,
\item of \emph{C-type} if there exist $i\in I\setminus X$ with $\tau(i)\neq i$ but $i \notin I_{\mcC}$; thus, $\mfg_{2\overline{\alpha_i}} \neq 0$.
\end{itemize}
\end{Def}

\begin{Lem}\label{LemSymHerm}
A symmetric pair is of $S$- or C-type if and only if it is Hermitian, in which case it can not be both. Moreover, if the symmetric pair is of S-type, there is exactly one element in $I_{\mcS}$, while if it is of C-type, there is exactly one two-element $\tau$-orbit $\{i,\tau(i)\}$ in $I\setminus X$ with $i\notin I_{\mcC}$.
\end{Lem}

For $\mfk\subseteq \mfu$ of Hermitian type, we will call a simple root $\alpha_i$ with $i\in I^*$ \emph{distinguished} if $\alpha_i \in I_{\mcS}$ (for S-type) or $\tau(i)\neq i$ but $i\notin I_{\mcC}$ (for C-type).

\begin{proof}
This follows by a type-by-type verification, see also the discussion below.
\end{proof}

In the following, we will list the S-type and C-type Hermitian symmetric pairs $(\mfk \subseteq \mfu)$ with the $\tau$-orbit of their distinguished simple root as in the statement of the lemma, using the numbering of roots as in \cite{MR1064110}*{Reference Chapter, table 9} (note that $\mfsp_l$ is a compact real form of $\mfsp_{2l}(\C)$):
\begin{enumerate}
\item S-type:
\begin{itemize}
\item AIII ($\mfs(\mfu_{p}\oplus \mfu_{p}) \subseteq \mfsu_{2p}$ for $p\geq 1$): $\alpha_p$.
\item DIII ($\mfu_{2p} \subseteq \mfso_{4p}$ for $p\geq 1$): $\alpha_{2p}$.
\item BDI ($\mfso_{2}\oplus \mfso_{q} \subseteq \mfso_{2+q}$ for $q\geq 3$):  $\alpha_1$.
\item CI ($\mfu_{l} \subseteq \mfsp_{l}$ with $l\geq 2$): $\alpha_l$.
\item EVII ($\mfe_6 \oplus \mfu_{1} \subseteq \mfe_7$): $\alpha_1$.
\end{itemize}
\item C-type:
\begin{itemize}
\item AIII ($\mfs(\mfu_{p}\oplus \mfu_{q})\subseteq \mfsu_{p+q}$ with $p < q$): $\{\alpha_{p},\alpha_{q}\}$.
\item DIII ($\mfu_{2p+1}\subseteq \mfso_{4p+2}$ with $p\geq 2$): $\{\alpha_{2p},\alpha_{2p+1}\}$.
\item EIII ($\mfso_{10} \oplus \mfu_{1} \subseteq \mfe_6$): $\{\alpha_1,\alpha_5\}$.
\end{itemize}
\end{enumerate}
In the C-type case, we will take resp.~$\alpha_p$, $\alpha_{2p+1}$ and $\alpha_1$ as the distinguished vector. This way the distinguished vector is always the noncompact simple positive root in the conventional presentation of inclusion $\mfk \subseteq \mfu$, cf.~\cite{MR1920389}*{Appendix C}.\footnote{Note that the above list of cases with non-trivial parameters can also be found in \cite{MR1996417}, however the case of $BDI$ (corresponding to the two cases $BI$ and $DI$) is erronously listed as never having parameters.}

\begin{Cor}\label{CorSymHerm}
Let $U_q^{\mbt}(\mfk)\subseteq U_q(\mfu)$ be a coideal $*$-subalgebra associated to $\mfk = \mfu^{\theta}\subseteq \mfu$, with all $c_r >0$.
\begin{enumerate}
\item If $\mfk \subseteq \mfu$ is not of Hermitian type, then $s_r = 0$ and for all $r\in I\setminus X$
\[
c_r= q^{\frac{1}{2}(\Theta(\alpha_r)-\alpha_{\tau(r)},\alpha_r)}.
\]
\item If $\mfk \subseteq \mfu$ is of Hermitian type S-type, then $s_r = 0$ for $r$ not distinguished and for all $r\in I\setminus X$
\[
c_r= q^{\frac{1}{2}(\Theta(\alpha_r)-\alpha_{\tau(r)},\alpha_r)}.
\]
At the distinguished vector one has $s_r \in i\R$, and all values are allowed.
\item If $\mfk \subseteq \mfu$ is of Hermitian type C-type, then $s_r = 0$ for all $r\in I\setminus X$ and for all  $r\in I^*$
\[
c_{r}= c_{\tau(r)} =  q^{\frac{1}{2}(\Theta(\alpha_{\tau(r)})-\alpha_{r},\alpha_{\tau(r)})}
\]
except possibly at the distinguished vertex $r$ where one has
\[
c_r = q^{\frac{1}{2}(\Theta(\alpha_r)-\alpha_{\tau(r)},\alpha_r)}q^\lambda,\qquad c_{\tau(r)} =  q^{\frac{1}{2}(\Theta(\alpha_{r})-\alpha_{\tau(r)},\alpha_i)}q^{-\lambda}
\]
for some $\lambda\in \R$. Any value of $\lambda$ can occur.
\end{enumerate}
\end{Cor}

We will show now, in fact, that the parameter case can be obtained in a simple way from the non-parameter case. For this, we will study the set of $*$-characters on the coideals.

\begin{Def}
If $B$ is a right coideal $*$-subalgebra and $\chi$ a $*$-character on $B$, we call $$B' = (\chi\otimes \id)\Delta(B)$$ the $\chi$-conjugate of $B$.
\end{Def}

It is clear that $B'$ is again a coideal $*$-subalgebra. Our aim is to show now that any two coideal subalgebras $U_q^{\mbt}(\mfk)$, $U_q^{\mbt'}(\mfk)$ corresponding to the same symmetric pair are $\chi$-conjugate for some $\chi$. We start by determining the relations for the maximal commutative quotient of $U_q^{\mbt}$. In the following lemma, we write
\[
J = \{r \in I \setminus X\mid \alpha_r \perp \alpha_s\textrm{ for all }s\in X\}.
\]
We will only determine the relations in the case all $|a_{rs}|\leq 2$, as this is the only case of interest for us.

In the following we denote the generators $F_r$ of
$U_q^{\mbt}(\mfk)$ for $r\in X$ by $B_r$.

\begin{Lem}
Let $C_{\mbt}$ be the maximal commutative quotient of $U_q^{\mbt}(\mfk)$, and write $\overline{Z}$ for the images of the generators $Z$ of $U_q^{\mbt}(\mfk)$ in $C_{\mbt}$. Then $C_{\mbt}$ is a universal commutative algebra with generators $\bar Z$ and the following relations:
\begin{itemize}
\item $\overline{K}_{\omega+\omega'}=\overline{K}_\omega\overline{K}_{\omega'}$ for $\omega,\omega'\in P^\Theta$.
\item $\overline{K}_r^2 = 1$ for $r\in X$.
\item $\overline{E}_r = \overline{B}_r = 0$ for $r\in X$.
\item $\overline{B}_r = 0$ for $\Theta(\alpha_r)\neq -\alpha_r$.
\item For $a_{r,\tau(r)} = 0$,
\begin{equation}\label{eq:comm1}
\delta_{r\in J}(c_r \overline{K}_{\alpha_{\tau(r)}-\alpha_r}^2 -c_{\tau(r)}) = 0.
\end{equation}
\item For $a_{rs}= -1$,
\begin{multline} \label{eq:comm2}
(1-q_r)(1-q_r^{-1})\overline{B}_r^2\overline{B}_s =   -\delta_{r\in J}\delta_{r,\tau(r)}q_r c_r \overline{K}_{\alpha_{\tau(r)}-\alpha_r}\overline{B}_s  \\ + \delta_{r,\tau(s)} (q_r+q_r^{-1})(\delta_{s\in J}q_rc_s \overline{K}_{\alpha_{\tau(s)}-\alpha_s} +\delta_{r\in J}q_r^{-2}c_r \overline{K}_{\alpha_{\tau(r)} - \alpha_r})\overline{B}_r.
\end{multline}
\item For $a_{r,\tau(r)} = -2$, \[(q_r^{-8}c_r \overline{K}_{\alpha_{\tau(r)}-\alpha_r}-c_{\tau(r)})\overline{B}_r = 0.\]
\end{itemize}
If furthermore $D_{\mbt}$ is the maximal commutative $*$-algebra
quotient, then the extra relations in $D_{\mbt}$~are
\begin{itemize}
\item $\overline{B}_r = 0$ for $r\notin J$.
\item For $r\in J$,
\begin{equation}\label{eq:comm3}
\overline{B}_r^*  = -c_{\tau(r)}^{-1}q^{-(\alpha_r,\alpha_{\tau(r)})}\overline{K}_{\alpha_{\tau(r)}-\alpha_r}\overline{B}_{\tau(r)}.
\end{equation}
\end{itemize}
\end{Lem}

\begin{proof}
The first two relations are immediate, since we can use the ordinary quantum group relations of $U_q(\mfg_X)$. Also the third relation is immediate from \cite{MR3269184}*{Theorem 7.1}, as we have non-trivial $q$-commutation relations between $U_q(\mfh^\theta)$ and $B_r$ when $r\notin J$.

For the remainder of the relations, we claim that the elements $\mathcal{Z}_r = 0$, introduced just before \cite{MR3269184}*{Lemma 7.2}, vanish in the quotient, except when $r\in J$. Indeed, by its definition $\mathcal{Z}_r$ is a linear combination of elements of the form $E_{r_1}\ldots E_{r_p}K_{\alpha_{\tau(r)}-\alpha_r}K_{\gamma}$ for $\gamma\in \Z X$ variable and some fixed $p$, which vanishes in the quotient unless $p = 0$. But $p=0$ only when $\alpha_r$ is fixed under all $s_{r'}$ for $r'\in X$, i.e., when $r\in J$.

Similarly, we claim that the elements $\mathcal{W}_{rs}$, introduced in \cite{MR3269184}*{Lemma 7.7}, disappear in the quotient for all $r\in I\setminus X$ and $s\in X$. Indeed, we have again by definition that $\mathcal{W}_{rs}$ is a linear combination of elements of the form $E_{s_1}\ldots E_{s_p}K_{\gamma}$  for $\gamma\in \Z X$ variable and some fixed $p$. This does not vanish in the quotient only if $p = 0$, which is the case precisely when $w_X(\alpha_r) = \alpha_r + \alpha_s$. But this situation does not occur, cf. the proof of \cite{MR3269184}*{Lemma 5.11, Step 1}.

The remainder of the relations are now obvious from \cite{MR3269184}*{Theorems 7.1, 7.4 and~7.8}. Indeed, the relations for $r\in I\setminus X$ and $s\in X$ all vanish, while the relations for $r,s\in I\setminus X$ follow from the observation that
\[
\sum_{n} (-1)^n \qbin{m}{n}{q} = (1-q^{m-1})(1-q^{m-3})\ldots (1-q^{-m+1}).
\]

If we now demand that also the $*$-structure passes to the quotient, then, as follows from the proof of Theorem~\ref{thm:Star-invariance}, the element $B_r^*$ for $r\notin J$ coincides up to a scalar factor with
\[
K_r^{-1}\Ad_q(S(Z_{\tau(r)}^+)^*)(K_rB_{\tau(r)}K_{\alpha_{\tau(r)}-\alpha_r}),
\]
which shows that necessarily $\overline{B}_r = 0$. On the other hand, for $r\in J$ it follows immediately that
\[
\overline{B}_r^*  = -c_{\tau(r)}^{-1}q^{-(\alpha_r,\alpha_{\tau(r)})}\overline{K}_{\alpha_{\tau(r)}-\alpha_r}\overline{B}_{\tau(r)}.
\]
This completes the proof.
\end{proof}

We can now determine for each Hermitian symmetric space a family of $*$-characters as follows.

\begin{Cor}\label{CorDetChar}
On each $U_q^{\mbt}(\mfk)$ associated to a Hermitian symmetric space, one has a one-parameter family of $*$-characters $\chi_{t}$, $t \in \R$, determined as follows:
\begin{itemize}
\item For S-type: $\chi_t(K_{\omega})=1$ for all $\omega \in P^{\Theta}$ and $\chi_{t}(B_r)=0$ for all $r$ except for the distinguished vertex~$p$ where $\chi_t(B_p) = it$.
\item For C-type: $\chi_t(B_r)=0$ for all $r$ and $\chi_t(K_{\omega})= q^{f(\omega)}$ for all $\omega\in P^{\Theta}$, where $f$ is the unique linear functional on $\mfh^*$ which vanishes on the roots in $X$ and is such that $f(\alpha_r) =0$ for all $r\in I \setminus X$ except for the distinguished vertex $p$ where $f(\alpha_p)=t$.
\end{itemize}
\end{Cor}

\bp
Note that a vertex $r\in I\setminus X$ satisfies both $r\in J$ and $\Theta(\alpha_r)=-\alpha_r$ precisely when $r\in I_\ns$. It follows that every $*$-character $\chi$ of  $U_q^{\mbt}(\mfk)$ vanishes on all generators of  $U_q^{\mbt}(\mfk)$ except possibly for $K_\omega$, with $\omega\in P^\Theta$, and $B_r$, with $r\in I_\ns$. We will consider only characters such that $\chi(K_\omega)=q^{f(\omega)}$ for a linear functional~$f$ which is real on $P^\Theta$ and vanishes on $X$. The values of such an $f$ on $P^\Theta$ are determined by $f(\alpha_{\tau(r)}-\alpha_r)$ for $r\in I\setminus X$, $\tau(r)\ne r$. These values and the numbers $\chi(B_r)$ ($r\in I_\ns$) must satisfy the relations following from \eqref{eq:comm1}--\eqref{eq:comm3}.

In the S-type case, the distinguished vertex $p$ is the unique element of $I_\mcS$ and we have $c_r=c_{\tau(r)}$ for all $r\in I\setminus X$. It follows that if we let $f=0$ and $\chi(B_r)=0$ for all nondistinguished vertices in $I_\ns$, then the relations \eqref{eq:comm1} and \eqref{eq:comm2} will be satisfied and will impose no restrictions on $\chi(B_p)$ (since by the definition of $I_\mcS$ there is no vertex $r\in I$ such that $a_{rp}=-1$ and $\delta_{r\in J}\delta_{r,\tau(r)}=1$). On the other hand, relation \eqref{eq:comm3} gives $\chi(B_p^*)=-\chi(B_p)$, so $\chi(B_p)$ can be any purely imaginary number.

In the C-type case, if we let $\chi(B_r)=0$ for all $r\in I_\ns$, then the only restrictions on $f$ will be given by \eqref{eq:comm1}. If $r\in J$ and $a_{r,\tau(r)}=0$, then $r\in I_\mcC$, so \eqref{eq:comm1} will be satisfied if we let $f(\alpha_{\tau(r)}-\alpha_r)=0$ for all $r\in I_\mcC$. The complement in $I\setminus X$ of $I_\mcC$ and the vertices fixed by $\tau$ consists precisely of the $\tau$-orbit of the distinguished vertex $p$. Therefore to define $f$ on $P^\Theta$ it remains to define its value on $\alpha_{\tau(p)}-\alpha_p$, and we can take any real number as this value.
\ep

We can now prove the following theorem.

\begin{Theorem}\label{TheoCoidConj}
All $U_q^{\mbt}(\mfk)$ associated to the same symmetric pair are
$\chi$-conjugate for some $*$-cha\-racter~$\chi$.
\end{Theorem}
\begin{proof}
We can of course restrict to the Hermitian symmetric pairs. Assume first that we are in the S-type case, with distinguished vertex $p$. Then if $\chi_t$ is one of the $*$-characters on $U_q^{\mbt}(\mfk)$ as determined above, it follows from
\[
\Delta(B_{r}) - B_{r}\otimes K_{r}^{-1} - 1\otimes F_{r} \in U_q(\mfg_X)^+ K_{\tau(r)}K_{r}^{-1}\otimes U_q(\mfg)
\]
and the fact that $\chi_t$ acts as the counit on $ U_q(\mfg_X)^+  U_q(\mfh^\theta)$ that
\[
(\chi_t \otimes \id)\Delta(B_{r}) = B_{r} + (\chi(B_{r})-s_{r})K_{r}^{-1}.
\]
But this is precisely adding the value $(it-s_{p})K_p^{-1}$ to $B_p$ at the distinguished vector, and doing nothing for the $B_{r}$ with $r\neq p$. Also on $U_q(\mfg_X)U_q(\mfh^\theta)$ the map $(\chi_t \otimes \id)\Delta$ acts as the identity. Hence $$(\chi_t\otimes \id)(\Delta(U_q^{(\mbc,\mbs)}(\mfk))) = U_q^{(\mbc,\mbs')}(\mfk),$$ where $s_i' = it$ at the distinguished vector and zero elsewhere. It follows that we obtain all coideal $*$-subalgebras corresponding to the same symmetric pair.

Assume now that we are in C-type. Then the same reasoning (using that all $s_r= 0$ in this case and that $\chi_t$ acts as $q^{f(\alpha_{\tau(r)}-\alpha_r)}$ times the counit on $U_q(\mfg_X)^+K_{\tau(r)}K_r^{-1}$) gives that
\[
(\chi_t\otimes \id)\Delta(B_r) = q^{f(\alpha_{\tau(r)}-\alpha_r)}B_r +(1-  q^{f(\alpha_{\tau(r)}-\alpha_r)} )F_r,\quad (\chi_t\otimes \id)(\Delta(K_{\alpha})) = q^{f(\alpha)}K_{\alpha}.
\]
Since all $c_r$ are necessarily positive, this implies that we get $U_q^{(\mbc',\mathbf{0})}(\mfk)$ where $c_p' = q^{-t}c_p$, $c_{\tau(p)}' = q^{t}c_{\tau(p)}$, and all other $c_r = c_r' = 1$. Again, it follows that we obtain all coideal $*$-subalgebras corresponding to the same symmetric pair.
\end{proof}

Define $c_r = q^{\frac{1}{2}(\Theta(\alpha_r)-\alpha_r,\alpha_{\tau(r)})}$ and $s_r = 0$, so that $U_q^0(\mfk) = U_q^{(\mbc,\mbs)}(\mfk)$. Let us write $U_q^t(\mfk) = U_q^{(\mbc(t),\mbs(t))}(\mfk)$, where
\begin{itemize}
\item $\mbc(t) = \mbc$, $\mbs(t)_r = \delta_{r,p} it$ for $p$ distinguished,
\item  $\mbc(t)_r = c_r$ for $r$ or $\tau(r)$ not distinguished, $\mbc(t)_r = q^{-t}c_r, \mbc(t)_{\tau(r)} = q^{t}c_r$ for $r$ distinguished, and $\mbs(t)_r= 0$ for all $r$.
\end{itemize}

\begin{Cor} The maps
\[
\pi_t\colon U_q^0(\mfk) \rightarrow U_q^t(\mfk),\quad X \mapsto (\chi_t\otimes \id)\Delta(X)
\]
are $*$-isomorphisms satisfying
\[
\Delta \circ \pi_t = (\pi_t\otimes \id)\circ \Delta.
\]
\end{Cor}

\section{Cylinder twist}

The notion of \emph{cylinder twist} and \emph{cylinder braiding} were introduced in \cite{MR1644317}. In this section, we clarify the connection between cylinder twists and ribbon braided module categories, complementing the algebraic discussion of the connection in \cite{arXiv:1705.04238}*{Section 2.3}.

Let $(\mcC, \beta)$ be a braided tensor category, $\sigma$ a braided autoequivalence, and $\mcD$ a right $\mcC$-module category.  For the sake of simplicity we assume that $\mcC$ is strict and $\mcD$ is a strict module category, although this is not an essential restriction.

\begin{Def}
We call \emph{cylinder $\sigma$-twist} $\theta$ at $X\in \mcD$ any natural isomorphism $\theta: X\odot \sigma(-) \rightarrow X\odot -$ of functors $\mcC\rightarrow \mcD$ such that $\theta$ satisfies the \emph{cylinder $\sigma$-twist equations}:
\begin{align}
\label{eq:cyl-tw-eq}
\theta_{U \otimes V} (X \odot (\sigma_2)_{U,V}) &= (X \odot \beta_{V,U}) (\theta_{V} \otimes U) (X \odot \beta_{U,\sigma(V)}) (\theta_{U} \otimes \sigma(V))  \\
\label{eq:cyl-tw-eq-2}
&= (\theta_{U} \otimes V) (X \odot \beta_{V,\sigma(U)}) (\theta_{V} \otimes \sigma(U)) (X \odot \beta_{\sigma(U),\sigma(V)}).
\end{align}
\end{Def}

For example, if $\mcD$ is a ribbon $\sigma$-braided $\mcC$-module category with $\sigma$-braid $\eta$, then $\eta_{X,-}$ is a cylinder $\sigma$-twist at $X$ for each $X\in \mcD$, using the ribbon $\sigma$-twist equation and naturality on $X \odot U$.

These equalities can be diagrammatically represented by Figure
\ref{fig:cyl-twist-eq}.
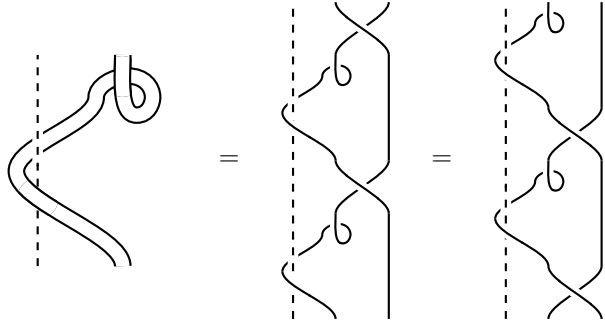
\begin{figure}[h]
\leavevmode
\begin{tikzpicture}[]
\path (0,0) (0,7); 
\begin{scope}[shift={(0,1)},scale=2]
\begin{knot}[
] \sribtwbr[only when rendering/.style={
    double distance=5pt,
  }]{0}
\strand[only when rendering/.style={dashed}] (0.2,0) -- (0.2,2);
\flipcrossings{2}
\end{knot}
\end{scope}
\draw node at (4,3) {$=$};
\begin{scope}[shift={(5,0)}]
\begin{knot}[
] \sribtwbr{0} \braidgen{1}{2} \braidgen{1}{5} \sribtwbr{3}
\strand[] (2,0) -- (2,2); \strand[] (2,3) -- (2,5); \strand[only
when rendering/.style={dashed}] (0.2,0) -- (0.2,6);
\flipcrossings{2,4,5,7}
\end{knot}
\end{scope}
\draw node at (8,3) {$=$};
\begin{scope}[shift={(9,0)}]
\begin{knot}[
] \sribtwbr{1} \braidgen{1}{0} \braidgen{1}{3} \sribtwbr{4}
\strand[] (2,1) -- (2,3); \strand[] (2,4) -- (2,6); \strand[only
when rendering/.style={dashed}] (0.2,0) -- (0.2,6);
\flipcrossings{2,4,5,7}
\end{knot}
\end{scope}
\end{tikzpicture}
\caption{cylinder $\sigma$-twist equation} \label{fig:cyl-twist-eq}
\end{figure}

Let us show how one can construct a ribbon $\sigma$-braided module category from  a cylinder $\sigma$-twist $\theta$ at $X$ on a right $\mcC$-module category  $\mcD$. Put
$$
\mcC_X(U, V) = \mcD(X \odot U, X \odot V)
$$
for objects $U, V$ of $\mcC$. This defines a right $\mcC$-module category $\mcC_X$ with the same objects as $\mcC$, and contains $\mcC$ as a subcategory. Moreover, $U \mapsto X \odot U$ together with the natural inclusion of morphisms define a functor of module categories from $\mcC_X$ to $\mcD_X$. The only deficit of $\mcC_X$ is that it is not subobject complete, but this can always be amended by passing to the subobject completion.

Put
$$
\eta_{U, V} = (X \odot \beta_{V, U}) (\theta_{V} \odot U) (X \odot
\beta_{U, \sigma(V)}) \colon X \odot U \otimes \sigma(V) \to X \odot
U \otimes V,
$$
so $\eta_{U,V} \in \mcC_X(U\otimes \sigma(V),U\otimes V)$. Diagrammatically, we represent $\theta_{V}$ as Figure
\ref{fig:cyl-tw}, and $\eta_{U,V}$ as Figure
\ref{fig:rib-tw-br-from-cyl-tw}.

\begin{figure}[h]
\begin{subfigure}[b]{0.4\textwidth}
\centering
\leavevmode
\begin{tikzpicture}
\begin{knot}[
] \varsribtwbr{0} \strand[only when rendering/.style={dashed}]
(0.2,0) -- (0.2,2); \flipcrossings{1,3}
\end{knot}
\end{tikzpicture}
\caption{cylinder $\sigma$-twist} \label{fig:cyl-tw}
\end{subfigure}
\begin{subfigure}[b]{0.4\textwidth}
\centering
\leavevmode
\begin{tikzpicture}
\begin{scope}[shift={(0,1)}]
\begin{knot}[
] \sribtwbr{0} \strand[] (0.4,0) -- (0.4,2); \strand[only when
rendering/.style={dashed}] (0.2,0) -- (0.2,2); \flipcrossings{2,3}
\end{knot}
\end{scope}
\node at (2,2) {$=$};
\begin{scope}[shift={(3,0)}]
\begin{knot}[
] \braidgen{1}{0} \varsribtwbr{1} \braidgen{1}{3} \strand[only when
rendering/.style={dashed}] (0.2,0) -- (0.2,4); \strand[] (2,1) --
(2,3); \flipcrossings{1,2,4,5}
\end{knot}
\end{scope}
\end{tikzpicture}
\caption{ribbon $\sigma$-braid from cylinder $\sigma$-twist}
\label{fig:rib-tw-br-from-cyl-tw}
\end{subfigure}
\caption{}
\end{figure}

\begin{Prop}
The family $\eta_{U, V}$ above satisfies the $\sigma$-octagon equation
and the ribbon $\sigma$-twist equation.
\end{Prop}

\begin{proof}
The $\sigma$-octagon relation can be verified as in Figure \ref{fig:oct-rel-for-zeta}, where we used the multiplicativity of $\beta$ such as $\beta_{U \otimes V, W} = (\beta_{U,W} \otimes V) (U \otimes \beta_{V,W})$. For the ribbon $\sigma$-twist relation we manipulate as in Figures \ref{fig:ribbon-tw-eq-for-zeta-1} and \ref{fig:ribbon-tw-eq-for-zeta-2}. The dotted boxes indicate use of the braid relation, naturality at $U$ in  $Y \odot U$, and the $\sigma$-octagon relation (at the final step).
\end{proof}

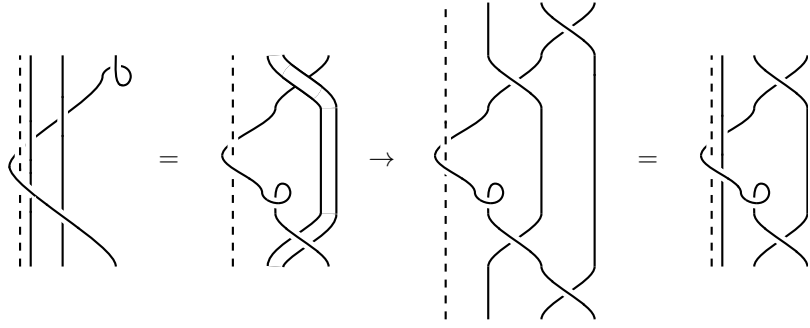
\begin{figure}
\leavevmode\beginpgfgraphicnamed{bkz_conj_fig11}
\begin{tikzpicture}
\begin{scope}[shift={(0,1)}]
\begin{knot}[
] \cribtwbr{2}{0} \strand (0.4,0) -- (0.4,4); \strand[only when
rendering/.style={dashed}] (0.2,0) -- (0.2,4); \strand[] (1,0) --
(1,4); \flipcrossings{2,3,4,5}
\end{knot}
\end{scope}
\draw node at (3,3) {$=$};
\begin{scope}[shift={(4,1)}]
\begin{knot}[
] \braidtwocldbl[][5pt]{1}{0} \varsribtwbr{1}
\braidtencldbl[][5pt]{1}{3} \strand[only when rendering/.style={
    double distance=5pt,
  }] (2,1) -- (2,3);
\strand[only when rendering/.style={dashed}] (0.2,0) -- (0.2,4);
\flipcrossings{1,2,4,5}
\end{knot}
\end{scope}
\draw node at (7,3) {$\to$};
\begin{scope}[shift={(8,0)}]
\begin{knot}[
] \braidgen{2}{0} \strand (1, 0) -- (1, 1); \braidgen{1}{1} \strand
(3, 1) -- (3, 5); \varsribtwbr{2} \strand (2, 2) -- (2, 4);
\braidgen{1}{4} \braidgen{2}{5} \strand (1, 5) -- (1, 6);
\strand[only when rendering/.style={dashed}] (0.2,0) -- (0.2,6);
\flipcrossings{1,2,4,6,7,8}
\end{knot}
\end{scope}
\draw node at (12,3) {$=$};
\begin{scope}[shift={(13,1)}]
\begin{knot}[
] \braidgen{1}{0} \varsribtwbr{1} \braidgen{1}{3} \strand[] (2,1) --
(2,3); \strand (0.4,0) -- (0.4,4); \strand[only when
rendering/.style={dashed}] (0.2,0) -- (0.2,4);
\flipcrossings{1,2,5,6,7}
\end{knot}
\end{scope}
\end{tikzpicture}
\endpgfgraphicnamed
\caption{$\sigma$-octagon relation} \label{fig:oct-rel-for-zeta}
\end{figure}

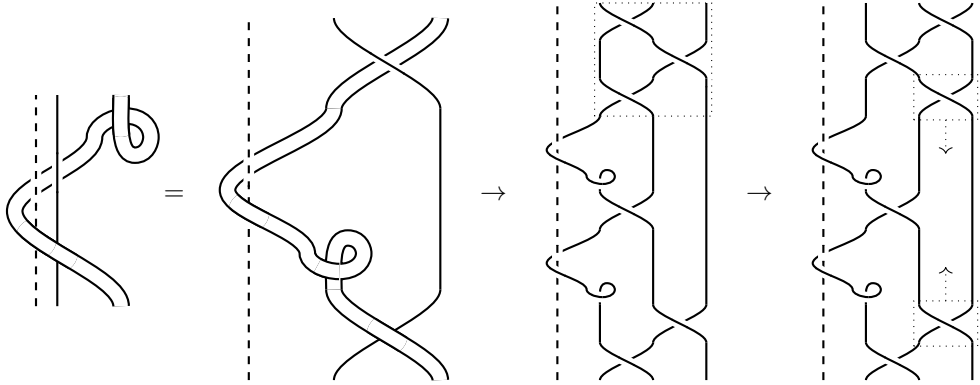
\begin{figure}
\leavevmode\beginpgfgraphicnamed{bkz_conj_fig12}
\begin{tikzpicture}
\begin{scope}[scale=2,shift={(0,0.7)}]
\begin{knot}[
] \sribtwbr[only when rendering/.style={
    double distance=5pt,
  }]{0}
\strand (0.4,0) -- (0.4,2); \strand[only when
rendering/.style={dashed}] (0.2,0) -- (0.2,2); \flipcrossings{2,3}
\end{knot}
\end{scope}
\node at (3,3.5) {$=$};
\begin{scope}[scale=2,shift={(2,0)}]
\begin{knot}[
  y=0.6cm,
] \braidtencldbl[][5pt]{1}{0} \varsribtwbr[only when
rendering/.style={
    double distance=5pt,
  }]{1}
\strand (2, 1) -- (2, 3); \braidtwocldbl[][5pt]{1}{3} \strand[only
when rendering/.style={dashed}] (0.2,0) -- (0.2,4);
\flipcrossings{1,2,4,5}
\end{knot}
\end{scope}
\draw node at (9,3.5) {$\to$};
\begin{scope}[shift={(10,0)}]
\begin{knot}[
  y=0.5cm,
  clip radius=3pt,
] \braidgen{1}{0} \strand (3,0) -- (3,1); \varsribtwbr{2}
\braidgen{2}{1} \strand (1,1) -- (1,2); \strand (2,2) --
(2,4);\strand (3,2) -- (3,8); \braidgen{1}{4} \varsribtwbr{5}
\strand[] (2,5) -- (2,7); \draw[dotted] (0.9,7) rectangle (3.1,10);
\braidgen{1}{7} \braidgen{2}{8} \strand (1,8) -- (1,9);
\braidgen{1}{9} \strand (3,9) -- (3,10); \strand[only when
rendering/.style={dashed}] (0.2,0) -- (0.2,10);
\flipcrossings{1,2,4,5,7,8,10,11,12,13}
\end{knot}
\end{scope}
\draw node at (14,3.5) {$\to$};
\begin{scope}[shift={(15,0)}]
\begin{knot}[
  y=0.5cm,
  clip radius=3pt,
] \braidgen{1}{0} \strand (3,0) -- (3,1); \varsribtwbr{2}
\braidgen{2}{1} \strand (1,1) -- (1,2); \strand (2,2) --
(2,4);\strand (3,2) -- (3,7); \braidgen{1}{4} \varsribtwbr{5}
\strand[] (2,5) -- (2,7); \braidgen{2}{7} \strand (1,7) -- (1,8);
\braidgen{1}{8} \strand (3,8) -- (3,9); \braidgen{2}{9} \strand
(1,9) -- (1,10); \strand[only when rendering/.style={dashed}]
(0.2,0) -- (0.2,10); \draw[dotted] (1.9,0.9) rectangle (3.1, 2.1);
\draw[dotted,->] (2.5,2.1) -- (2.5, 3); \draw[dotted] (1.9,6.9)
rectangle (3.1, 8.1); \draw[dotted,->] (2.5,6.9) -- (2.5, 6);
\flipcrossings{1,2,4,5,6,7,9,10,11,12}
\end{knot}
\end{scope}
\end{tikzpicture}
\endpgfgraphicnamed
\caption{ribbon $\sigma$-twist equation} \label{fig:ribbon-tw-eq-for-zeta-1}
\end{figure}

\begin{figure}
{\centering
\leavevmode\beginpgfgraphicnamed{bkz_conj_fig13}
\begin{tikzpicture}
\path (0,0) (0,10*5/7+1); 
\begin{scope}
\begin{knot}[
  y=0.5cm,
  clip radius=3pt,
] \braidgen{1}{0} \strand (3,0) -- (3,3); \varsribtwbr{1} \strand
(2,1) -- (2,3); \braidgen{2}{3} \strand (1, 3) -- (1, 4);
\braidgen{1}{4} \strand (3,4) -- (3,5); \braidgen{2}{5} \strand
(3,6) -- (3,9); \strand (1,5) -- (1,6); \draw[dotted] (0.9,3)
rectangle (3.1,6); \varsribtwbr{6} \strand (2,6) -- (2,8);
\braidgen{1}{8} \braidgen{2}{9} \strand (1,9) -- (1,10);
\strand[only when rendering/.style={dashed}] (0.2,0) -- (0.2,10);
\flipcrossings{1,2,3,5,6,7,8,9,10,12,13,14}
\end{knot}
\end{scope}
\node at (4, 3.5) {$\to$};
\begin{scope}[shift={(5,0)}]
\begin{knot}[
  y=0.5cm,
  clip radius=3pt,
] \braidgen{1}{0} \strand (3,0) -- (3,4); \varsribtwbr{1} \strand
(2,1) -- (2,3); \braidgen{1}{3} \braidgen{2}{4} \strand (1,4) --
(1,5); \braidgen{1}{5} \strand (3,5) -- (3,9); \varsribtwbr{6}
\strand (2,6) -- (2,8); \braidgen{1}{8} \braidgen{2}{9} \strand
(1,9) -- (1,10); \strand[only when rendering/.style={dashed}]
(0.2,0) -- (0.2,10);
\draw[dotted] (-0.1,4) rectangle (3.1,10.1);
\flipcrossings{1,2,4,5,6,7,8,10,11,12}
\end{knot}
\end{scope}
\node at (9, 3.5) {$\to$};
\begin{scope}[shift={(10,0.5)}]
\begin{knot}[
] \sribtwbr{0} \strand[] (2,0) -- (2,2); \cribtwbr{2}{2} \strand[]
(1,2) -- (1,6); \strand (0.4,0) -- (0.4,6); \strand[only when
rendering/.style={dashed}] (0.2,0) -- (0.2,6);
\flipcrossings{2,3,4,7,8,9}
\end{knot}
\end{scope}
\end{tikzpicture}
\endpgfgraphicnamed
}
\caption{ribbon $\sigma$-twist equation, cont.}
\label{fig:ribbon-tw-eq-for-zeta-2}
\end{figure}
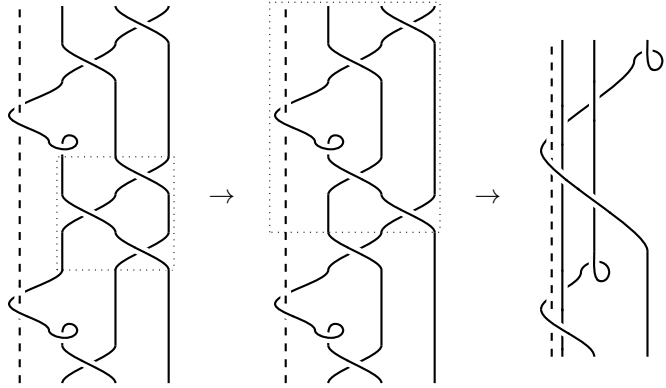

The $\mcC$-module category $\mcC_X$ will however not necessarily be a ribbon $\sigma$-braided module category, as we do not yet know if $\eta_{U,V}$ is natural in $U$. Therefore, we first restrict the morphisms spaces of $\mcC_X$ as follows:
$$
\tilde{\mcC}_X(U, V) = \{ f \in \mcD(X \odot U, X \odot V) \mid
\forall W\colon (f \odot W) \eta_{U,W} = \eta_{V,W} (f \odot
\sigma(W)) \}
$$
This still defines a right $\mcC$-module
category $\tilde{\mcC}_X$ with $\mcC$ as a
subcategory.

\begin{Prop}
We have $\eta_{U,V} \in \tilde{\mcC}_X(U \otimes \sigma(V), U
\otimes V)$ for any $U$ and $V$.
\end{Prop}

\begin{proof}
We need to prove $(\eta_{U,V} \odot W) \eta_{U,W} = \eta_{V,W} (\eta_{U,V} \odot \sigma(W))$ for all $W$. The proof goes as in Figures \ref{fig:naturality-1} and \ref{fig:naturality-2}. This time, the dotted boxes represent use of the braid relation, naturality at $U$ in $Y \odot U$, and the cylinder $\sigma$-twist equation for $\theta$ (at the penultimate step). The last step uses the ribbon $\sigma$-twist relation.
\end{proof}

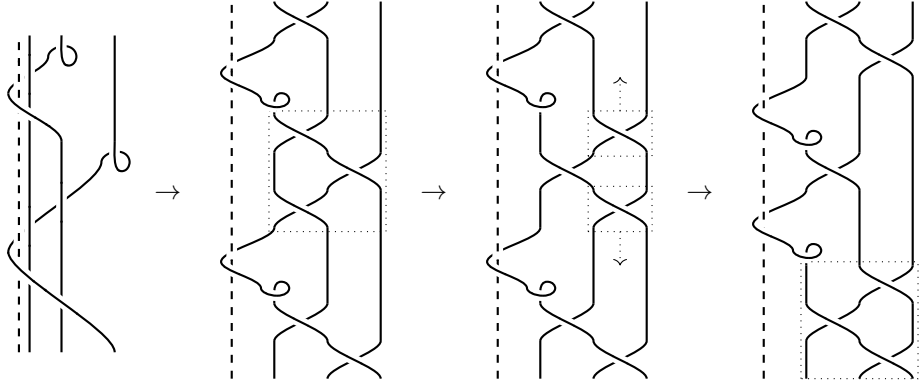
\begin{figure}
\leavevmode\beginpgfgraphicnamed{bkz_conj_fig14}
\begin{tikzpicture}
\begin{scope}[shift={(0,0.5)}]
\begin{knot}[
] \cribtwbr{2}{0} \strand[] (1,0) -- (1,4); \sribtwbr{4} \strand[]
(2,4) -- (2,6); \strand (0.4,0) -- (0.4,6); \strand[only when
rendering/.style={dashed}] (0.2,0) -- (0.2,6);
\flipcrossings{2,3,4,8,10,11}
\end{knot}
\end{scope}
\node at (3,3.5) {$\to$};
\begin{scope}[shift={(4,0)}]
\begin{knot}[
  y=0.5cm,
  clip radius=3pt,
] \braidgen{2}{0}  \strand (1,0) -- (1,1); \braidgen{1}{1} \strand
(3,1) -- (3,5); \varsribtwbr{2} \strand (2,2) -- (2,4);
\braidgen{1}{4} \strand (1, 5) -- (1, 6); \braidgen{2}{5}
\braidgen{1}{6} \strand (3,6) -- (3,10); \varsribtwbr{7} \strand
(2,7) -- (2,9); \braidgen{1}{9} \strand[only when
rendering/.style={dashed}] (0.2,0) -- (0.2,10); \draw[dotted] (0.9,
3.9) rectangle (3.1, 7.1);
\flipcrossings{1,2,3,5,6,7,8,9,11,12,13,14}
\end{knot}
\end{scope}
\node at (8, 3.5) {$\to$};
\begin{scope}[shift={(9,0)}]
\begin{knot}[
  y=0.5cm,
  clip radius=3pt,
] \braidgen{2}{0}  \strand (1,0) -- (1,1); \braidgen{1}{1} \strand
(3,1) -- (3,4); \varsribtwbr{2} \strand (2,2) -- (2,4);
\braidgen{2}{4} \strand (1, 4) -- (1, 5); \braidgen{1}{5} \strand
(3,5) -- (3,6); \braidgen{2}{6} \strand (3,7) -- (3,10); \strand
(1,6) -- (1,7); \varsribtwbr{7} \strand (2,7) -- (2,9);
\braidgen{1}{9} \strand[only when rendering/.style={dashed}] (0.2,0)
-- (0.2,10); \draw[dotted] (1.9, 3.9) rectangle (3.1, 5.1);
\draw[dotted,->] (2.5,3.9) -- (2.5,3); \draw[dotted] (1.9, 5.9)
rectangle (3.1, 7.1); \draw[dotted,->] (2.5,7.1) -- (2.5,8);
\flipcrossings{1,2,3,4,6,7,8,9,10,12,13,14}
\end{knot}
\end{scope}
\node at (13, 3.5) {$\to$};
\begin{scope}[shift={(14,0)}]
\begin{knot}[
  y=0.5cm,
  clip radius=3pt,
] \braidgen{2}{0} \strand (1,0) -- (1,1); \braidgen{1}{1} \strand
(3,1) -- (3,2); \braidgen{2}{2} \strand (1,2) -- (1, 3); \strand
(3,3) -- (3,8); \varsribtwbr{3} \braidgen{1}{5} \varsribtwbr{6}
\strand[] (2,3) -- (2,5); \strand[] (2,6) -- (2,8); \braidgen{2}{8}
\strand (1,8) -- (1,9); \braidgen{1}{9} \strand (3,9) -- (3,10);
\strand[only when rendering/.style={dashed}] (0.2,0) -- (0.2,10);
\draw[dotted] (0.9,0) rectangle (3.1, 3.1);
\flipcrossings{1,2,3,4,6,7,8,10,11,12,13}
\end{knot}
\end{scope}
\end{tikzpicture}
\endpgfgraphicnamed
\caption{naturality} \label{fig:naturality-1}
\end{figure}

\begin{figure}
\leavevmode\beginpgfgraphicnamed{bkz_conj_fig15}
\begin{tikzpicture}
\begin{scope}[shift={(0,0)}]
\begin{knot}[
  y=0.5cm,
  clip radius=3pt,
] \braidgen{1}{0} \strand (3,0) -- (3,1); \braidgen{2}{1} \strand
(1,1) -- (1,2); \braidgen{1}{2} \strand (3,2) -- (3,8);
\varsribtwbr{3} \braidgen{1}{5} \varsribtwbr{6} \strand[] (2,3) --
(2,5); \strand[] (2,6) -- (2,8); \braidgen{2}{8} \strand (1,8) --
(1,9); \braidgen{1}{9} \strand (3,9) -- (3,10); \strand[only when
rendering/.style={dashed}] (0.2,0) -- (0.2,10); \draw[dotted]
(-0.1,1.9) rectangle (2.1, 8.1);
\flipcrossings{1,2,3,4,5,7,8,9,11,12,13}
\end{knot}
\end{scope}
\node at (4,3.5) {$\to$};
\begin{scope}[shift={(5,0)},scale=2]
\begin{knot}[
  y=0.6cm,
] \braidtencldbl[][5pt]{1}{0} \varsribtwbr[only when
rendering/.style={
    double distance=5pt,
  }]{1}
\strand (2, 1) -- (2, 3); \braidtwocldbl[][5pt]{1}{3} \strand[only
when rendering/.style={dashed}] (0.2,0) -- (0.2,4);
\flipcrossings{1,2,4,5}
\end{knot}
\end{scope}
\node at (10,3.5) {$\to$};
\begin{scope}[shift={(11,0)}]
\begin{knot}[
] \sribtwbr{0} \strand[] (2,0) -- (2,2); \cribtwbr{2}{2} \strand[]
(1,2) -- (1,6); \strand (0.4,0) -- (0.4,6); \strand[only when
rendering/.style={dashed}] (0.2,0) -- (0.2,6);
\flipcrossings{2,3,4,7,8,9}
\end{knot}
\end{scope}
\end{tikzpicture}
\endpgfgraphicnamed
\caption{naturality, cont.} \label{fig:naturality-2}
\end{figure}
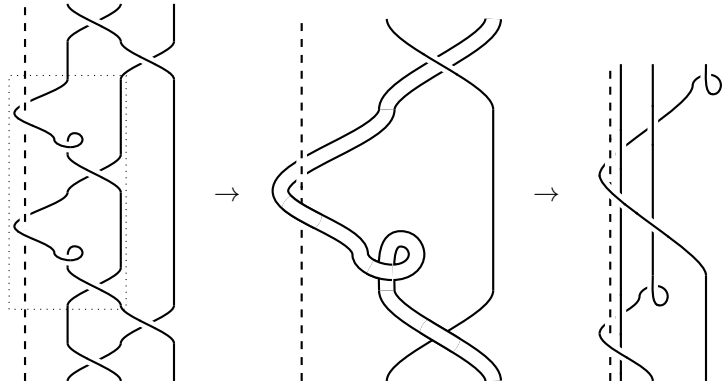

From the above, we obtain that $\eta_{U,V} \in \tilde{\mcC}_X(U\otimes V,U\otimes V)$, and the definition of $\tilde{\mcC}_X$ forces $\eta_{U,V}$ to be natural in $U$. It follows that $\tilde{\mcC}_X$ is a ribbon $\sigma$-braided $\mcC$-module category.

\end{document}